\documentclass[11pt,reqno,titlepage]{amsart}
\usepackage{tikz}
\usepackage{subfig}
\usepackage{epstopdf}
\usepackage{epsfig}
\usepackage{lipsum}
\usepackage{hyperref}
\usepackage{amsthm}
\usepackage{url}
\usepackage{amsmath}
\usepackage{amsfonts}
\usepackage{lastpage}
\usepackage{graphicx}
\usepackage{multirow}
\usepackage{tabu}
\usepackage{ams-essay-def}
\usepackage{clrscode}
\usepackage{caption}
\usepackage{appendix}
\usepackage{mathrsfs}
\usepackage{algorithm}
\usepackage{algorithmic}
\usepackage{natbib}

\usepackage[foot]{amsaddr}

\usepackage{tikz}
\usetikzlibrary{shapes,arrows}

\tikzstyle{decision} = [diamond, draw, fill=blue!20, 
    text width=4.5em, text badly centered, node distance=3cm, inner sep=0pt]
\tikzstyle{block} = [rectangle, draw, fill=blue!20, 
    text width=5em, text centered, rounded corners, minimum height=4em]
\tikzstyle{line} = [draw, -latex']
\tikzstyle{cloud} = [draw, ellipse,fill=red!20, node distance=3cm,
    minimum height=2em]

\usetikzlibrary{positioning}
\tikzset{main node/.style={circle,fill=blue!20,draw,minimum size=1cm,inner sep=0pt},  }

\makeatletter
\g@addto@macro{\endabstract}{\@setabstract}
\newcommand{\authorfootnotes}{\renewcommand\thefootnote{\@fnsymbol\c@footnote}}%
\makeatother

\newcommand{\titlemark}{Accelerated Information Gradient flow}
\title{\textbf{\titlemark}}

\begin{document}
\begin{center}
  \LARGE 
  \titlemark
  \par \bigskip

  \normalsize
  \authorfootnotes
  Yifei Wang\footnote{\texttt{wangyf18@stanford.edu}}\textsuperscript{1} and Wuchen Li\footnote{\texttt{wuchen@mailbox.sc.edu}}\textsuperscript{2},
 \par \bigskip
 \textsuperscript{1}Department of Electrical Engineering, Stanford University\par
  \textsuperscript{2}Department of Mathematics, University of South Carolina \\
\par \bigskip
\end{center}

\begin{abstract}
We present a framework for Nesterov's accelerated gradient flows in probability space to design efficient mean-field Markov chain Monte Carlo (MCMC) algorithms for Bayesian inverse problems. Here four examples of information metrics are considered, including Fisher-Rao metric, Wasserstein-2 metric, Kalman-Wasserstein metric and Stein metric. For both Fisher-Rao and Wasserstein-2 metrics, we prove convergence properties of accelerated gradient flows. In implementations, we propose a sampling-efficient discrete-time algorithm for Wasserstein-2, Kalman-Wasserstein and Stein accelerated gradient flows with a restart technique. We also formulate a kernel bandwidth selection method, which learns the gradient of logarithm of density from Brownian-motion samples. Numerical experiments, including Bayesian logistic regression and Bayesian neural network, show the strength of the proposed methods compared with state-of-the-art algorithms.

\end{abstract}

\section{Introduction}
Optimization problems in probability space, arising from Bayesian inference \cite{SVGD} and inverse problems \cite{ivabp}, attract increasing attentions in machine learning communities \cite{afomo,lmcaj,plarc}. One typical example here is to draw samples from an intractable target distribution. Such sampling problem is important in providing exploration in distribution of interest and quantifying uncertainty among data. From an optimization viewpoint, this problem suffices to minimize an objective functional, such as Kullback-Leibler (KL) divergence, which is to measure the closeness between current density and the target distribution. 

Gradient descent methods play essential roles in solving these optimization problems. Here the gradient direction relies on the information metric in probability space. In literature, two important metrics, such as Fisher-Rao metric and Wasserstein-2 (in short, Wasserstein) metric, are of great interests \cite{tdmac,ngwei,tgode}. 
The information gradient direction in terms of density corresponds to the update rule in a set of samples. This is known as sampling formulation or particle implementation of gradient flow, which yields various sampling algorithms. 
For Fisher-Rao metric, its gradient flow relates to birth-death dynamics, which is important in model selection and modeling population games \cite{igaia}. The Fisher-Rao gradient, also known as natural gradient, is also useful in designing fast and reliable algorithms in probability models \cite{ngwei,adam,ngfmm, onnwk}. For Wasserstein metric, the gradient flow of KL divergence is the Fokker-Planck equation of overdamped Langevin dynamic. In sampling algorithms, the time discretization of overdamped Langevin dynamics yields the classical Langevin Markov chain Monte Carlo (MCMC) method and the proximal Langevin algorithm \cite{lmcaj, plarc}. In recent years, various first-order sampling methods via generalized Wasserstein gradient direction are proposed. For example, the Stein variational gradient descent \cite{SVGD} formulates kernelized interacting Langevin dynamics. The Kalman-Wasserstein gradient, also known as ensemble Kalman sampling \cite{ild}, induces covariance-preconditioned mean-field interacting Langevin dyanmics.

For classical optimization problems in Euclidean space, the Nesterov's accelerated gradient method \cite{nesterov} is a wide-applied optimization method and it accelerates gradient descent methods. The continuous-time limit of this method is known as the accelerated gradient flow \cite{adefm}. Natural questions arise: \textit{What is the accelerated gradient flow in probability space under general information metrics? What is the corresponding discrete-time sampling algorithm?} For optimization problems on a Riemannian manifold,  accelerated gradient methods are studied in \cite{afomf,tragm}. The probability space embedded with information metric can be viewed as a Riemannian manifold, known as density manifold \cite{tdmac}. 
Several previous works explore accelerated methods in this manifold under Wasserstein metric. 
An acceleration framework of particle-based variational inference (ParVI) methods is proposed in \cite{afomo, uaapb} based on manifold optimization.  Taghvaei and Mehta \cite{affpd} introduce accelerated flows from an optimal control perspective. Similar dynamics has been studied from a fluid dynamics viewpoint \cite{coeiw}. Underdamped Langevin dynamics is another way to accelerate on MCMC \cite{ulman,itaao}. 

In this paper, we present a unified framework of accelerated gradient flows in probability space embedded with information metrics, named Accelerated Information Gradient (AIG) flows. From a transport-information-geometry perspective, we derive AIG flows by damping Hamiltonian flows. Examples include Fisher-Rao metric, Wasserstein-2 metric, Kalman-Wasserstein metric and Stein metric. 
In Gaussian families, we verify the existence of AIG flows. Here we show that the AIG flow corresponds to a well-posed ODE system in the space of symmetric positive definite matrices. 
We rigorously prove the convergence rate of AIG flows based on the geodesic convexity of the loss function under both Fisher-Rao metric and Wasserstein metric. 
Besides, we handle two difficulties in numerical implementations of AIG flows under Wasserstein metric for sampling. On the one hand, as pointed out in \cite{uaapb,affpd}, the logarithm of density term (gradient of KL divergence) is difficult to approximate in particle formulations. We propose a novel kernel selection method, whose bandwidth is learned by sampling from Brownian motions. We call it the BM method. On the other hand, we notice that the AIG flow can be a numerically stiff system, especially in high-dimensional sample spaces. This is because the solution of AIG flows can be close to the boundary of the probability space. To handle this issue, we propose an adaptive restart technique, which accelerates and stabilizes the discrete-time algorithm. Numerical results in Bayesian Logistic regression and Bayesian neural networks indicate the validity of the BM method and the acceleration effects of proposed AIG flows. 

This paper is organized as follows. Section \ref{sec:review} briefly reviews gradient flows and accelerated gradient flows in Euclidean space. Then, the information metrics  in probability space and their corresponding gradient and Hamiltonian flows are introduced. In Section \ref{sec:aigf}, we formulate AIG flows, under Fisher-Rao metric, Wasserstein metric, Kalman-Wasserstein metric and Stein metric. We theoretically prove the convergence rate of AIG flows in Section \ref{sec:conv_aig}. Section \ref{sec:dis_alg} presents the discrete-time algorithm for W-AIG flows, including the BM method and the adaptive restart technique. Section \ref{sec:num} provides numerical experiments. In supplementary materials, we also provide discrete-time algorithms for both Kalman-Wasserstein AIG and Stein AIG flows.

\section{Reviews}
\label{sec:review}
In this section, we review gradient flows and accelerated gradient flows in Euclidean space. Then, we introduce the optimization problems in probability spaces, and review several definitions of information metrics therein. Based on these metrics, we demonstrate gradient and Hamiltonian flows in probability space. These formulations serve necessary preparations for us to derive accelerated gradient flows in probability space. 
See detailed analysis on metrics in probability space in \cite{dgisi,agffm,fasda}. 

\subsection{Accelerated gradient flows in Euclidean space}
\label{ssec:eucl}
Consider an optimization problem in Euclidean space:
$$
\min_{x\in\mbR^n} f(x),
$$
where $f(x)$ is a given convex function with $L$-Lipschitz continuous gradient. 
Here $\la\cdot,\cdot \ra $ and $\|\cdot\|$ are the Euclidean inner product and norm in $\mbR^n$. 
The gradient descent method has the update rule
$$
x_{k+1} = x_k-\tau_k\nabla f(x_k),
$$
where $\tau_k>0$ is a step size. With the limit $\tau_k\to 0$, the continuous-time limit of gradient descent method is the gradient flow (GF)
$$
\dot x_t = -\nabla f(x_t).
$$
To accelerate the gradient descent method, Nesterov introduced an accelerated method \cite{nesterov}:
$$
\left\{
\begin{aligned}
&\mfx_k=\mfy_{k-1}-\tau_k\nabla f(\mfy_{k-1}),\\
&\mfy_k=\mfx_k+\alpha_k(\mfx_k-\mfx_{k-1}).
\end{aligned}
\right.
$$
Here $\alpha_k$ depends on the convexity of $f(x)$. If $f(x)$ is $\beta$-strongly convex, then $\alpha_k = \frac{\sqrt{L}-\sqrt{\beta}}{\sqrt{L}+\sqrt{\beta}}$; otherwise, $\alpha_k=\frac{k-1}{k+2}$. \cite{adefm} show that the continuous-time limit of Nesterov's accelerated method satisfies an ODE, which is known as the accelerated gradient flow (AGF):  
\begin{equation}\label{equ:nesterov}
    \ddot x_t+\alpha_t\dot x_t+\nabla f(x_t) = 0. 
\end{equation}
Here $\alpha_t=2\sqrt{\beta}$ if $f(x)$ is $\beta$-strongly convex; $\alpha_t=3/t$ for general convex $f(x)$.  
An important observation in \cite{hdm} is that the accelerated gradient flow \eqref{equ:nesterov} can be formulated as a damped Hamiltonian flow:
$$
\begin{bmatrix}
\dot x_t\\\dot p_t
\end{bmatrix}
+\begin{bmatrix}
0\\\alpha_t p_t
\end{bmatrix}-\begin{bmatrix} 0&I\\-I&0\end{bmatrix}\begin{bmatrix}\nabla_x H^E(x_t,p_t)\\\nabla_p H^E(x_t,p_t)\end{bmatrix}=0.
$$
where $x$ is the state variable and $p$ is the momentum variable. The Hamiltonian function satisfies $H^E(x,p) = \frac{\|p\|^2}{2}+f(x)$, which consists of Euclidean kinetic function $\frac{\|p\|^2}{2}$ and potential function $f(x)$.
In other words, one can formulate an accelerated gradient flow by adding a linear momentum term into the Hamiltonian flow. 
Later on, we follow this damped Hamiltonian perspective and derive related accelerated gradient flows in probability space.
\subsection{Metrics in probability space}
\label{ssec:metric}
In practice, machine learning problems, especially Bayesian sampling problems, can be formulated as optimization problems in probability space. In other words, consider 
\begin{equation*}
    \min_{\rho\in\mcP(\Omega)} E(\rho),
\end{equation*}
where $\Omega\subset \mbR^n$ is a region and the set of probability density is denoted by $\mcP(\Omega) =\{\rho\in \mathcal{F}(\Omega)\colon \int_\Omega \rho dx=1,\quad \rho\geq0\}$.
Here $\mcF(\Omega)$ represents the set of smooth functions on $\Omega$. In practice, $E(\rho)$ is often chosen as a divergence or metric functional between $\rho$ and a target density $\rho^*\in\mcP(\Omega)$. 

In literature, it has been shown that various sampling algorithms correspond to gradient flows of $E(\rho)$, depending on the metrics in probability space. We brief review the definition of metrics in probability space as follows.
\begin{definition}[Metric in probability space]
Denote the tangent space at $\rho\in\mcP(\Omega)$ by $T_\rho\mathcal{P}(\Omega)=\left\{\sigma\in\mcF(\Omega):\int \sigma dx=0.\right\}$. The cotangent space at $\rho$, $T^*_\rho\mcP(\Omega)$,  can be treated as the quotient space $\mcF(\Omega)/\mbR$.
A metric tensor $G(\rho): T_\rho\mcP(\Omega)\to T^*_\rho\mcP(\Omega)$ is an invertible mapping from $T_\rho \mcP(\Omega)$ to $T_\rho^* \mcP(\Omega)$. This metric tensor defines the metric (inner product) on tangent space $T_\rho\mcP(\Omega)$:
$$
g_\rho(\sigma_1,\sigma_2) =\int   \sigma_1G(\rho)\sigma_2dx=\int   \Phi_1G(\rho)^{-1}\Phi_2dx,\quad \sigma_1, \sigma_2\in T_\rho\mcP(\Omega)
$$
where $\Phi_i$ is the solution to $\sigma_i= G(\rho)^{-1}\Phi_i$, $i=1,2$. 
\end{definition}
Along with a given metric, the probability space $\mcP(\Omega)$ can be viewed as an infinite-dimensional Riemannian manifold, which is known as the density manifold \cite{tdmac}. We review four examples of metrics in $\mcP(\Omega)$: the Fisher-Rao metric from information geometry, the Wasserstein metric from optimal transport, the Kalman-Wasserstein metric from ensemble Kalman sampling and the Stein metric from Stein variational gradient method. For simplicity, we denote $\mbE_{\rho}[\Phi] = \int \Phi\rho dx$. 

\begin{example}[Fisher-Rao metric]
The inverse of Fisher-Rao metric tensor is defined by
$$
G^F(\rho)^{-1}\Phi = \rho\lp\Phi -\mbE_\rho[\Phi]\rp, \quad \Phi\in T_\rho^*\mcP(\Omega).
$$
\end{example}
\begin{example}[Wasserstein metric]
The inverse of Wasserstein metric tensor writes
$$
G^W(\rho)^{-1}\Phi = -\nabla\cdot(\rho\nabla\Phi), \quad \Phi\in T_\rho^*\mcP(\Omega).
$$
\end{example}

\begin{example}[Kalman-Wasserstein metric, \cite{ild}]
The inverse of metric tensor is defined by
$$
G^{KW}(\rho)^{-1}\Phi = -\nabla\cdot(\rho C^\lambda(\rho)\nabla\Phi), \quad \Phi\in T_\rho^*\mcP(\Omega).
$$
Here $\lambda\geq 0$ is a given regularization constant and $C^\lambda(\rho)\in \mbR^{n\times n}$ follows
$$
C^\lambda(\rho) = \int (x-m(\rho))(x-m(\rho))^T\rho dx+\lambda I,\quad m(\rho) = \int x\rho dx.
$$
\end{example}

\begin{example}[Stein metric, \cite{Liu2017_steina,otgos}]
The inverse of Stein metric tensor is defined by
$$
G^S(\rho)^{-1}\Phi(x) = -\nabla_x\cdot\pp{\rho(x) \int k(x,y) \rho(y) \nabla_y\Phi(y) dy}.
$$
Here $k(x,y)$ is a given positive kernel function.
\end{example}

\subsection{Gradient flows and Hamiltonian flows in probability space}
The gradient flow for $E(\rho)$ in $(\mcP(\Omega), g_\rho)$ takes the form
\begin{equation*}
\p_t\rho_t = -G(\rho_t)^{-1}\frac{\delta E}{\delta \rho_t}.
\end{equation*}
Here $\frac{\delta E}{\delta \rho_t}$ is the $L^2$ first variation w.r.t. $\rho_t$. 
For example, the Wasserstein gradient flow writes
\begin{equation*}
\begin{split}
\p_t\rho_t =&-G^W(\rho_t)^{-1}\frac{\delta E}{\delta\rho_t}=\nabla \cdot \lp\rho_t\nabla \frac{\delta E}{\delta \rho_t}\rp.
\end{split}
\end{equation*}

We then briefly review Hamiltonian flows in probability space. 
Given a metric $\mcG(\rho)$, denote the density function $\rho_t$ as a state variable while function $\Phi_t$ as a momentum variable. The Hamiltonian flow in probability space follows
\begin{equation}\label{equ:hf2}
\p_t\begin{bmatrix}\rho_t\\\Phi_t\end{bmatrix}
-\begin{bmatrix} 0&1\\-1&0\end{bmatrix}\begin{bmatrix}\frac{\delta}{\delta \rho_t} \mcH(\rho_t,\Phi_t)\\\frac{\delta}{\delta \Phi_t} \mcH(\rho_t,\Phi_t)
\end{bmatrix} = 0,
\end{equation}
with respect to the Hamiltonian in density space by  $$\mcH(\rho_t,\Phi_t)=\frac{1}{2}\int   \Phi_t G(\rho_t)^{-1}\Phi_tdx+E(\rho_t).
$$
Similar to the Euclidean Hamiltonian function, the Hamiltonian functional in density space consists of a kinetic energy $\frac{1}{2}\int \Phi G(\rho)^{-1}\Phi dx$ and a potential energy $E(\rho)$. 



\section{Accelerated information gradient flow}
\label{sec:aigf}
We introduce the accelerated gradient flow in probability density space as follows. Let $\alpha_t\geq0$ be a scalar function of $t$. We add a damping term $\alpha_t\Phi_t$ to the Hamiltonian flow \eqref{equ:hf2}:  
\begin{equation}\label{equ:hf3}
\p_t\begin{bmatrix}\rho_t\\\Phi_t\end{bmatrix}
+\begin{bmatrix}0\\\alpha_t\Phi_t\end{bmatrix}
-\begin{bmatrix} 0&1\\-1&0\end{bmatrix}\begin{bmatrix}\frac{\delta}{\delta \rho_t} \mcH(\rho_t,\Phi_t)\\\frac{\delta}{\delta \Phi_t} \mcH(\rho_t,\Phi_t)
\end{bmatrix} = 0.
\end{equation}
We call dynamics \eqref{equ:hf3} {\em Accelerated Information Gradient (AIG) flow}.
\begin{proposition}
The accelerated information gradient flow satisfies
\begin{equation}\label{equ:aig}
\left\{\begin{aligned}
&\p_t\rho_t-G(\rho_t)^{-1}\Phi_t=0,\\
&\p_t \Phi_t+\alpha_t\Phi_t+\frac{1}{2}\frac{\delta}{\delta \rho_t}\lp \int  \Phi_tG(\rho_t)^{-1}\Phi_tdx\rp+\frac{\delta E}{\delta \rho_t}=0,\\
\end{aligned}\right.\tag{AIG}
\end{equation}
with initial values $\rho_t|_{t=0}=\rho_0$ and $\Phi_t|_{t=0}=0$. 
\end{proposition}

We give examples of AIG flows under several metrics, such as Fisher-Rao metric, Wasserstein metric, Kalman-Wasserstein metric and Stein metric. See detailed derivations in the supplementary material.
\begin{example}[Fisher-Rao AIG flow]
\begin{equation}\label{equ:f_aig}
\left\{\begin{aligned}
&\p_t\rho_t-\lp\Phi_t -\mbE_{\rho_t}[\Phi_t]\rp\rho_t=0,\\
&\p_t\Phi_t+\alpha_t\Phi_t+\frac{1}{2}\Phi_t^2-\mbE_{\rho_t}[\Phi_t]\Phi_t+\frac{\delta E}{\delta \rho_t}=0.
\end{aligned}\right.\tag{F-AIG}
\end{equation}
\end{example}

\begin{example}[Wasserstein AIG flow, \cite{coeiw,affpd}]
\begin{equation}\label{equ:w_aig}
\left\{\begin{aligned}
&\p_t\rho_t+\nabla \cdot(\rho_t\nabla \Phi_t)=0,\\
&\p_t \Phi_t+\alpha_t\Phi_t+ \frac{1}{2}\|\nabla \Phi_t\|^2+\frac{\delta E}{\delta \rho_t}=0.\\
\end{aligned}\right.\tag{W-AIG}
\end{equation}
\end{example}
\begin{example}[Kalman-Wasserstein AIG flow]
\begin{equation}\label{equ:kw_aig}
\left\{\begin{aligned}
&\p_t\rho_t+\nabla \cdot(\rho_tC^\lambda(\rho_t)\nabla \Phi_t)=0,\\
&\p_t \Phi_t+\alpha_t\Phi_t+ \frac{1}{2}\Big((x-m(\rho_t))^TB_{\rho_t}(\Phi_t)(x-m(\rho_t))\\
&+\nabla\Phi_t(x)^TC^\lambda(\rho_t)\nabla \Phi_t(x)\Big)+\frac{\delta E}{\delta \rho_t}=0.\\
\end{aligned}\right.\tag{KW-AIG}
\end{equation}
Here we denote $B_{\rho}(\Phi) = \int \nabla \Phi \nabla \Phi^T \rho dx$. 
\end{example}

\begin{example}[Stein AIG flow]
\begin{equation}\label{equ:s_aig}
\left\{\begin{aligned}
&\p_t\rho_t(x) +\nabla_x\cdot\pp{\rho_t(x) \int k(x,y) \rho_t(y) \nabla_y\Phi_t(y) dy}=0,\\
&\p_t\Phi_t(x) +\alpha_t \Phi_t(x)+\int   \nabla \Phi_t(x)^T\nabla \Phi_t(y) k(x,y) \rho_t(y)dy +\frac{\delta E}{\delta \rho_t}(x)=0.
\end{aligned}\right.\tag{S-AIG}
\end{equation}
\end{example}


To design fast sampling algorithms, we need to reformulate  the evolution of probability in term of samples. In other words, PDEs in term of $(\rho, \Phi)$ is the Eulerian formulation in fluid dynamics, while the particle formulation is the flow map equation, known as the Lagrangian formulation. We present examples for W-AIG flow, KW-AIG flow and S-AIG flow, which have particle formulations. We suppose that $X_t\sim \rho_t$ and $V_t=\nabla\Phi_t(X_t)$ are the position and the velocity of a particle at time $t$.
\begin{example}[Particle W-AIG flow]\label{exm:waig_p}
The particle dynamical system for the flow \eqref{equ:w_aig} writes
\begin{equation}\label{equ:w_aig_p}
\left\{
\begin{aligned}
&\frac{d}{dt}X_t = V_t,\\ 
&\frac{d}{dt}V_t =-\alpha_tV_t -\nabla\lp\frac{\delta E}{\delta \rho_t}\rp(X_t) .
\end{aligned}\right.
\end{equation}
\end{example}

\begin{example}[Particle KW-AIG flow]\label{exm:kwaig_p}
The particle dynamical system for the flow \eqref{equ:kw_aig} writes
\begin{equation}\label{equ:kw_aig_p}
\left\{
\begin{aligned}
&\frac{dX_t}{dt} = C^\lambda(\rho_t)V_t,\\ 
&\frac{dV_t}{dt} =-\alpha_tV_t-\mathbb{E}[V_tV_t^T](X_t-\mbE[X_t])-\nabla\lp\frac{\delta E}{\delta \rho_t}\rp(X_t).
\end{aligned}\right.
\end{equation}
Here the expectation is taken over the particle system. 
\end{example}

\begin{example}[Particle S-AIG flow]\label{exm:saig_p}
The particle dynamical system for the flow \eqref{equ:s_aig} writes
\begin{equation}\label{equ:s_aig_p}
\left\{
\begin{aligned}
&\frac{d X_t}{dt} = \int k(X_t,y) \nabla \Phi_t (y)\rho_t(y)dy,\\
&\frac{d V_t}{dt} =-\alpha_t V_t-\int V_t^T\nabla \Phi_t (y) \nabla_x k(X_t,y) \rho_t(y)dy-\nabla\lp\frac{\delta E}{\delta \rho_t}\rp(X_t).
\end{aligned}\right.
\end{equation}
\end{example}
We notice that dynamics in examples \ref{exm:waig_p} to \ref{exm:saig_p} are mean-field dynamics. Here the mean-field represents that the dynamics evolves its own probability density function in its path. In addition, they are also mean field Markov process. Here the Markov property holds in the sense that the update of dynamics only depends on the current time probability density. Shortly, we will design a finite dimensional particle dynamical system to simulate these proposed dynamics. 

In later on algorithm and convergence analysis, the choice of $\alpha_t$ is important. Similar as the ones in Euclidean space, $\alpha_t$ depends on the convexity of $E(\rho)$ w.r.t. given metrics.  
\begin{definition}[Convexity in probability space]\label{def:hess}
For a functional $E(\rho)$ defined on the probability space, we say that $E(\rho)$ is $\beta$-strongly convex w.r.t. metric $g_\rho$ if there exists a constant $\beta\geq 0$ such that for any $\rho\in \mcP(\Omega)$ and any $\sigma\in T_\rho \mcP(\Omega)$, we have
\begin{equation*}
g_\rho(\Hess E(\rho)\sigma,\sigma)\geq \beta g_\rho(\sigma,\sigma).
\end{equation*}
Here $\Hess$ is the Hessian operator w.r.t. $g_\rho$. If $\beta=0$, we say that $E(\rho)$ is convex w.r.t. metric $g_\rho$. 
\end{definition}
Again, if $E(\rho)$ is $\beta$-strongly convex for $\beta>0$, then $\alpha_t=2\sqrt{\beta}$; if $E(\rho)$ is convex, then $\alpha_t=3/t$. 



We can also formulate W-AIG flows in probability models. For instance, the W-AIG flow in Gaussian families becomes an ODE system, which corresponds to updates of covariance matrices. 

\begin{proposition}[W-AIG flows in Gaussian families]
\label{thm:connect}
Suppose that $\rho_0,\rho^*$ are Gaussian distributions with zero means and their covariance matrices are $\Sigma_0$ and $\Sigma^*$. $E(\Sigma)$ evaluates the KL divergence from $\rho$ to $\rho^*$:
\begin{equation}\label{esigma}
E(\Sigma)=\frac{1}{2}\lb \tr (\Sigma (\Sigma^*)^{-1})-\log\det(\Sigma (\Sigma^*)^{-1})-n\rb,
\end{equation}
Let $(\Sigma_t,S_t)$ be the solution to 
\begin{equation}\label{equ:w_aig_g}
\left\{
\begin{aligned}
&\dot \Sigma_t-2(S_t\Sigma_t+\Sigma_t S_t)=0,\\
&\dot S_t+\alpha_tS_t+2S_t^2+\nabla_{\Sigma_t} E(\Sigma_t)=0,
\end{aligned}\right.\tag{W-AIG-G}
\end{equation}
with initial values $\Sigma_t|_{t=0}=\Sigma_0$ and $S_t|_{t=0}=0$. Here $\Sigma_t$ and $S_t$ are symmetric matrices. Then, for any $t\geq 0$, $\Sigma_t$ is well-defined and stays positive definite. Furthermore, we denote
$$
\begin{aligned}
\rho_t(x)=\frac{(2\pi)^{-n/2}}{\sqrt{\det(\Sigma_t)}}\exp\lp-\frac{1}{2}x^T\Sigma_t^{-1}x\rp,\quad\Phi_t(x)=x^TS_tx+C(t),
\end{aligned}
$$
where $C(t)=-t+\frac{1}{2}\int_0^t\log\det(\Sigma_s(\Sigma^*)^{-1})ds.$
Then, $(\rho_t,\Phi_t)$ is the solution to \eqref{equ:w_aig} with initial values $\rho_t|_{t=0}=\rho_0$ and $\Phi_t|_{t=0}=0$.
\end{proposition}
\begin{remark}
If the means of $\rho_0,\rho^*$ are $\mu_0$ and $\mu^*$ instead of $0$, the objective function turns to be
\begin{equation*}
\begin{aligned}
E(\Sigma,\mu)=&\frac{1}{2}[ \tr (\Sigma (\Sigma^*)^{-1})-\log\det(\Sigma (\Sigma^*)^{-1})-n]\\
&+\frac{1}{2}(\mu-\mu^*)^T(\Sigma^*)^{-1}(\mu-\mu^*)].\\
\end{aligned}
\end{equation*}
It is separable in terms of $\mu$ and $\sigma$. For simplicity and clarity, we focus on the case with zero means.
\end{remark}
\begin{remark}
AIG flows can be formulated into general probability models, such as Gaussian mixture models and generative models. We leave the systematic study of AIG flows in models in future works.  
\end{remark}

\section{Convergence rate analysis on AIG flows}
\label{sec:conv_aig}
In this section, we prove the convergence rates of AIG flows under either the Wasserstein metric or the Fisher-Rao metric. This validates the acceleration effect. The proof is motivated by Lyapunov functions of Euclidean accelerated gradient flows in subsection \ref{ssec:eucl}.
\begin{theorem}\label{thm:conv_e}
Suppose that $E(\rho)$ is $\beta$-strongly convex for $\beta>0$. The solution $\rho_t$ to \eqref{equ:f_aig} or \eqref{equ:w_aig} with $\alpha_t=2\sqrt{\beta}$ satisfies
\begin{equation*}
E(\rho_t)\leq C_0e^{-\sqrt{\beta}t}=\mcO\pp{e^{-\sqrt{\beta}t}}.
\end{equation*}
If $E(\rho)$ is convex, then the solution $\rho_t$ to \eqref{equ:f_aig} or \eqref{equ:w_aig} with $\alpha_t=3/t$ satisfies
\begin{equation*}
E(\rho_t)\leq C_0't^{-2}=\mcO(t^{-2}).
\end{equation*}
Here the constants $C_0,C_0'$ only depend on $\rho_0$.
\end{theorem}



\begin{remark}
For $\beta$-strongly convex $E(\rho)$ under the Wasserstein metric, \cite{coeiw} study a compressed Euler equation. They prove similar results with a constant damping coefficient $\alpha_t$. For convex $E(\rho)$ under the Wasserstein metric, \cite{affpd} prove similar results with a technical assumption.
\end{remark}

\begin{remark}
Compared to underdamped Langevin dynamics, W-AIG has the accelerated convergence rate guarantee compared to W-GF and it has a closer relation with the Euclidean accelerated gradient flow. 
\end{remark}

\begin{remark}
The Fisher metric and the Wasserstein metric are two popular metrics to consider in the probability space. Therefore, we focus on deriving the convergence analysis for these two metrics. The convergence results for other general information metrics are interesting problems for future studies.
\end{remark}

In Euclidean case, the convergence rate of accelerated gradient flow is based on the construction of Lyapunov functions. Namely, for $\beta$-strongly convex $f(x)$, consider a Lyapunov function:
$$
\mcE(t) =\frac{e^{\sqrt{\beta}t}}{2} \|\sqrt{\beta}(x_t-x^*)+\dot x_t\|^2+e^{\sqrt{\beta}t}(f(x_t)-f(x^*)).
$$
For general convex $f(x)$, consider a Lyapunov function
$$
\mcE(t) = \frac{1}{2}\left\|(x_t-x^*)+\frac{t}{2}\dot x_t\right\|^2+\frac{t^2}{4}(f(x_t)-f(x^*)). 
$$
Based on different assumptions on the convexity of $f(x)$, we can prove that these Lyapunov function are not increasing w.r.t. $t$. Hence, the convergence rates are obtained.

\begin{remark}
Our choices of the damping parameter $\alpha_t$ are analogous to these in the Euclidean case. In the Euclidean case, the damped Hamiltonian system and the related Lyapunov functions are derived from the Bregman Lagrangian introduced in \cite{avpoa, wilson2016lyapunov}. For simplicity, we only focus on two specific choices of the damping parameters $\alpha_t$, based on the convexity of the energy functional.
\end{remark}


Following Lyapunov functions in Euclidean space, we provide a sketch of the proof for Theorem 1. We first consider the case where $E(\rho)$ is $\beta$-strongly convex for $\beta>0$. Let $T_t$ denote the optimal transport plan from $\rho_t$ to $\rho^*$. Consider a Lyapunov function
\begin{equation}\label{equ:lya}
\begin{aligned}
\mcE(t)=&\frac{e^{\sqrt{\beta}t}}{2}\int  \left\| -\sqrt{\beta} (T_t(x)-x)+ \nabla \Phi_t(x)\right\|^2\rho_t(x) dx\\
&+e^{\sqrt{\beta}t}(E(\rho_t)-E(\rho^*)).
\end{aligned}
\end{equation}
Here the $-(T_t(x)-x)$ term can be viewed as $x_t-x^*$ and $\nabla \Phi_t$ can be viewed as $\dot x
_t$. Different from the Euclidean case, we introduce an important lemma in proving that $\mcE(t)$ is non-increasing.

\begin{lemma}
Denote $u_t=\p_t (T_t)^{-1}\circ T_t$. Then,$u_t$ satisfies
\begin{equation*}
\nabla\cdot \lp\rho_t(u_t-\nabla\Phi_t)\rp=0.
\end{equation*}
We also have
\begin{equation*}
\p_tT_t(x)=-\nabla T_t(x)u_t(x).
\end{equation*}
More importantly, we have
$$
\begin{aligned}
&\int  \la \nabla \Phi_t-u_t, \nabla T_t\nabla \Phi_t\ra\rho_tdx\geq 0,\\
& \int \la\nabla \Phi_t-u_t, \nabla T_t(x)(T_t(x)-x) \ra\rho_t=0.
\end{aligned}
$$
\end{lemma}

We then demonstrate that $\mcE(t)$ is not increasing w.r.t. $t$.

\begin{proposition}\label{prop:lya}
Suppose that $E(\rho)$ satisfies Hess($\beta$) for $\beta> 0$. $\rho_t$ is the solution to (W-AIG) with $\alpha_t=2\sqrt{\beta}$. Then, $\mcE(t)$ defined in \eqref{equ:lya} satisfies $\dot \mcE(t)\leq 0.$ As a result,
$$
E(\rho_t)\leq e^{-\sqrt{\beta}t}\mcE(t)\leq e^{-\sqrt{\beta}t}\mcE(0)=\mcO(e^{-\sqrt{\beta}t}).
$$
\end{proposition}
Note that $\mcE(0)$ only depends on $\rho_0$. This proves the first part of Theorem 1. 

We now focus on the case where $E(\rho)$ is convex. Similarly, we construct the following Lyapunov function.
\begin{equation}\label{equ:lya2}
\begin{aligned}
\mcE(t)=&\frac{1}{2}\int  \left\| - (T_t(x)-x)+\frac{t}{2} \nabla \Phi_t(x)\right\|^2\rho_t(x) dx\\
+&\frac{t^2}{4}(E(\rho_t)-E(\rho^*)).
\end{aligned}
\end{equation}
\begin{proposition}\label{prop:lya2}
Suppose that $E(\rho)$ satisfies Hess($0$). $\rho_t$ is the solution to (W-AIG) with $\alpha_t=3/t$. Then, $\mcE(t)$ defined in \eqref{equ:lya2} satisfies $\dot \mcE(t)\leq 0.$ As a result,
$$
E(\rho_t)\leq \frac{4}{t^2}\mcE(t)\leq \frac{4}{t^2}\mcE(0)=\mcO(t^{-2}).
$$
\end{proposition}
Because $\mcE(0)$ only depends on $\rho_0$, we complete the proof. 

\section{Discrete-time algorithms for AIG flows}
\label{sec:dis_alg}
In this section, we present the discrete-time particle implementation of the flow 
\eqref{equ:w_aig} based on the particle W-AIG flow \eqref{equ:w_aig_p}. Similar discrete-time algorithms of \eqref{equ:kw_aig} and \eqref{equ:s_aig} are provided in the supplementary material. Here we mainly introduce a kernel bandwidth selection method and an adaptive restart technique to deal with difficulties in numerical implementations.

A typical choice of $E(\rho)$ for sampling is the KL divergence 
\begin{equation*}
\mathrm{D}_{\textrm{KL}}(\rho\|\rho^*) = \int\rho \log\frac{\rho}{e^{-f}}dx-\log Z,
\end{equation*}
where the target density $\rho^*(x)\propto \exp(-f(x))$ and $Z=\int \exp(-f(x))dx$. Then, \eqref{equ:w_aig_p} is equivalent to
\begin{equation}\label{equ:w_aig_p_kl}
\left\{
\begin{aligned}
&dX_t = V_t dt,\\
&dV_t =-\alpha_tV_t dt-\nabla f(X_t)dt-\nabla \log \rho_t(X_t)dt.
\end{aligned}
\right.
\end{equation}

Consider a particle system $\{X_0^i\}_{i=1}^N$ and let $V_0^i=0$. 
In $k$-th iteration, the update rule follows
\begin{equation}\label{equ:vx}
\left\{
\begin{aligned}
&X_{k+1}^i = X_k^i+\sqrt{\tau_k} V_{k+1}^i,\\
&V_{k+1}^i = \alpha_kV_k^i-\sqrt{\tau_k}(\nabla f(X_k^i)+\xi_k(X_k^i)),\\
\end{aligned}\right.
\end{equation}
for $i=1,2\dots N$. If $E(\rho)$ is $\beta$-strongly convex, then $\alpha_k = \frac{1-\sqrt{\beta\tau_k}}{1+\sqrt{\beta\tau_k}}$; if $E(\rho)$ is convex or $\beta$ is unknown, then $\alpha_k = \frac{k-1}{k+2}$. 
Here $\xi_k(x)$ is an approximation of $\nabla\log\rho_k(x)$.
For a general distribution, we use the kernel density estimation (KDE) \cite{ioskn}, $\tilde \rho_k(x) =\frac{1}{N}\sum_{i=1}^NK(x,X_k^i)$ to approximate $\rho_k(x)$. Here $K(x,y)$ is a positive kernel function. Then, $\xi_k$ writes
\begin{equation}\label{xi:rbf}
\xi_k(x) = \nabla\log\tilde \rho_k(x) = \frac{\sum_{i=1}^N\nabla_xK(x,X_k^i)}{\sum_{i=1}^NK(x,X_k^i)}.
\end{equation}
A common choice of $K(x,y)$ is a Gaussian kernel with the bandwidth $h$, $K(x,y) = (2\pi h)^{-n/2}\exp\lp-\|x-y\|^2/(2h)\rp$. Such approximation can also be found in information-theoretic learning \cite{itl} and independent component analysis (ICA) \cite{aitat}. 

There are two difficulties in the time discretization. For one thing, the bandwidth $h$ strongly affects the estimation of $\nabla\log\rho_t$, so we propose the BM method to learn the bandwidth from Brownian-motion samples. For another, the second equation in \eqref{equ:w_aig} is the Hamilton-Jacobi equation, which usually has strong stiffness. In numerical trials, we observe that the densities from the particles may collapse in certain dimensions following W-AIG flows, even for Gaussian target density. Therefore, we propose an adaptive restart technique to deal with this problem.

\begin{remark}
Using symplectic integrators for the particle implementation of W-AIG could help improve the performance. It is important to study the time-discretization of the (damped) Hamiltonian flow in the future.
\end{remark}

\subsection{Learn the bandwidth via Brownian motion}
\label{ssec:bm}
SVGD uses a median (MED) method to choose the bandwidth, i.e.,
\begin{equation}\label{h:med}
h_k = \frac{1}{2\log(N+1)}\text{median}\lp\{\|X_k^i-X_k^j\|^2\}_{i,j=1}^N\rp. 
\end{equation}
Liu et al. \cite{afomo} propose a Heat Equation (HE) method to adaptively adjust bandwidth. Motivated by the HE method, we introduce the Brownian motion (BM) method to adaptively learn the kernel bandwidth based on Brownian-motion samples generated in each iteration. 

Given the bandwidth $h$, $\{X_k^i\}_{i=1}^N$ and a step size $s$, we can compute two particle systems:
\begin{equation*}
Y_k^i(h) = X_k^i-s\xi_k(x;h), \quad Z_k^i = X_k^i+\sqrt{2s}B^i,\quad i=1,\dots N
\end{equation*}
where $B^i$ is the standard Brownian motion. Denote the empirical distributions of $\{X_k^i\}_{i=1}^N$, $\{Y_k^i\}_{i=1}^N$ and $\{Z_k^i\}_{i=1}^N$ by $\hat \rho_X, \hat \rho_Y$ and $\hat \rho_Z$. With $n\to \infty$, we shall have $\hat \rho_Y=\hat\rho_Z=\rho_t|_{t=s}$, where $\hat \rho_t$ satisfies $\p_t \hat \rho_t=\Delta\hat \rho_t=\nabla\cdot(\hat \rho_t\nabla\log \hat \rho_t)$ 
with initial value $\hat \rho_t|_{t=0}=\hat \rho_X$. With an appropriate bandwidth $h$, we shall also have $\hat\rho_Y=\rho_t|_{t=s}$.
Hence, we consider the following optimization problem
\begin{equation}\label{equ:mmd}
\min_h \MMD(\hat\rho_Y,\hat\rho_Z)=\int\int  (\hat \rho_Y(y)-\hat \rho_Z(y)) k(y,z) (\hat \rho_Y(z)-\hat \rho_Z(z))dydz.
\end{equation}
where MMD (maximum mean discrepancy) evaluates the similarity between $\{Y_k^i\}_{i=1}^N$ and $\{Z_k^i\}_{i=1}^N$. Here, the kernel $k(y,z)$ in MMD is chosen as a Gaussian kernel with bandwidth $1$. So we optimize \eqref{equ:mmd} using the bandwidth $h_{k-1}$ from the last iteration as the initialization. For simplicity we denote 
$$
\text{BM}(h_{k-1},\{X_k^i\}_{i=1}^N,s)
$$ 
as the minimizer of problem \eqref{equ:mmd}. It is the output of the BM method.

\begin{remark}
Besides KDE, there are other methods that approximate the term $\nabla\log\rho_t(x)$ (compute $\xi_k$) via a kernel function, such as the blob method \cite{abmfd} and the diffusion map \cite{affpd}. The BM method can also select the kernel bandwidth for these methods. 
\end{remark}

\subsection{Adaptive restart}
To enhance the practical performance, we introduce an adaptive restart technique, which shares the same idea of gradient restart in \cite{arfag,tsdcm} under the Euclidean case. Consider
\begin{equation}\label{phik}
\varphi_k =- \sum_{i=1}^N\la V_{k+1}^i, \nabla f(X_k^i)+\xi_k(X_k^i)\ra,
\end{equation}
which can be viewed as discrete-time approximation of 
$$
-g_{\rho_t}^W(\p_t\rho_t,G^W(\rho_t)^{-1}\frac{\delta E}{\delta \rho_t})=-\p_tE(\rho_t).
$$
If $\varphi_k<0$, then we restart the algorithm with initial values $X_0^i=X_k^i$ and $V_0^i=0$. This essentially keeps $\p_tE(\rho_t)$ negative along the trajectory. 
The overall algorithm is summarized below.
\begin{algorithm}[!htp]
\caption{Discrete-time particle implementation of W-AIG flow}
\label{alg:w_aig}
\begin{algorithmic}[1]
\REQUIRE initial positions $\{X_0^i\}_{i=1}^N$, step size $\tau$, number of iteration $L$.
\STATE Set $k=0$, $V_0^i=0, i=1,\dots N$. Set the bandwidth $h_0$ by MED \eqref{h:med}. 
\FOR{$l=1,2,\dots L$}
\STATE Compute $h_l$ based on BM method:  $h_l=\text{BM}(h_{l-1},\{X_k^i\}_{i=1}^N,\sqrt{\tau})$.
\STATE Calculate $\xi_k(X_k^i)$ by \eqref{xi:rbf} with bandwidth $h_{l}$. 
\STATE 
For $i=1,2,\dots N$, update $V_{k+1}^i$ and $X_{k+1}^i$ by \eqref{equ:vx}.
%
\STATE Compute $\varphi_k$ by \eqref{phik}. 
\STATE If $\varphi_k<0$, set $X_0^i=X_k^i$ and $V_0^i=0$ and $k=0$; otherwise set $k=k+1$. 
\ENDFOR
\end{algorithmic}
\end{algorithm}

\section{Numerical experiments}
\label{sec:num}
In this section, we present several numerical experiments to demonstrate the effectiveness of BM method, the acceleration effect of AIG flows, and the strength of adaptive restart technique. 
Implementation details 
are provided in the supplementary material.  

\subsection{Toy examples}
We first generate samples from a toy bi-modal distribution in \citep{viwnf}. We compare sampling algorithms based on gradient flows and accelerated gradient flows under Wasserstein metric, Kalman-Wasserstein metric and Stein metric. The number of particles follow $N=200$. The initial distribution of the particle system follows $\mcN([0,10]',I)$. 

For the approximation of $\nabla\log \rho_k$, we use a Gaussian kernel and the kernel bandwidth is selected by the BM method. We apply the restart technique for discrete-time algorithms of AIG flows. For W-GF, W-AIG, SVGD and S-AIG, we take the step size $\tau_k=0.1$. For KW-GF and KW-AIG, we set the regularization parameter $\lambda=1$ and the step size $\tau_k=0.02$. We choose a smaller step size for the Kalman-Wasserstein metric because the particle system may blow up for a larger step size. For SVGD and S-AIG, we use a Gaussian kernel with fixed bandwidth $1$. The step size of SVGD is adjusted by Adagrad. 

From Figure \ref{fig:toy}, the convergence rate of the particle system depends on the metric. For a fixed metric, samples generated by accelerated gradient flows always converge faster than the ones generated by gradient flows. 

\begin{figure}[!htbp]
\centering
\begin{minipage}[t]{0.24\textwidth}
\centering
\includegraphics[width=\linewidth]{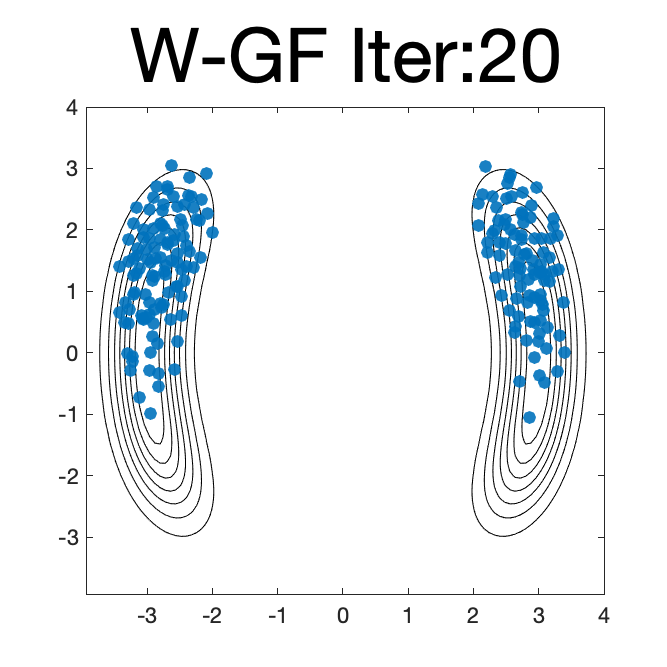}
\end{minipage}
\begin{minipage}[t]{0.24\textwidth}
\centering
\includegraphics[width=\linewidth]{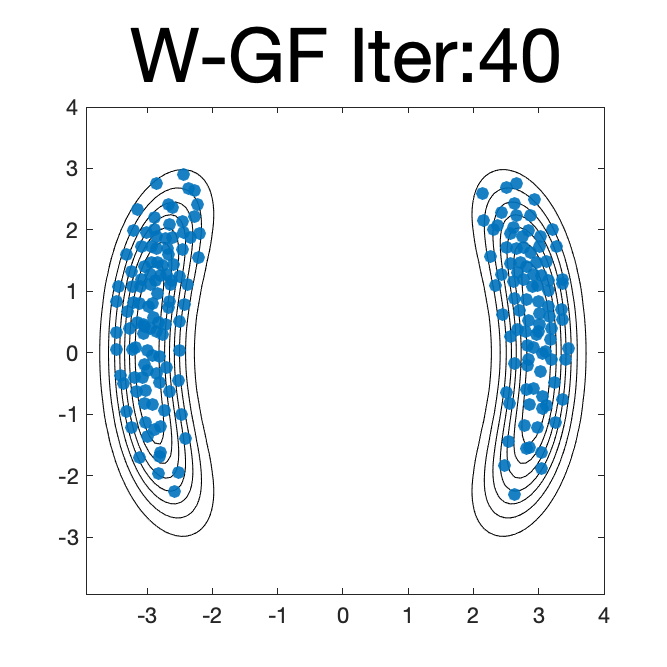}
\end{minipage}
\begin{minipage}[t]{0.24\textwidth}
\centering
\includegraphics[width=\linewidth]{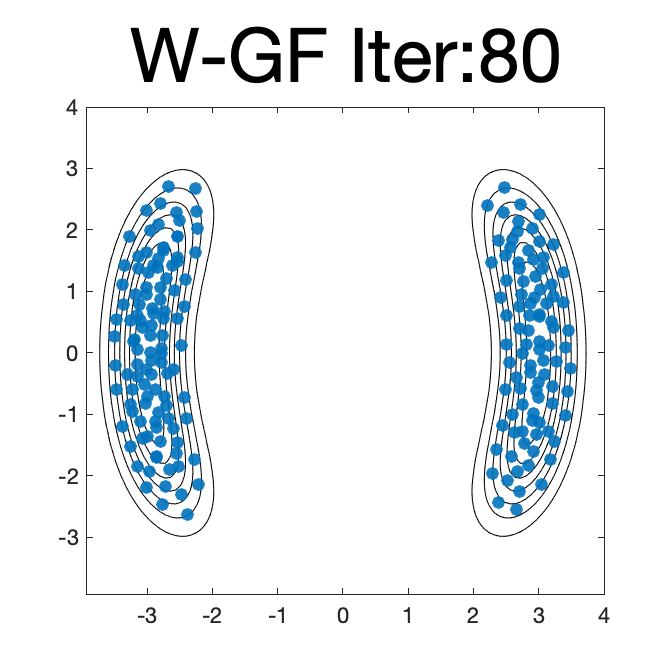}
\end{minipage}
\begin{minipage}[t]{0.24\textwidth}
\centering
\includegraphics[width=\linewidth]{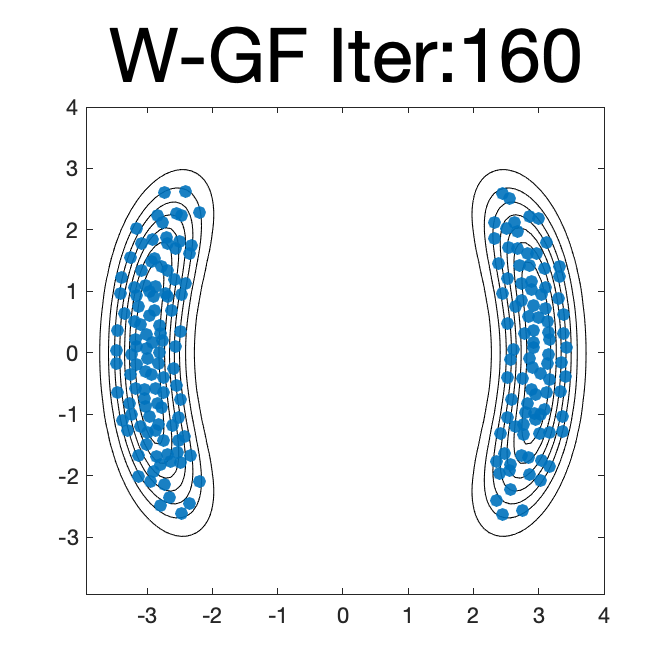}
\end{minipage}
\begin{minipage}[t]{0.24\textwidth}
\centering
\includegraphics[width=\linewidth]{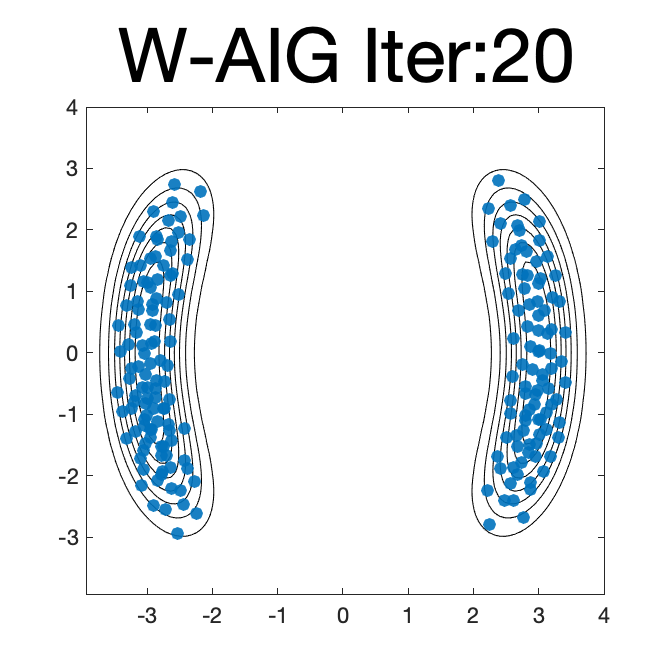}
\end{minipage}
\begin{minipage}[t]{0.24\textwidth}
\centering
\includegraphics[width=\linewidth]{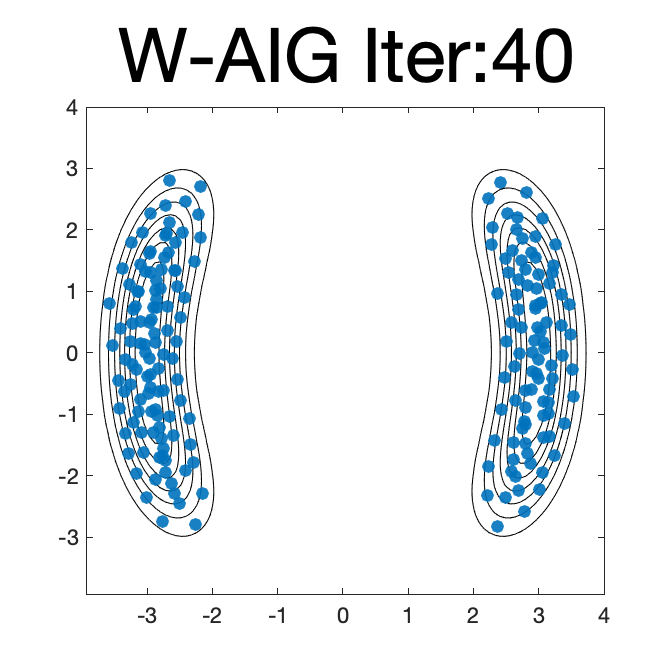}
\end{minipage}
\begin{minipage}[t]{0.24\textwidth}
\centering
\includegraphics[width=\linewidth]{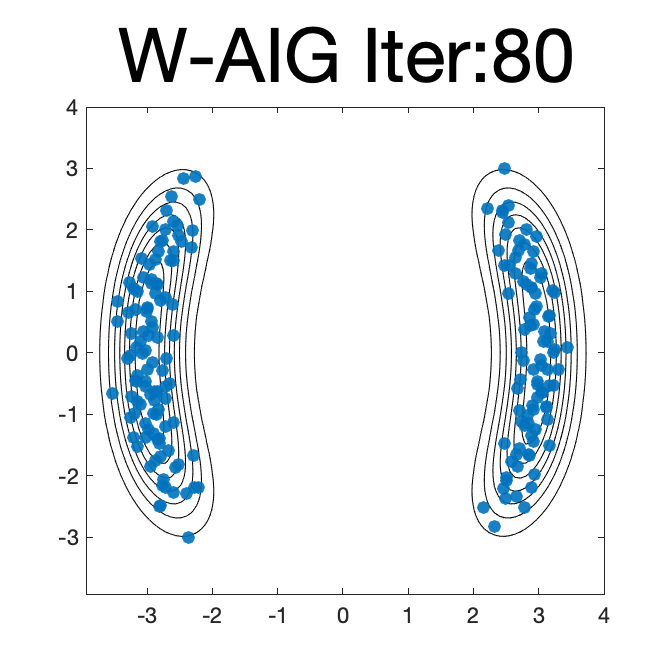}
\end{minipage}
\begin{minipage}[t]{0.24\textwidth}
\centering
\includegraphics[width=\linewidth]{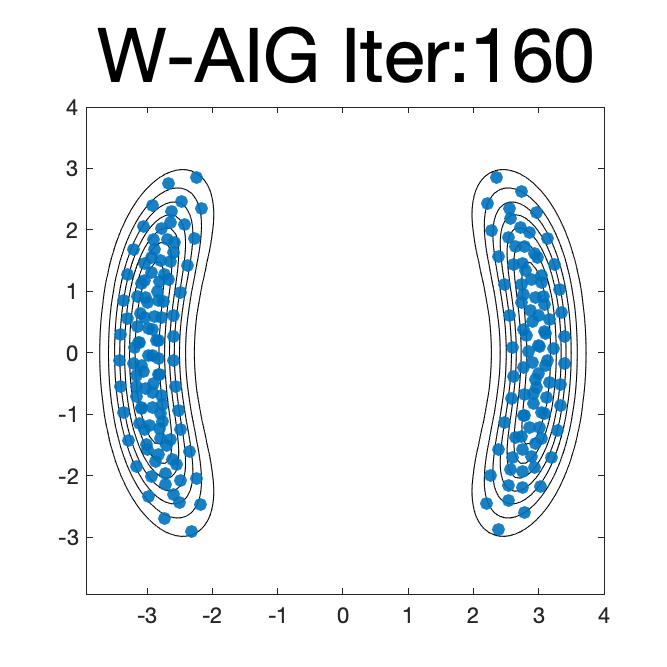}
\end{minipage}
\begin{minipage}[t]{0.24\textwidth}
\centering
\includegraphics[width=\linewidth]{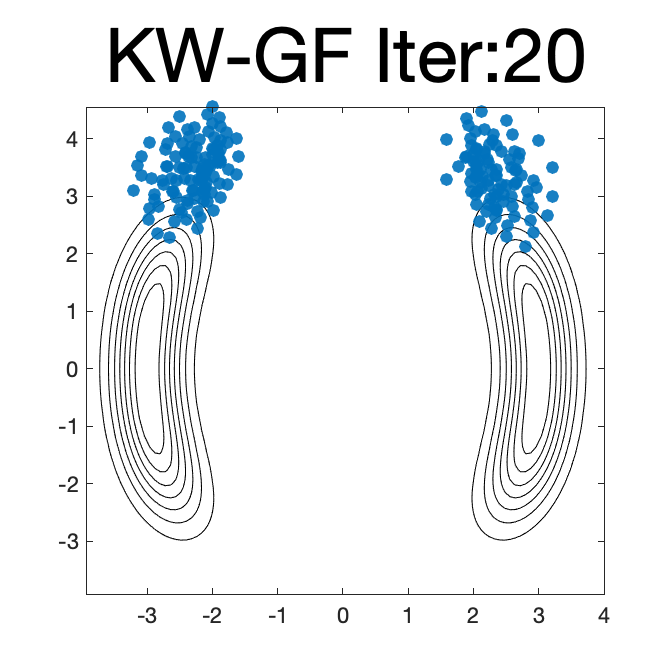}
\end{minipage}
\begin{minipage}[t]{0.24\textwidth}
\centering
\includegraphics[width=\linewidth]{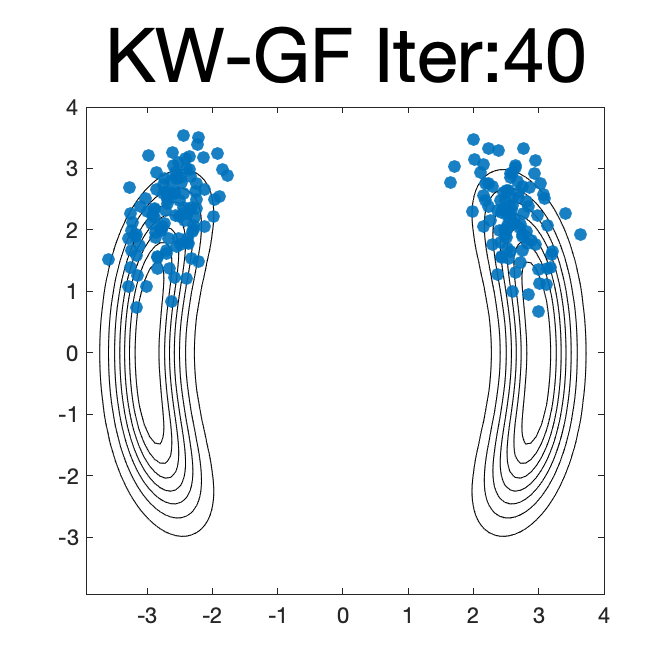}
\end{minipage}
\begin{minipage}[t]{0.24\textwidth}
\centering
\includegraphics[width=\linewidth]{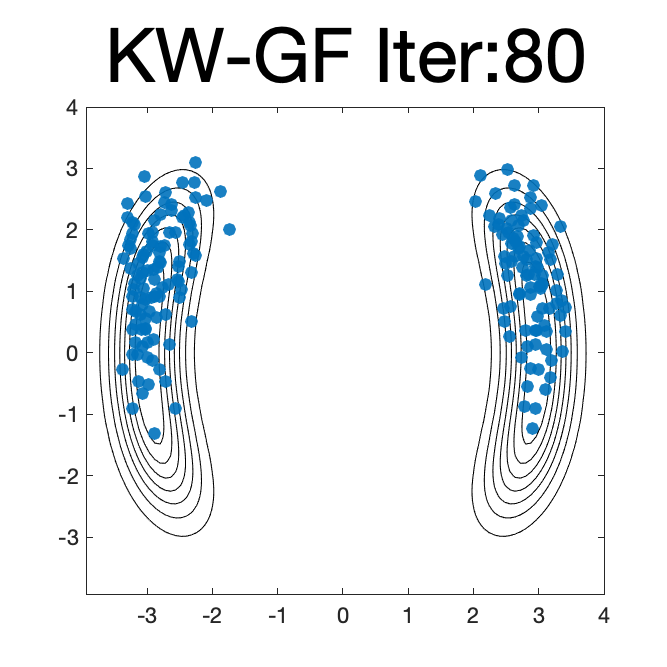}
\end{minipage}
\begin{minipage}[t]{0.24\textwidth}
\centering
\includegraphics[width=\linewidth]{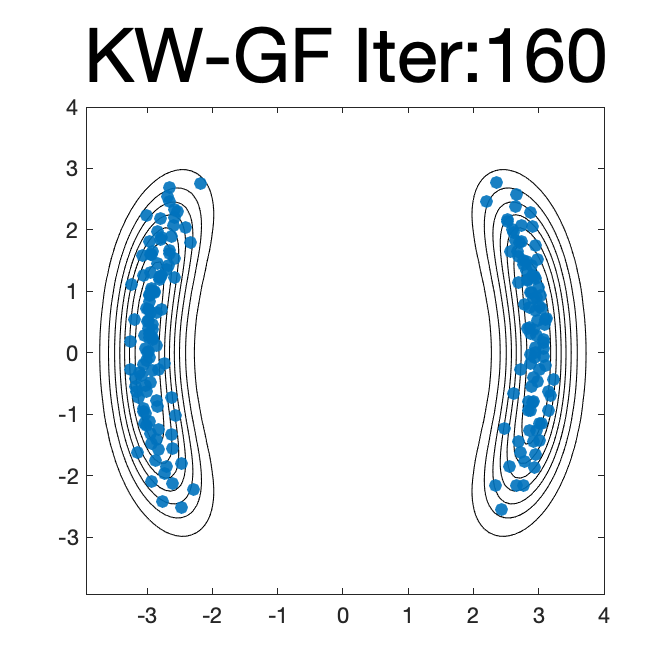}
\end{minipage}
\begin{minipage}[t]{0.24\textwidth}
\centering
\includegraphics[width=\linewidth]{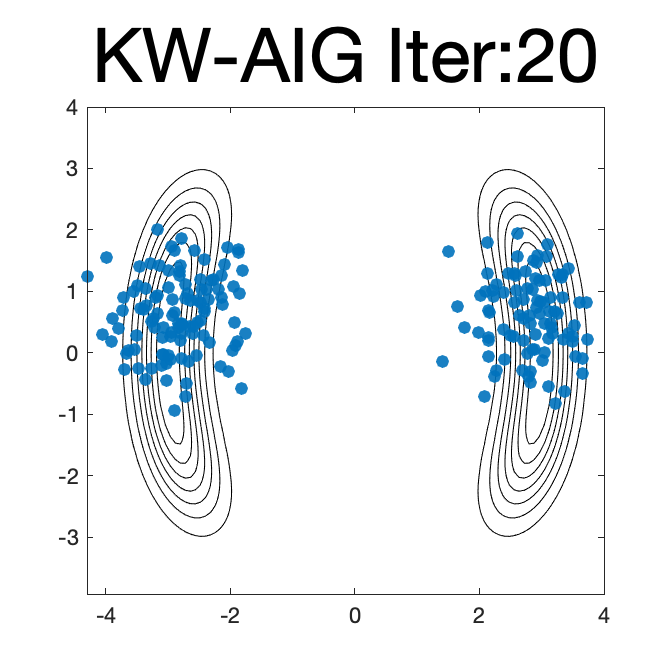}
\end{minipage}
\begin{minipage}[t]{0.24\textwidth}
\centering
\includegraphics[width=\linewidth]{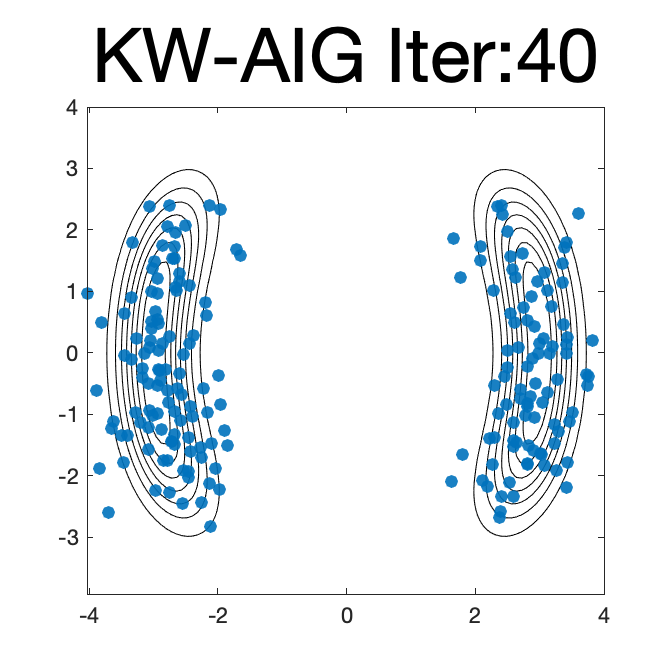}
\end{minipage}
\begin{minipage}[t]{0.24\textwidth}
\centering
\includegraphics[width=\linewidth]{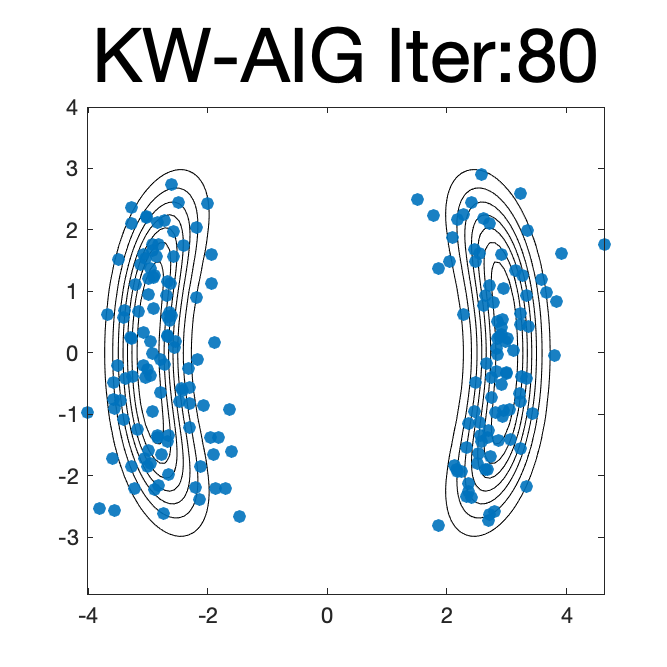}
\end{minipage}
\begin{minipage}[t]{0.24\textwidth}
\centering
\includegraphics[width=\linewidth]{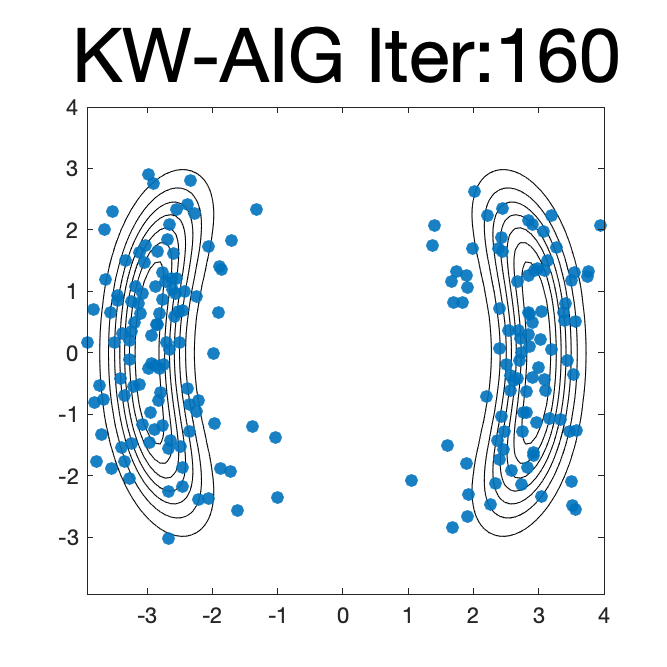}
\end{minipage}
\begin{minipage}[t]{0.24\textwidth}
\centering
\includegraphics[width=\linewidth]{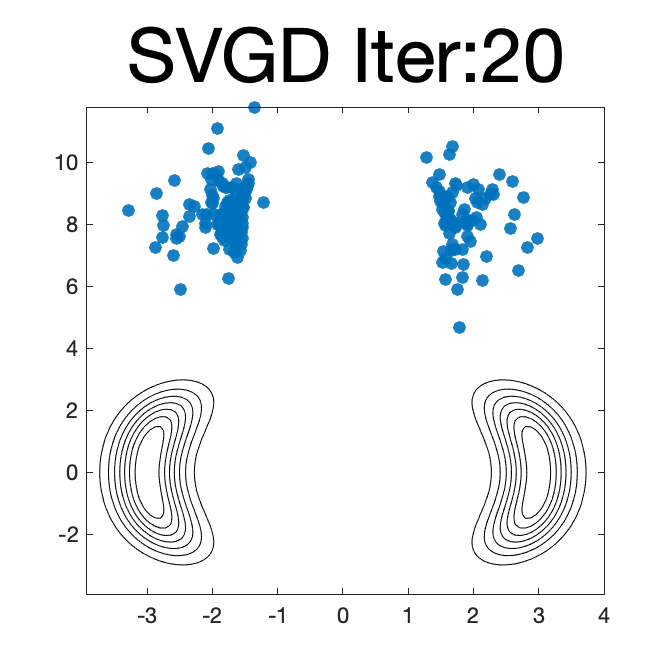}
\end{minipage}
\begin{minipage}[t]{0.24\textwidth}
\centering
\includegraphics[width=\linewidth]{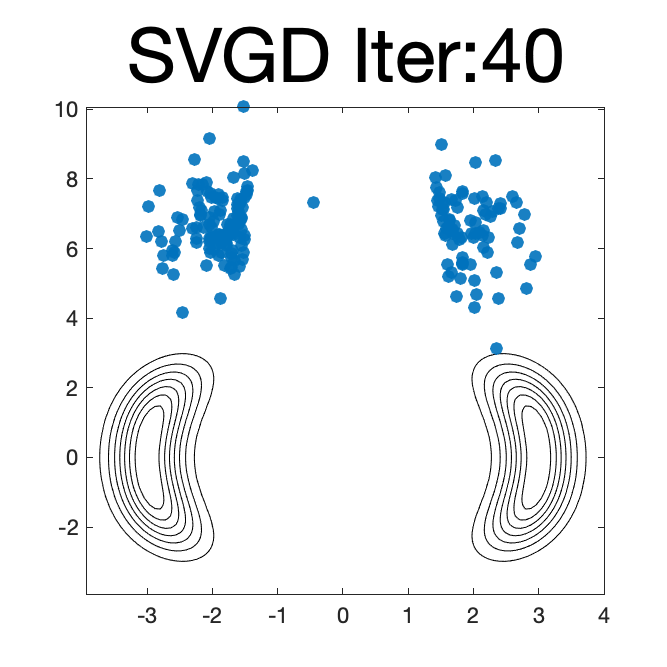}
\end{minipage}
\begin{minipage}[t]{0.24\textwidth}
\centering
\includegraphics[width=\linewidth]{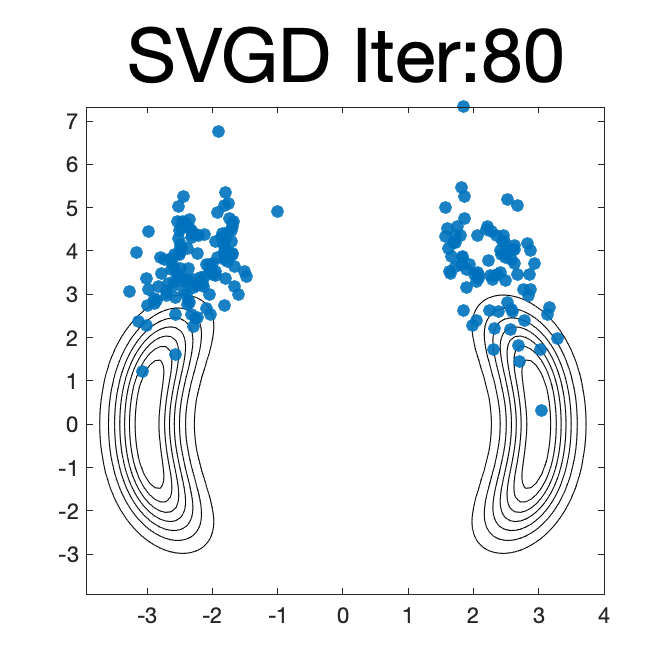}
\end{minipage}
\begin{minipage}[t]{0.24\textwidth}
\centering
\includegraphics[width=\linewidth]{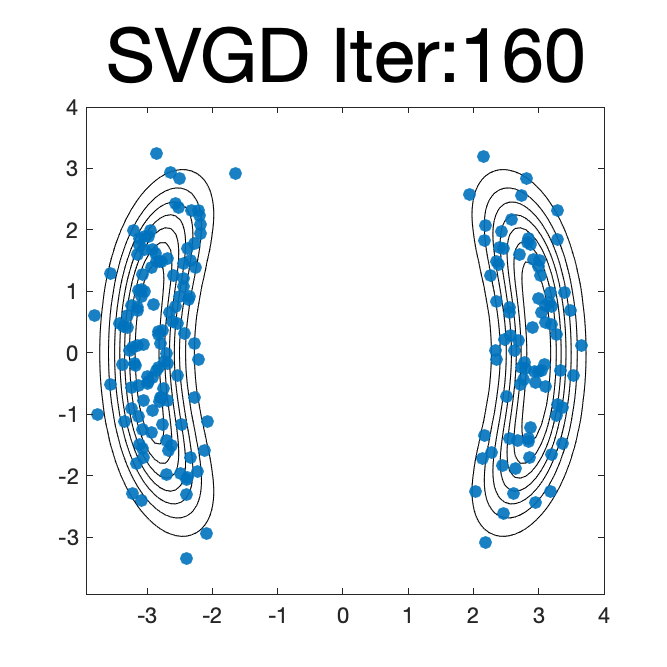}
\end{minipage}
\begin{minipage}[t]{0.24\textwidth}
\centering
\includegraphics[width=\linewidth]{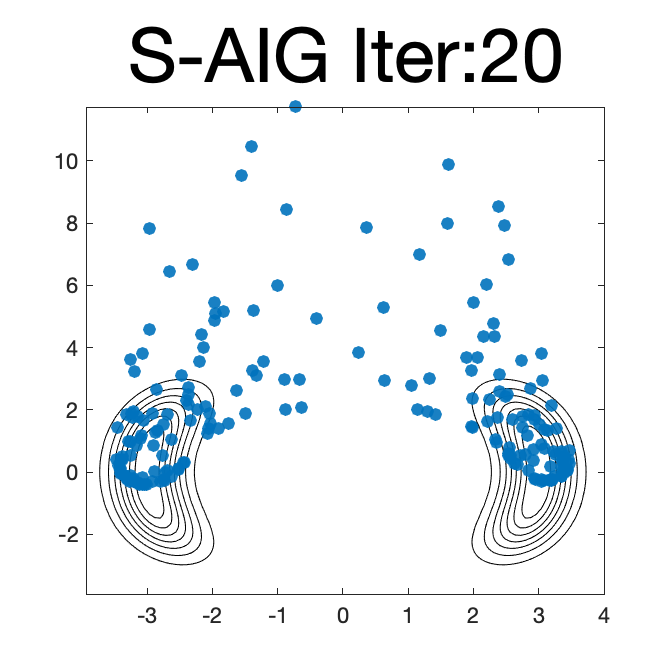}
\end{minipage}
\begin{minipage}[t]{0.24\textwidth}
\centering
\includegraphics[width=\linewidth]{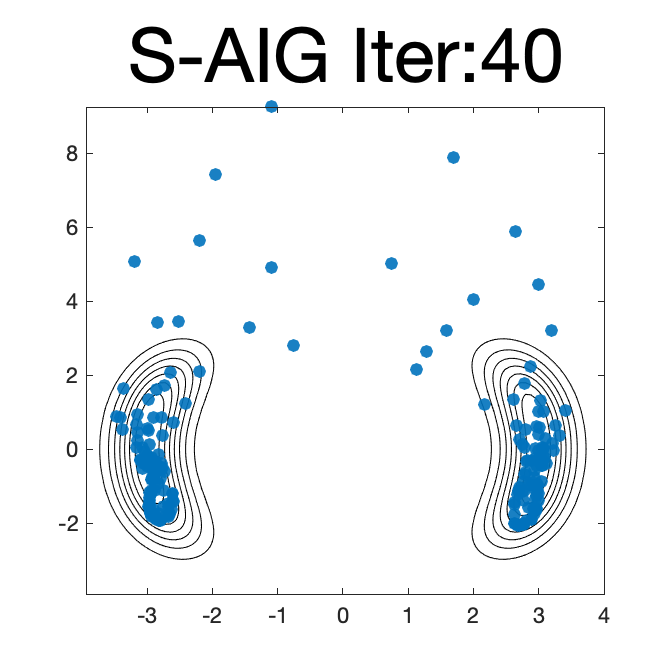}
\end{minipage}
\begin{minipage}[t]{0.24\textwidth}
\centering
\includegraphics[width=\linewidth]{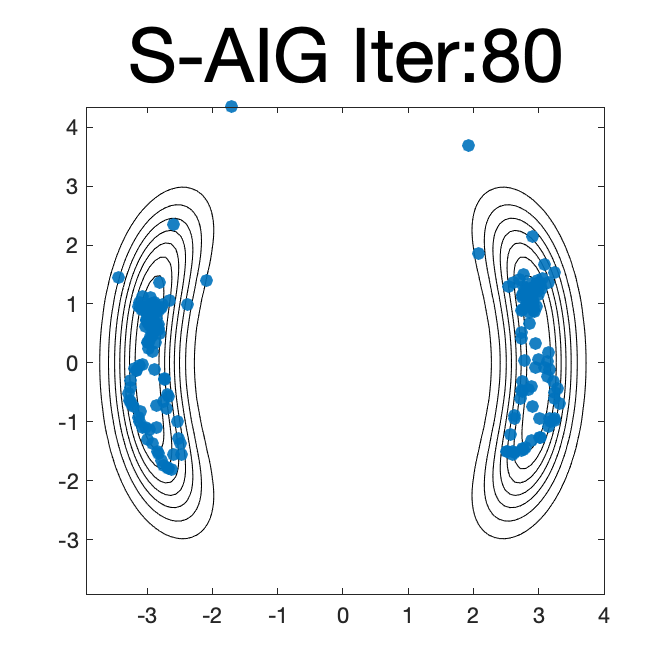}
\end{minipage}
\begin{minipage}[t]{0.24\textwidth}
\centering
\includegraphics[width=\linewidth]{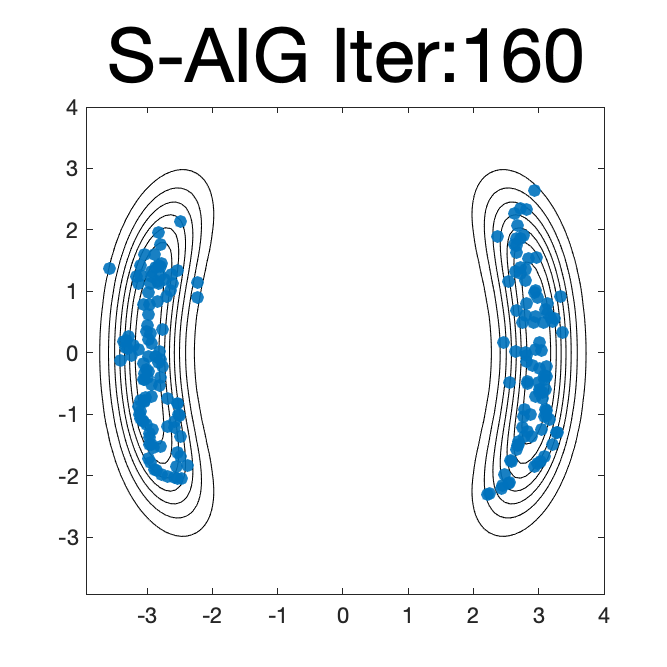}
\end{minipage}
\caption{Comparison of different AIG flows on a toy example. }\label{fig:toy}
\end{figure}

\subsection{Effect of BM method}
\label{ssec:te}
We first investigate the validity of the BM method in selecting the bandwidth. The target density $\rho^*$ is a toy bi-modal distribution \citep{viwnf}. We compare two types of particle implementations of the Wasserstein gradient flow over KL divergence:
\begin{equation*}
\begin{aligned}
&X_{k+1}^i = X_k^i-\tau\nabla f(X_k^i)+\sqrt{2\tau}B_{k}^i,\\
&X_{k+1}^i = X_k^i-\tau(\nabla f(X_k^i)+\xi_k(X_k^i)).
\end{aligned}
\end{equation*}
Here $B_k^i\sim\mcN(0,1)$ is the standard Brownian motion and $\xi_k$ is estimated via KDE. The first method is known as the Langevin MCMC method and the second method is called the ParVI method. For ParVI methods, the bandwidth $h$ is selected by MED/HE/BM respectively. The initial distribution of the particle system follows the standard Gaussian $\mcN(0,I)$. The objective density function follows
\begin{equation*}
\begin{aligned}
\rho^*(x)\propto& \exp(-2(\|x\|-3)^2)\\
 &\times (\exp(-2(x_1-3)^2)+\exp(-2(x_1+3)^2)).
\end{aligned}
\end{equation*}
All methods run for $200$ iterations using the same fixed step size $\tau = 0.1$.

Figure \ref{fig:toy} shows the distribution of $200$ samples based on different methods. Samples from MCMC match the target distribution in a stochastic way; samples from MED collapse; samples from HE align tidily around contour lines; samples from BM arrange neatly and are closer to samples from MCMC. This indicates that the BM method makes the particle system behave similar to MCMC, though in a deterministic way.  
\begin{figure}[!htbp]
\centering
\begin{minipage}[t]{0.24\textwidth}
\centering
\includegraphics[width=\linewidth]{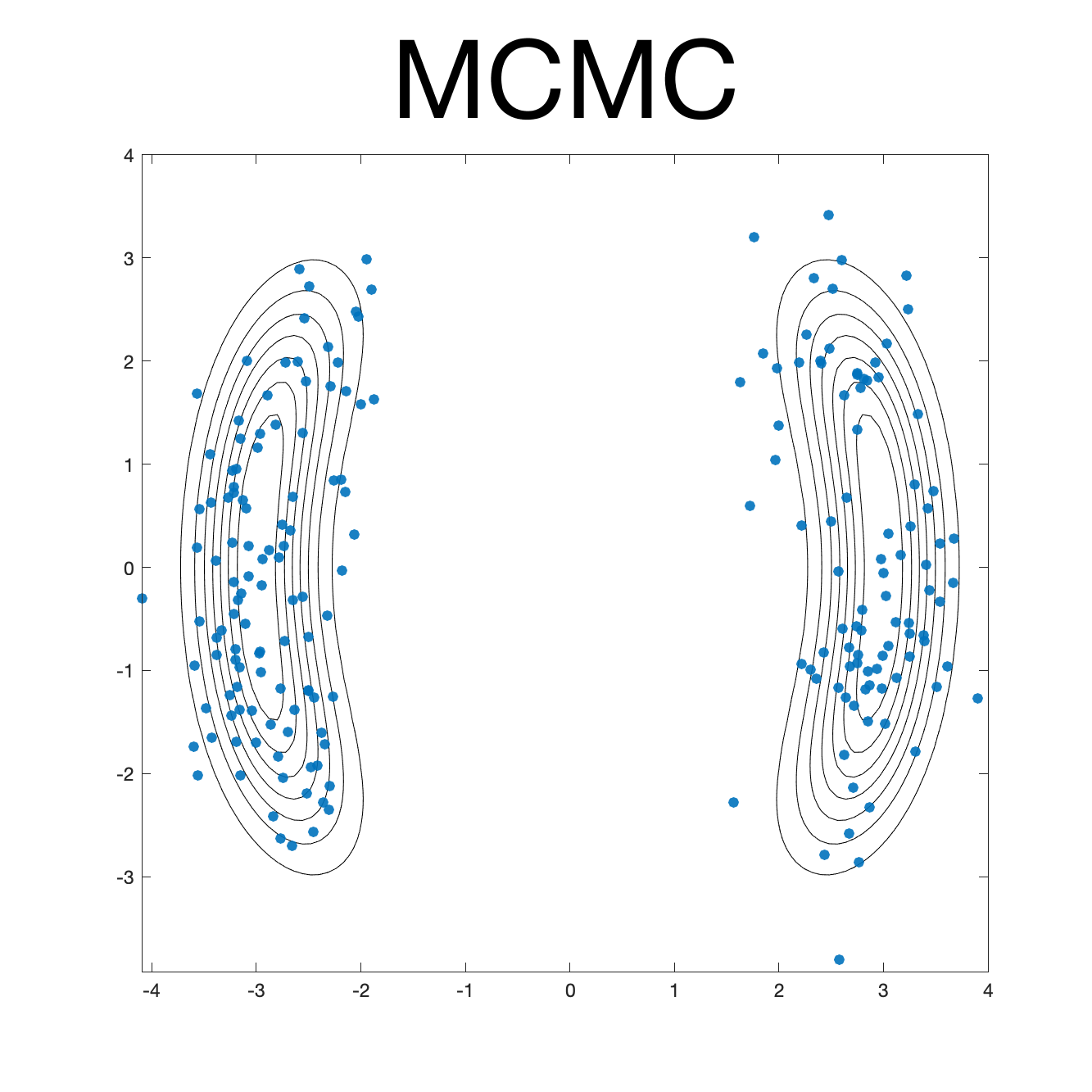}
\end{minipage}
\begin{minipage}[t]{0.24\textwidth}
\centering
\includegraphics[width=\linewidth]{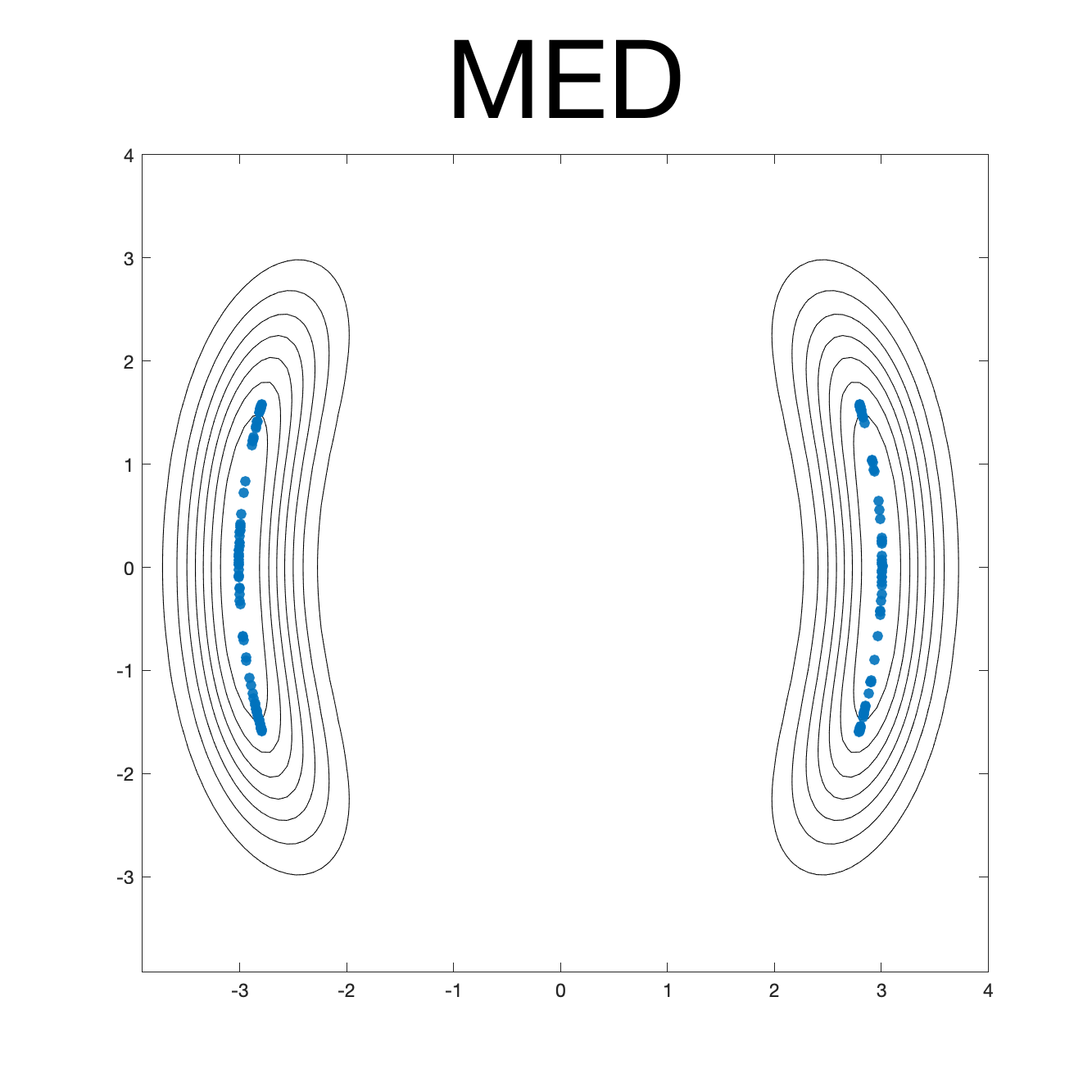}
\end{minipage}
\begin{minipage}[t]{0.24\textwidth}
\centering
\includegraphics[width=\linewidth]{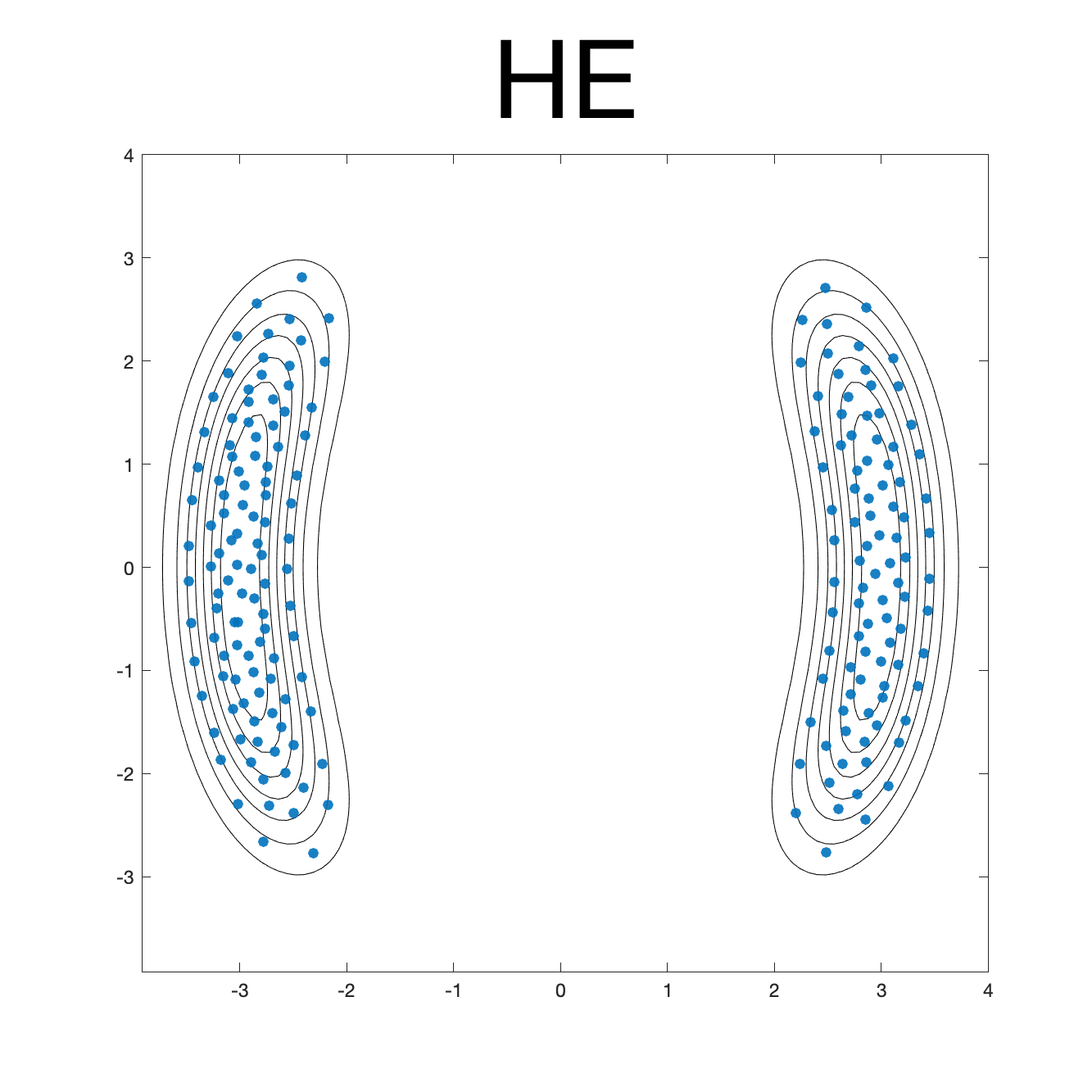}
\end{minipage}
\begin{minipage}[t]{0.24\textwidth}
\centering
\includegraphics[width=\linewidth]{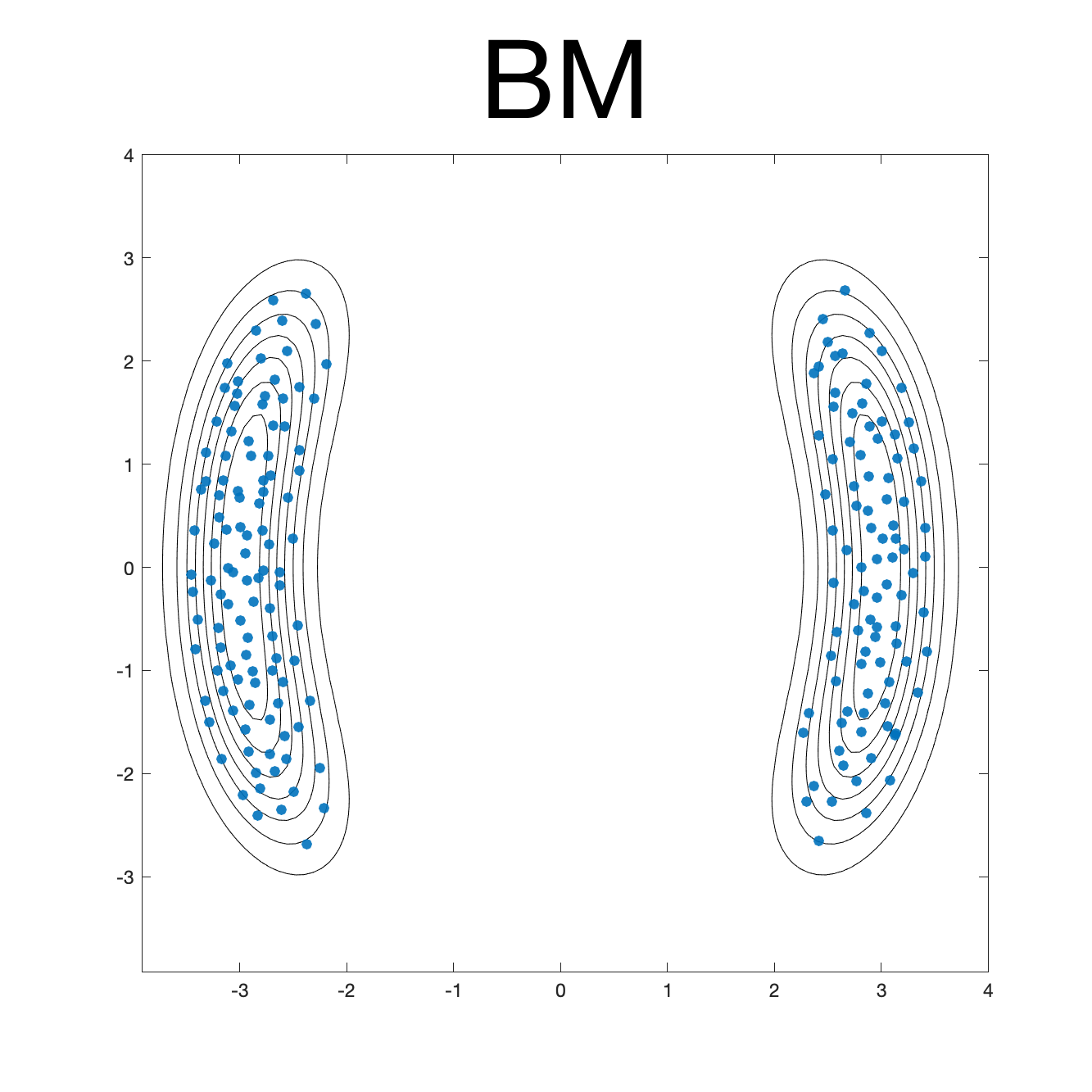}
\end{minipage}
\caption{The effect of the BM method. Samples are plotted as blue dots. Left to right: MCMC, MED, HE and BM. All methods are run for $200$ iterations with the same initialization.}\label{fig:toy_band}
\end{figure}

\subsection{Bayesian logistic regression}
\label{ssec:blr}
We perform the standard Bayesian logistic regression experiment on the Covertype dataset, following the same settings as \cite{SVGD}. Our methods are compared with MCMC, SVGD \cite{SVGD}, WNAG \cite{afomo} and WNes \cite{uaapb}. SVGD is a gradient descent method based on the Stein metric, which approximates W-GF, see \cite[Theorem 2]{uaapb}. WNAG and WNes are two accelerated methods based on W-GF.

We select the kernel bandwidth using either the MED method or the proposed BM method. Figure \ref{fig:blr} indicates that the BM method accelerates and stabilizes the performance of GFs and AIGs. The performance of MCMC and WGF are similar and they achieve the best log-likelihood. For a given metric, AIG flows have better test accuracy and test log-likelihood in first 2000 iterations. W-AIG and KW-AIG achieve $75\%$ test accuracy in less than 500 iterations.

\begin{figure}[!htbp]
\centering
\begin{minipage}[t]{0.48\textwidth}
\centering
\includegraphics[width=\linewidth]{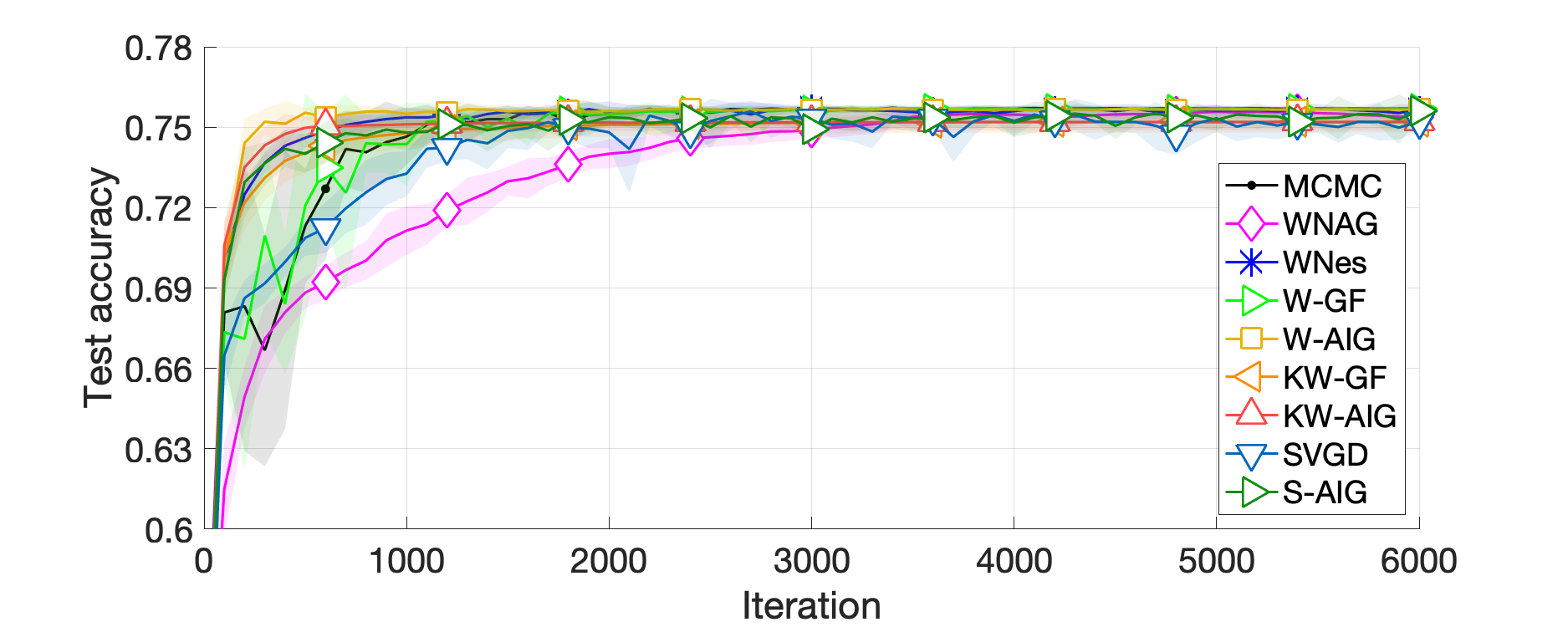}
\end{minipage}
\begin{minipage}[t]{0.48\textwidth}
\centering
\includegraphics[width=\linewidth]{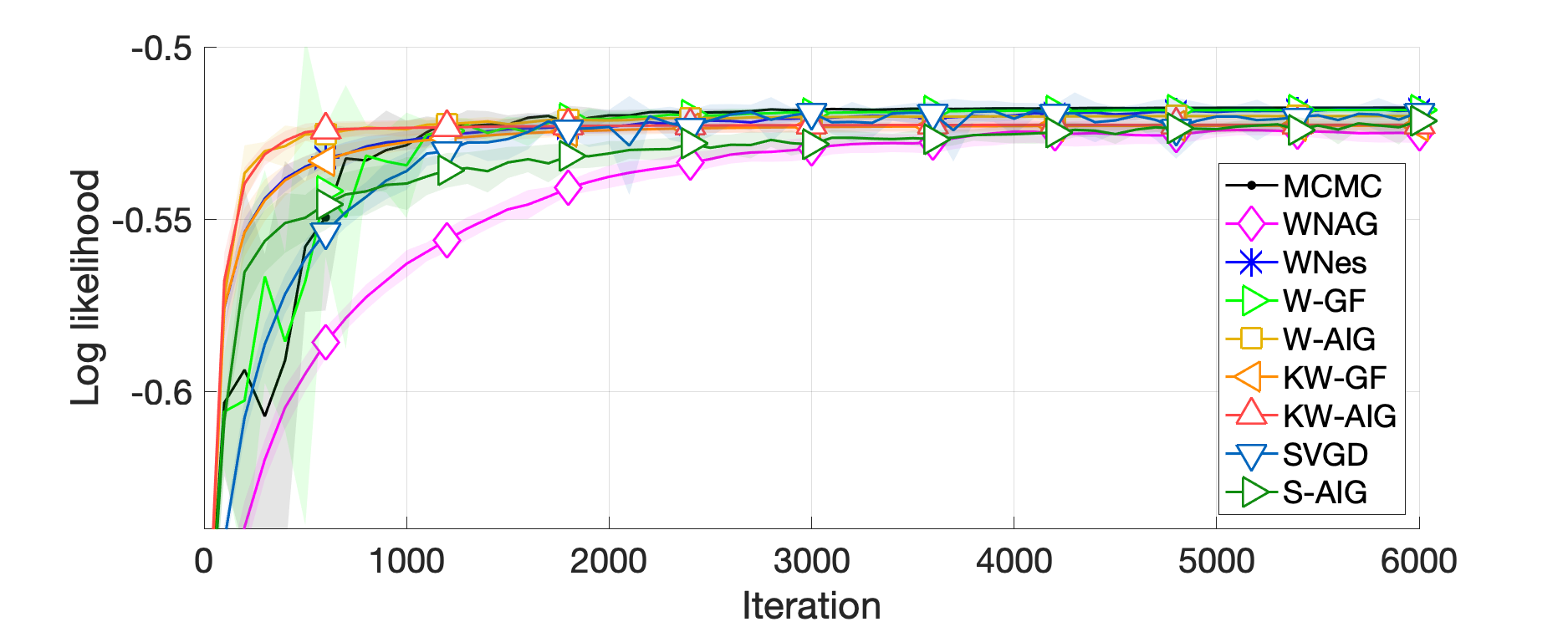}
\end{minipage}
\hfill
\centering
\begin{minipage}[t]{0.48\textwidth}
\centering
\includegraphics[width=\linewidth]{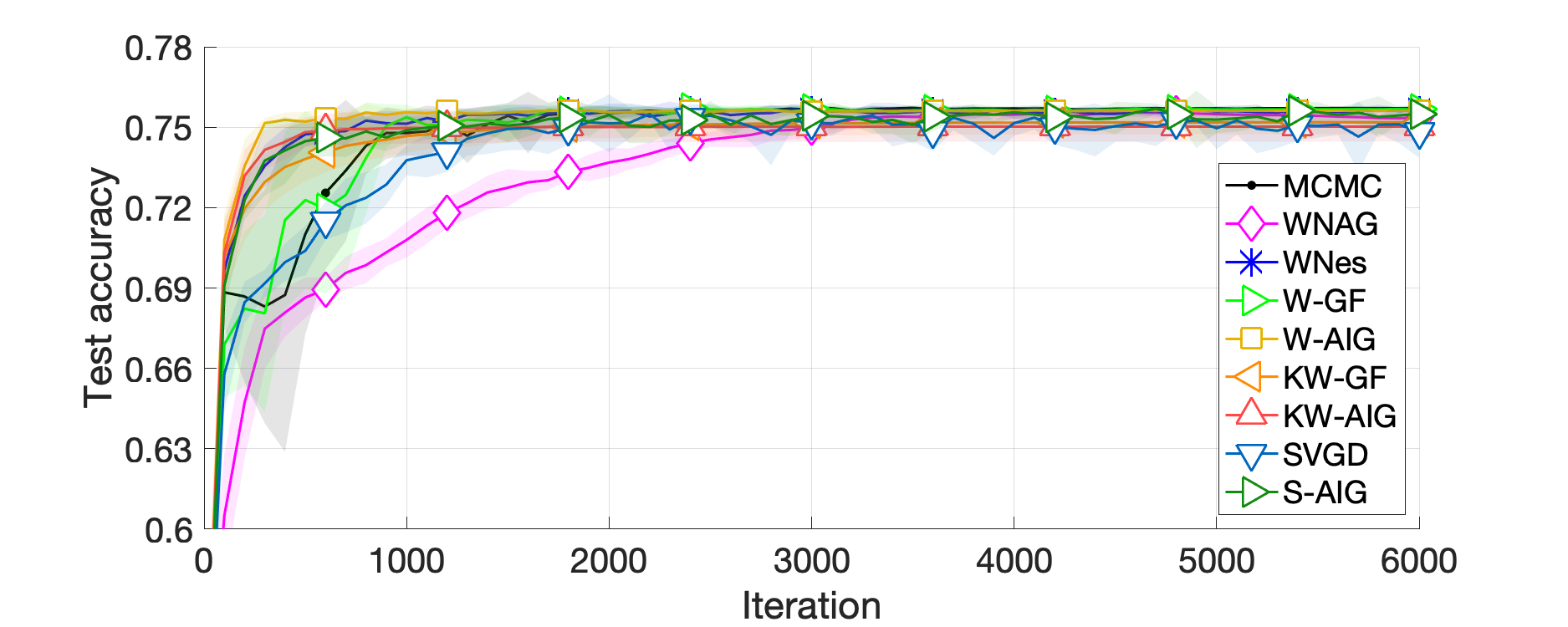}
\end{minipage}
\begin{minipage}[t]{0.48\textwidth}
\centering
\includegraphics[width=\linewidth]{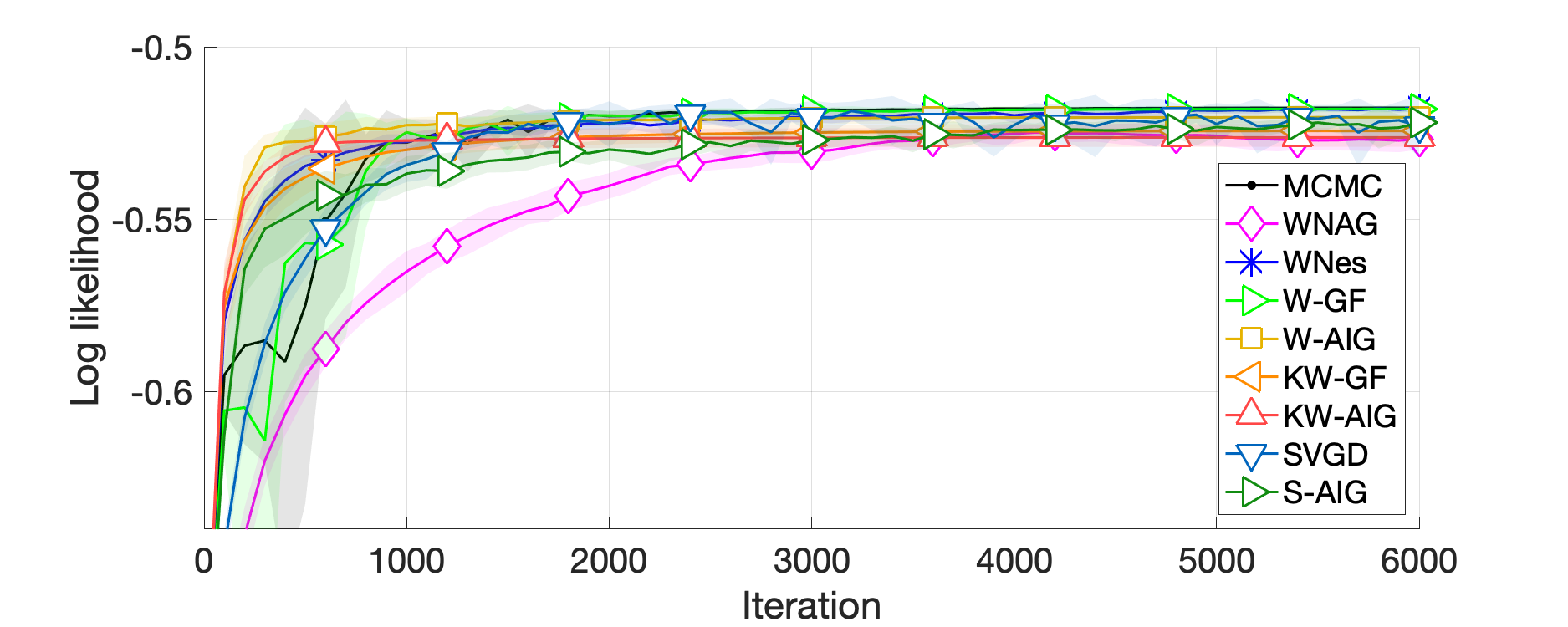}
\end{minipage}
\caption{Results on Bayesian logistic regression, averaged over $10$ independent trials. The shaded areas show the variance. Top: BM; Bottom: MED. Left: Test accuracy; Right: Test log-likelihood.}\label{fig:blr}
\end{figure}

\subsection{Bayesian neural network}
We apply our proposed method on Bayesian neural network over the UCI datasets\footnote{https://archive.ics.uci.edu/ml/datasets.php}, with the same setting as \cite{svgd_mat}. We compare W-AIG, W-GF and SVGD. For all methods, we use $N=10$ particles. The averaged results over $20$ independent trials are collected in Table \ref{tab:bnn1} and \ref{tab:bnn2}. We observe that on most datasets, W-AIG has better test root-mean-square-error and test log-likelihood than W-GF and SVGD. This indicates that W-AIG may have better generalization than W-GF and SVGD.

\begin{table}[ht]
    \centering
    \small{
    \setlength{\tabcolsep}{0.1mm}{
    \begin{tabular}{|c|c|c|c|}
    \hline
         Dataset&AIG&WGF&SVGD \\\hline
         Boston&$2.871_{\pm3.41e-3}$&$3.077_{\pm5.52e-3}$&{$\mathbf{2.775_{\pm3.78e-3}}$}\\\hline
         Combined&$\mathbf{4.067_{\pm 9.27e-1}}$&$4.077_{\pm 3.85e-4}$&$4.070_{\pm 2.02e-4}$\\\hline
         Concrete&$\mathbf{4.440_{\pm 1.34e-1}}$&$4.883_{\pm 1.93e-1}$&$4.888_{\pm 1.39e-1}$\\\hline
         Kin8nm&$\mathbf{0.094_{\pm 5.56e-6}}$&$0.096_{\pm 3.36e-5}$&$0.095_{\pm 1.32e-5}$\\\hline
         Wine&$0.606_{\pm 1.40e-5}$&$0.614_{\pm 3.48e-4}$&$\mathbf{0.604_{\pm 9.89e-5}}$\\\hline
         Year&$8.876_{\pm 3.71e-4}$&$\mathbf{8.872_{\pm 2.81e-4}}$&$8.873_{\pm 7.19e-4}$\\\hline
    \end{tabular}
    }}
    \caption{Test root-mean-square-error (RMSE).}
    \label{tab:bnn1}
\end{table}

\begin{table}[ht]
    \centering
    \small{
    \setlength{\tabcolsep}{0.1mm}{
    \begin{tabular}{|c|c|c|c|}
    \hline
         Dataset&AIG&WGF&SVGD \\\hline
         Boston&$\mathbf{-2.609_{\pm 1.34e-4}}$&$-2.694_{\pm 2.83e-4}$&$-2.611_{\pm 1.36e-4}$\\\hline
         Combined&$\mathbf{-2.822_{\pm 5.72e-3}}$&$ -2.825_{\pm 2.36e-5}$&$-2.823_{\pm 1.24e-5}$\\\hline
         Concrete&$\mathbf{-2.884_{\pm 8.84e-3}}$&$-2.971_{\pm 8.93e-3}$&$-2.978_{\pm 6.05e-3}$\\\hline
         Kin8nm&$\mathbf{0.951_{\pm 6.43e-4}}$&$0.923_{\pm 3.37e-3}$&$0.932_{\pm 1.43e-3}$\\\hline
         Wine&$-0.961_{\pm 1.28e-4}$&$-0.961_{\pm 3.17e-4}$&$\mathbf{-0.952_{\pm 9.89e-5}}$\\\hline
         Year&$-3.654_{\pm 1.00e-5}$&$-3.655_{\pm 7.82e-6}$&$\mathbf{-3.652_{\pm 1.28e-5}}$\\\hline
    \end{tabular}
    }}
    \caption{Test log-likelihood.}
    \label{tab:bnn2}
\end{table}

\section{Conclusion}
In summary, we propose the framework of AIG flows by damping Hamiltonian flows with respect to certain information metrics in probability space.
In theory, we establish the convergence rate of F-AIG and W-AIG flows. In algorithm, we propose particle formulations for W-AIG flow, KW-AIG and S-AIG flows. Numerically, we propose discrete-time algorithms and an adaptive restart technique to overcome numerical stiffness of AIG flows. 
To efficiently approximate $\nabla \log \rho_k(x)$, we introduce a novel kernel selection method by learning from Brownian-motion samples. Numerical experiments verify the acceleration effect of AIG flows and the strength of adaptive restart. 

In future works, we intend to systematically explain the stiffness of AIG flows and effects of adaptive restart. We shall apply our results to general information metrics, especially for generalized Wasserstein metrics. We expect to study the related sampling efficient optimization methods and discrete-time algorithms. We also plan to incorporate Hessian operators in probability space \cite{infso} in designing higher-order accelerated algorithms. We shall compare these information metrics induced methods in terms of both computational complexity and sampling efficiency. We expect that the proposed accelerated algorithms will be useful in scientific computing of Bayesian inverse problems. 

\bibliography{WIG}
\bibliographystyle{apalike}

\appendix

In this appendix, we formulate detailed derivations of examples and proofs of propositions. We also design particle implementations of KW-AIG flows, S-AIG flows and provide detailed implementations of experiments. 
\section{Euler-Lagrange equation, Hamiltonian flows and AIG flows}
In this section, we review and derive Euler-Lagrange equation, 
Hamiltonian flows and Euler-Lagrange formulation of AIG flows in probability space.
\subsection{Derivation of the Euler-Lagrange equation}
In this subsection, we derive the Euler-Lagrange equation in probability space. For a given metric $g_\rho$ in probability space, we can define a Lagrangian by
$$\mcL(\rho_t,\p_t\rho_t)=\frac{1}{2}g_{\rho_t}(\p_t\rho_t,\p_t\rho_t)-E(\rho_t).$$
\begin{proposition}
The Euler-Lagrange equation for this Lagrangian follows
\begin{equation*}
\p_t\lp\frac{\delta \mcL}{\delta (\p_t \rho_t)}\rp=\frac{\delta \mcL}{\delta \rho_t}+C(t),
\end{equation*}
where $C(t)$ is a spatially-constant function. 
\end{proposition}
\begin{proof}
For a fixed $T>0$ and two given densities $\rho_0,\rho_T$, consider the variational problem
{\small
$$
I(\rho_t)=\inf\limits_{\rho_t}\left\{\left.\int_0^T\mcL(\rho_t,\p_t\rho_t)dt\right |\rho_t|_{t=0}=\rho_0,\rho_t|_{t=T}=\rho_T\right\}.
$$
}

Let $h_t\in\mcF(\Omega)$ be the smooth perturbation function that satisfies  $\int   h_t dx=0, t\in\bb{0,T}$ and $h_t|_{t=0}=h_t|_{t=T}\equiv 0$. Denote $\rho_t^\epsilon = \rho_t+\epsilon h_t$. 
Note that we have the Taylor expansion
$$
\begin{aligned}
I(\rho_t^\epsilon)=&\int_0^T \mcL(\rho_t,\p_t\rho_t)dt\\
&+\epsilon \int_0^T\int \lp\frac{\delta \mcL}{\delta\rho_t}h_t+\frac{\delta \mcL}{\delta (\p_t \rho_t)}\p_th_t\rp dxdt+o(\epsilon).
\end{aligned}
$$
From $\left.\frac{dI(\rho_t^\epsilon)}{d\epsilon}\right|_{\epsilon=0}=0$, it follows that
$$
\int_0^T\int \lp\frac{\delta \mcL}{\delta\rho_t}h_t+\frac{\delta \mcL}{\delta (\p_t \rho_t)}\p_th_t\rp dxdt=0.
$$
Note that $h_t|_{t=0}=h_t|_{t=T}\equiv 0$. Perform integration by parts w.r.t. $t$ yields
$$
\int_0^T\int \lp\frac{\delta \mcL}{\delta\rho_t}-\p_t\frac{\delta \mcL}{\delta (\p_t \rho_t)}\rp h_t dxdt=0.
$$
Because $\int  h_t dx=0$, the Euler-Lagrange equation holds with a spatially constant function $C(t)$. 
\end{proof}

\subsection{Derivation of Hamiltonian flow}\label{app:hamilton}
In this subsection, we derive the Hamiltonian flow in the probability space. Denote $\Phi_t = \delta L/\delta (\p_t \rho_t)=G(\rho_t)\p_t\rho_t$. Then, the Euler-Lagrange equation can be formulated as a system of $(\rho_t,\Phi_t)$, i.e.,
\begin{equation*}
\left\{
\begin{aligned}
&\p_t\rho_t-G(\rho_t)^{-1}\Phi_t=0,\\
&\p_t\Phi_t+\frac{1}{2}\frac{\delta}{\delta \rho_t}\lp \int  \Phi_tG(\rho_t)^{-1}\Phi_tdx\rp+\frac{\delta E}{\delta \rho_t}=0.
\end{aligned}
\right.
\end{equation*}

First, we give a useful identity. Given a metric tensor $G(\rho):T_\rho\mcP(\Omega)\to T_\rho^*\mcP(\Omega)$, we have
\begin{equation}\label{equ:id1}
\begin{aligned}
&\int   \sigma_1 G(\rho)\sigma_2 dx= \int   G(\rho)\sigma_1\sigma_2dx\\
= &\int   \Phi_1G(\rho)^{-1}\Phi_2 dx= \int   G(\rho)^{-1}\Phi_1\Phi_2 dx.
\end{aligned}
\end{equation}
Here $\Phi_1 = G(\rho)^{-1}\sigma_1$ and $\Phi_2 = G(\rho)^{-1}\sigma_2$. We then check that 
{\small
\begin{equation}\label{equ:check1}
\frac{\delta}{\delta \rho_t}\lp\int  \p_t\rho_t G(\rho_t) \p_t\rho_tdx\rp =- \frac{\delta}{\delta \rho_t}\lp \int  \Phi_tG(\rho_t)^{-1}\Phi_tdx\rp.
\end{equation}}
Let $\tilde\rho_t = \rho_t+\epsilon h$, where $h\in T_{\rho_t}\mcP(\Omega)$. For all $\sigma\in T_{\rho_t}\mcP$, it follows
\begin{equation*}
G(\rho_t+\epsilon h )^{-1}G(\rho_t+\epsilon h ) \sigma = \sigma.
\end{equation*}
The first-order derivative w.r.t. $\epsilon$ of the left hand side shall be $0$, i.e.,
\begin{equation*}
\lp\frac{\p G(\rho_t)^{-1}}{\p \rho_t} \cdot h\rp G(\rho_t) \sigma +G(\rho_t) ^{-1}\lp\frac{\p G(\rho_t)}{\p \rho_t}\cdot h\rp \sigma = 0.
\end{equation*}
Because $\p_t\rho_t=G(\rho)^{-1}\Phi_t$, applying \eqref{equ:id1} yields
\begin{equation}\label{equ:a1_1}
\begin{aligned}
&\int  \p_t \rho_t \lp\frac{\p G(\rho_t)}{\p \rho_t}\cdot h\rp \p_t \rho_tdx=\int  \Phi_t G(\rho_t)^{-1}\lp\frac{\p G(\rho_t)}{\p \rho_t}\cdot h\rp \p_t \rho_tdx\\
=& -\int  \Phi_t \lp\frac{\p G(\rho_t)^{-1}}{\p \rho_t}\cdot h\rp G(\rho_t)\p_t \rho_tdx=-\int  \Phi_t\lp\frac{\p G(\rho_t)^{-1}}{\p \rho_t}\cdot h\rp\Phi_tdx.
\end{aligned}
\end{equation}
Based on basic calculations, we can compute that
\begin{equation}\label{equ:a1_2}
\begin{aligned}
&\int  \p_t\rho_t G(\tilde\rho_t) \p_t\rho_tdx-\int  \p_t\rho_t G(\rho_t)\p_t\rho_tdx=\epsilon\int  \p_t \rho_t \lp\frac{\p G(\rho_t)}{\p \rho_t}\cdot h\rp \p_t \rho_tdx+o(\epsilon),
\end{aligned}
\end{equation}
\begin{equation}\label{equ:a1_3}
\begin{aligned}
&-\int  \Phi_tG(\tilde\rho_t)^{-1}\Phi_tdx+\int  \Phi_tG(\rho_t)^{-1}\Phi_tdx=-\epsilon \int  \Phi_t\lp\frac{\p G(\rho_t)^{-1}}{\p \rho_t}\cdot h\rp\Phi_tdx+o(\epsilon).
\end{aligned}
\end{equation}
Combining \eqref{equ:a1_1}, \eqref{equ:a1_2} and \eqref{equ:a1_3} yields \eqref{equ:check1}.
Hence, the Euler-Lagrange equation is equivalent to
\begin{equation*}
\begin{aligned}
&\p_t\Phi_t=\frac{1}{2}\frac{\delta}{\delta \rho_t}\lp\int  \p_t\rho_t G(\rho_t) \p_t\rho_tdx\rp -\frac{\delta E}{\delta \rho_t}=-\frac{1}{2}\frac{\delta}{\delta \rho_t}\lp \int  \Phi_tG(\rho_t)^{-1}\Phi_tdx\rp-\frac{\delta E}{\delta \rho_t}.
\end{aligned}
\end{equation*}
This equation combining with $\p_t\rho_t=G(\rho)^{-1}\Phi_t$ recovers the Hamiltonian flow. In short, the Euler-Lagrange equation is from the primal coordinates $(\rho_t,\p_t\rho_t)$ and the Hamiltonian flow is from the dual coordinates $(\rho_t,\Phi_t)$. Similar interpretations can be found in \citep{whf}. 

\subsection{The Euler-Lagrangian formulation of AIG flows}
We can formulate the AIG flow as a second-order equation of $\rho_t$,
\begin{equation*}
    \frac{D^2}{D t^2}\rho_t+\alpha_t \p_t\rho_t+ G(\rho_t)^{-1}\frac{\delta E}{\delta \rho_t}=0.
\end{equation*}
Here $D^2/D t^2$ is the covariant derivative w.r.t. metric $G(\rho)$. We can also explicitly write $\frac{D^2}{D t^2}\rho_t$ as
\begin{equation*}
\begin{aligned}
\frac{D^2}{D t^2}\rho_t = &\p_{tt}\rho_t-(\p_t G(\rho_t)^{-1})\p_t\rho_t+\frac{1}{2}G(\rho_t)^{-1}\frac{\delta}{\delta \rho_t}\lp \int  \partial_t\rho_t G(\rho_t)\partial_t\rho_t dx\rp.
\end{aligned}
\end{equation*}

\section{Derivation of examples in Section 3}
In this section, we present examples of gradient flows, Hamiltonian flows and derive particle dynamics examples in Section 3. 
\subsection{Examples of gradient flows}
We first present several examples of gradient flows w.r.t. different metrics. 
\begin{example}[Fisher-Rao gradient flow]
\begin{equation*}
\begin{split}
\p_t\rho_t=&-G^F(\rho_t)^{-1}\frac{\delta E}{\delta\rho_t}=-\rho_t\lp \frac{\delta E}{\delta \rho_t}-\int    \frac{\delta E}{\delta \rho_t}\rho_tdy\rp.
\end{split}
\end{equation*}
\end{example}

\begin{example}[Wasserstein gradient flow]
\begin{equation*}
\begin{split}
\p_t\rho_t =&-G^W(\rho_t)^{-1}\frac{\delta E}{\delta\rho_t}=\nabla \cdot \lp\rho_t\nabla \frac{\delta E}{\delta \rho_t}\rp.
\end{split}
\end{equation*}
\end{example}

\begin{example}[Kalman-Wasserstein gradient flow]
\begin{equation*}
\begin{split}
\p_t\rho_t =&-G^{KW}(\rho_t)^{-1}\frac{\delta E}{\delta\rho_t}=\nabla\cdot\pp{\rho_t C^\lambda(\rho_t)\nabla\pp{\frac{\delta E}{\delta\rho_t}}}.
\end{split}
\end{equation*}
\end{example}

\begin{example}[Stein gradient flow]
\begin{equation*}
\begin{split}
\p_t\rho_t =&-G^S(\rho_t)^{-1}\frac{\delta E}{\delta\rho_t}=\nabla_x\cdot\pp{\rho_t(x) \int k(x,y) \rho_t(y) \nabla_y\pp{\frac{\delta E}{\delta\rho_t}} dy}.
\end{split}
\end{equation*}
\end{example}

\subsection{Examples of Hamiltonian flows}
We next present several examples of Hamiltonian flows w.r.t. different metrics. The derivations simply follow from the definition of the given information metric and the formulations given in Appendix \ref{app:hamilton}.
\begin{example}[Fisher-Rao Hamiltonian flow]
The Fisher-Rao Hamiltonian flow follows
\begin{equation*}
\left\{
\begin{aligned}
&\p_t\rho_t-\rho_t\lp \Phi_t-\mbE_{\rho_t}[\Phi_t]\rp=0,\\
&\p_t\Phi_t+\frac{1}{2}\Phi_t^2-\mbE_{\rho_t}[\Phi_t]\Phi_t+\frac{\delta E}{\delta \rho_t}=0,
\end{aligned}
\right.
\end{equation*}
where the corresponding Hamiltonian is 
$$
\mcH^F(\rho_t,\Phi_t)=\frac{1}{2}\lp \mbE_{\rho_t}[\Phi_t^2]-\pp{\mbE_{\rho_t}[\Phi_t]}^2\rp+E(\rho_t). 
$$
The derivation comes from that
\begin{equation*}
\begin{aligned}
    &\frac{\delta }{\delta \rho_t} \int \Phi_t G^F(\rho_t) \Phi_t dx\\
    =& \frac{\delta }{\delta \rho_t}\lp \mbE_{\rho_t}[\Phi_t^2]-\pp{\mbE_{\rho_t}[\Phi_t]}^2\rp\\
    =&\Phi_t^2-2\mbE_{\rho_t}[\Phi_t]\Phi_t.
\end{aligned}
\end{equation*}
\end{example}

\begin{example}[Wasserstein Hamiltonian flow]
The Wasserstein Hamiltonian flow writes
\begin{equation*}
\left\{
\begin{aligned}
&\p_t\rho_t+\nabla\cdot(\rho_t\nabla \Phi_t)=0,\\
&\p_t\Phi_t+\frac{1}{2}\|\nabla \Phi_t\|^2+\frac{\delta E}{\delta \rho_t}=0,
\end{aligned}
\right.
\end{equation*}
where the corresponding Hamiltonian is 
$$
\mcH^W(\rho_t,\Phi_t)=\frac{1}{2}\int   \|\nabla \Phi_t\|^2\rho_tdx+E(\rho_t). $$
It is identical to the Wasserstein Hamiltonian flow introduced by \cite{whf}. The derivation simply comes from that
\begin{equation*}
\begin{aligned}
    \frac{\delta }{\delta \rho_t} \int \Phi_t G^W(\rho_t) \Phi_t dx= \frac{\delta }{\delta \rho_t}\lp \int \|\nabla \Phi_t\|_2^2\rho_t dx\rp=\|\nabla \Phi_t\|^2.
\end{aligned}
\end{equation*}
\end{example}

\begin{example}[Kalman-Wasserstein Hamiltonian flow]
The Kalman-Wasserstein Hamiltonian flow writes
\begin{equation*}
\left\{\begin{aligned}
&\p_t\rho_t+\nabla \cdot(\rho_tC^\lambda(\rho_t)\nabla \Phi_t)=0,\\
&\p_t \Phi_t+ \frac{1}{2}\lp(x-m(\rho_t))^TB_{\rho_t}(\Phi_t)(x-m(\rho_t))+\nabla\Phi_t(x)^TC^\lambda(\rho_t)\nabla \Phi_t(x)\rp+\frac{\delta E}{\delta \rho_t}=0,\\
\end{aligned}\right.
\end{equation*}
where the corresponding Hamiltonian is 
$$
\mcH^{KW}(\rho_t,\Phi_t)=\frac{1}{2}\int \nabla \Phi_t^T C^\lambda(\rho_t)\nabla \Phi_t \rho_t  dx+E(\rho_t). $$
The derivation comes from that
\begin{equation*}
    \begin{aligned}
    &\frac{\delta }{\delta \rho_t} \int \Phi_t G^{KW}(\rho_t) \Phi_t dx\\
    =&\frac{\delta }{\delta \rho_t}\pp{\int \nabla \Phi_t^T C^\lambda(\rho_t)\nabla \Phi_t \rho_t  dx}\\
    =&x-m(\rho_t))^TB_{\rho_t}(\Phi_t)(x-m(\rho_t))+\nabla\Phi_t(x)^TC^\lambda(\rho_t)\nabla \Phi_t(x).
    \end{aligned}
\end{equation*}
Here we recall that $B_{\rho_t}(\Phi_t)=\int \nabla\Phi_t\nabla \Phi_t^T\rho_t dx$.
\end{example}

\begin{example}[Stein Hamiltonian flow]
The Stein Hamiltonian flow writes
\begin{equation*}
\left\{\begin{aligned}
&\p_t\rho_t(x) = -\nabla_x\cdot\pp{\rho_t(x) \int k(x,y) \rho_t(y) \nabla_y\Phi_t(y) dy},\\
&\p_t\Phi_t(x) = \int   \nabla \Phi_t(x)^T\nabla \Phi_t(y) k(x,y) \rho_t(y)dy -\frac{\delta E}{\delta \rho_t}(x),
\end{aligned}\right.
\end{equation*}
where the corresponding Hamiltonian is 
$$
\mcH(\rho_t,\Phi_t)=\frac{1}{2}\int \int  \nabla \Phi_t(x)^T\nabla \Phi_t(y) k(x,y) \rho_t(x) \rho_t(y)dxdy +E(\rho_t).
$$
The derivation comes from that 
\begin{equation*}
    \begin{aligned}
    &\frac{\delta }{\delta \rho_t} \int \Phi_t G^{S}(\rho_t) \Phi_t dx\\
    =&\frac{\delta }{\delta \rho_t}\pp{\int \int  \nabla \Phi_t(x)^T\nabla \Phi_t(y) k(x,y) \rho_t(x) \rho_t(y)dxdy}\\
    =&2\int   \nabla \Phi_t(x)^T\nabla \Phi_t(y) k(x,y) \rho_t(y)dy.
    \end{aligned}
\end{equation*}
\end{example}

\subsection{The derivation of Example 9 (Wasserstein metric) in Section 3}
We start with an identity. For a twice differentiable $\Phi(x)$, we have
\begin{equation}\label{id1}
\frac{1}{2}\nabla \|\nabla\Phi\|^2 = \nabla^2\Phi\nabla\Phi = (\nabla\Phi\cdot\nabla)\nabla \Phi.
\end{equation}
From (W-AIG), it follows that
\begin{equation}\label{equ:cont}
\p_t\rho_t+\nabla\cdot(\rho_t\nabla\Phi_t) = 0.
\end{equation}
This is the continuity equation of $\rho_t$. Hence, on the particle level, $X_t$ shall follows
$$
dX_t = \nabla \Phi_t(X_t)dt.
$$
Let $V_t=\nabla \Phi_t(X_t)$. Then, by the material derivative in fluid dynamics and (W-AIG), we have
$$
\begin{aligned}
&\frac{dV_t}{dt} = \frac{d}{dt} \nabla \Phi_t(X_t)= (\p_t+\nabla \Phi_t(X_t)\cdot \nabla)\nabla \Phi_t(X_t)dt\\
=&\lp-\alpha_t\nabla \Phi_t(X_t)-\frac{1}{2}\nabla \|\nabla\Phi\|^2-\nabla \frac{\delta E}{\delta \rho_t}\rp dt+(\nabla\Phi\cdot\nabla)\nabla \Phi dt\\
=&-\alpha_t\nabla \Phi_t(X_t)dt-\nabla \frac{\delta E}{\delta \rho_t}(X_t)dt=-\alpha_tV_tdt-\nabla \frac{\delta E}{\delta \rho_t}(X_t)dt.
\end{aligned}
$$

\subsection{The derivations of Example 7 and 10 (Kalman-Wasserstein metric) in Section 3}
We first derive the Hamiltonian flow under the Kalman-Wasserstein metric. We fist show that 
\begin{equation}\label{equ:delta_rho}
    \frac{\delta}{\delta \rho} \left\{\int \Phi G^{KW}(\rho)^{-1}\Phi dx\right\}=(x-m(\rho))^T B_\rho(\Phi) (x-m(\rho))+\nabla\Phi(x)^TC^\lambda(\rho)\nabla \Phi(x).
\end{equation}
From the definition of Kalman-Wasserstein metric, we have
$$
\begin{aligned}
&\int \Phi G^{KW}(\rho)^{-1}\Phi dx=\int \nabla \Phi^T C^\lambda(\rho)\nabla \Phi \rho  dx \\
=&\lra{C^\lambda(\rho), \int \nabla \Phi^T \nabla \Phi \rho  dx}=\lra{C^\lambda(\rho), B_\rho(\Phi)}.
\end{aligned}
$$
Let $\hat \rho=\rho+\epsilon h$, where $h\in T_\rho \mcP(\Omega)$. Then, we can compute that
$$
\begin{aligned}
&\lra{C^\lambda(\rho+\epsilon h), B_{\rho+\epsilon h}(\Phi)}-\lra{C^\lambda(\rho), B_\rho(\Phi)}\\
=&\lra{C^\lambda(\rho+\epsilon h)-C^\lambda(\rho), B_{\rho}(\Phi)}+\lra{C^\lambda(\rho), B_{\rho+\epsilon h}(\Phi)-B_\rho(\Phi)}.
\end{aligned}
$$
We note that
$$
\begin{aligned}
&C^\lambda(\rho+\epsilon h)-C^\lambda(\rho)\\
=&\epsilon \int m(h) (x-m(\rho))^T \rho dx+\epsilon \int (x-m(\rho))m(h)^T \rho dx\\
&+\epsilon  \int (x-m(\rho))(x-m(\rho))^T h dx +O(\epsilon^2)\\
=&\epsilon  \int (x-m(\rho))(x-m(\rho))^T h dx +O(\epsilon^2).
\end{aligned}
$$
$$
B_{\rho+\epsilon h}(\Phi)-B_\rho(\Phi) = \epsilon \int h \nabla\Phi \nabla \Phi^T dx.
$$
Hence, we can derive
$$
\begin{aligned}
&\lra{C^\lambda(\rho+\epsilon h), B_{\rho+\epsilon h}(\Phi)}-\lra{C^\lambda(\rho), B_\rho(\Phi)}\\
=&\epsilon \int h \lra{\nabla\Phi \nabla \Phi^T, C(\rho)}dx+\epsilon \int h \lra{(x-m(\rho))(x-m(\rho))^T,B_\rho(\Phi)} dx+O(\epsilon^2).
\end{aligned}
$$
This proves \eqref{equ:delta_rho}. Hence, the Hamiltonian flow under the Kalman-Wasserstein metric follows
\begin{equation}\label{equ:kw_hf}
\left\{\begin{aligned}
&\p_t\rho_t+\nabla \cdot(\rho_tC^\lambda(\rho_t)\nabla \Phi_t)=0,\\
&\p_t \Phi_t+ \frac{1}{2}\lp(x-m(\rho_t))^TB_{\rho_t}(\Phi_t)(x-m(\rho_t))+\nabla\Phi_t(x)^TC^\lambda(\rho_t)\nabla \Phi_t(x)\rp+\frac{\delta E}{\delta \rho_t}=0.\\
\end{aligned}\right.
\end{equation}
Adding a linear damping term $\alpha_t\Phi_t$ to the second equation in \eqref{equ:kw_hf} yields Example 7.

For Example 10, suppose that $X_t$ follows $\rho_t$ and $V_t = \nabla \Phi_t(X_t)$. Then, we shall have
$$
\frac{d}{dt}X_t = C^\lambda(\rho_t) V_t,\\ 
$$
Note that $V_t = \nabla \Phi_t(X_t)$, we can establish that
$$
\begin{aligned}
&\frac{d}{dt} V_t = (\p_t+(C^\lambda(\rho_t)\nabla\Phi_t\cdot\nabla)\nabla\Phi_t(X_t)\\
=&\nabla\p_t\Phi_t(X_t)+\nabla^2\Phi_t(X_t)C^\lambda(\rho_t)\nabla \Phi_t(X_t).
\end{aligned}
$$
The last inequality can be established as follows. For $i=1,\dots,d$, we have
$$
\begin{aligned}
&(C^\lambda(\rho_t)\nabla\Phi_t\cdot\nabla)\nabla_i\Phi_t(X_t)=\sum_{j=1}^d(C^\lambda(\rho_t)\nabla\Phi_t)_j\nabla_j\nabla_i\Phi_t(X_t)\\
=&\sum_{j=1}^d\nabla_{ij}\Phi_t(X_t)(C^\lambda(\rho_t)\nabla\Phi_t)_j = (\nabla^2\Phi_tC^\lambda(\rho_t)\nabla\Phi_t)_i.
\end{aligned}
$$
According to the chain rule, we also have
$$
\nabla (\nabla\Phi_t(x)^TC^\lambda(\rho_t)\nabla \Phi_t(x)) = 2\nabla^2\Phi_t(x)C^\lambda(\rho_t)\nabla\Phi_t(x)
$$
As a result, we can establish that
\begin{equation}
\begin{aligned}
\frac{d}{dt}V_t =& -\alpha_t V_t-B_{\rho_t}(\Phi_t)(X_t-M(\rho_t))-\nabla\delta_{\rho_t} E \\
=&-\alpha_t V_t-\mbE[V_t V_t^T](X_t-\mbE[X_t])-\nabla\delta_{\rho_t} E.
\end{aligned}
\end{equation}
In summary, the KW-AIG flow in the particle formulation takes the form \eqref{equ:kw_aig_p}

\subsection{The derivations of Example 8 and 11 (Stein metric) in Section 3}
For an objective function $E(\rho)$, the Hamiltonian follows
$$
\mcH(\rho,\Phi)=\frac{1}{2}\int \int  \nabla \Phi(x)^T\nabla \Phi(y) k(x,y) \rho(x) \rho(y)dxdy +E(\rho).
$$
We note that
$$
\begin{aligned}
&\frac{\delta}{\delta \rho }\bb{\frac{1}{2}\int \int  \nabla \Phi(x)^T\nabla \Phi(y) k(x,y) \rho(x) \rho(y)dxdy}(x)\\
=&\int   \nabla \Phi(x)^T\nabla \Phi(y) k(x,y) \rho(y)dy.
\end{aligned}
$$
Hence, the Hamiltonian flow writes
\begin{equation}\label{equ:s_hf}
    \lbb{
&\p_t\rho_t(x) = -\nabla_x\cdot\pp{\rho_t(x) \int k(x,y) \rho_t(y) \nabla_y\Phi_t(y) dy},\\
&\p_t\Phi_t(x) = - \int   \nabla \Phi_t(x)^T\nabla \Phi_t(y) k(x,y) \rho_t(y)dy -\frac{\delta E}{\delta \rho_t}(x).
}
\end{equation}
Adding a linear damping term $\alpha_t\Phi_t$ to the second equation in \eqref{equ:s_hf} yields Example 8.

For Example 11, similarly, suppose that $X_t$ follows $\rho_t$ and $V_t = \nabla \Phi_t(X_t)$. Then, we shall have
$$
\frac{d}{dt} X_t = \int k(X_t,y) \nabla\Phi_t(y)\rho_t(y)dy.
$$
We note that
$$
\begin{aligned}
&\nabla\pp{\int   \nabla \Phi(x)^T\nabla \Phi(y) k(x,y) \rho(y)dy}\\
=&\nabla^2 \Phi(x) \int \nabla \Phi(y) k(x,y) \rho(y)dy+\int \nabla \Phi(x)^T \nabla \Phi (y) \nabla_x k(x,y) \rho(y)dy.
\end{aligned}
$$
Hence, we have
$$
\begin{aligned}
&\frac{d}{dt}  V_t = \p_t\nabla\Phi_t(X_t) +\nabla^2\Phi_t (X_t)\pp{\int k(x,y) \rho_t(y) \nabla_y\Phi_t(y) dy}\\
=&-\alpha_t\ \nabla \Phi_t(X_t)-\int \nabla \Phi_t(X_t)^T \nabla \Phi_t (y) \nabla_x k(X_t,y) \rho(y)dy-\nabla\pp{\frac{\delta E}{\delta \rho_t}}(X_t)\\
=&-\alpha_tV_t-\int V_t^T\nabla \Phi_t (y) \nabla_x k(X_t,y) \rho(y)dy-\nabla \pp{\frac{\delta E}{\delta \rho_t}}(X_t).
\end{aligned}
$$
This derives Example 11.

\section{Wasserstein metric in Gaussian families}
In this section, we first introduce the Wasserstein metric, gradient flows and Hamiltonian flows in Gaussian families. Then, we validate the existence of (W-AIG) in Gaussian families. Denote $\mcN_n^0$ to the multivariate Gaussian densities with zero means.
Namely, if $\rho_0, \rho^*\in \mcN_n^0$, then we show that (W-AIG) has a solution $(\rho_t,\Phi_t)$ and $\rho_t\in \mcN_n^0$. 


Let $\mbP^n$ and $\mbS^n$ represent symmetric positive definite matrix and symmetric matrix with size $n\times n$ respectively. Each $\rho\in \mcN_n^0$ is uniquely determined by its covariance matrix $\Sigma\in \mbP^n$.
 The Wasserstein metric $G^W(\rho)$ on $\mcP(\mbR^n)$ induces the Wasserstein metric $G^W(\Sigma)$ on $\mbP^n$, which is also known as the Bures metric, see \citep{owgot,gomds,wrgop}. For $\Sigma\in \mbP^n$, the tangent and cotangent space follow $T_{\Sigma}\mbP^n\simeq T^*_{\Sigma}\mbP^n\simeq \mbS^n$. 
 \begin{definition}[Wasserstein metric in Gaussian families]
 For $\Sigma\in \mbP^n$, the metric tensor $G^W(\Sigma):\mbS^n\to \mbS^n$ is defined by 
 $$
 G^W(\Sigma)^{-1}S = 2(\Sigma S+S\Sigma).
 $$
 The Wasserstein metric on $\mbS^n$ follows
 $$
 g_{\Sigma}^W(A_1,A_2) = \tr(A_1G(\Sigma)A_2)= 4\tr(S_1\Sigma S_2),
 $$
 where $S_i\in \mbS^n$ is the solution to 
 $$
 A_i = 2(\Sigma S_i+S_i\Sigma),\quad i=1,2.
 $$  
 \end{definition}

\subsection{Gradient flows and Hamiltonian flows in Gaussian families}
We derive the Wasserstein gradient flow and the Wasserstein Hamiltonian flow  in Gaussian families as follows.
 \begin{proposition}
 \label{prop:ghf}
 The Wasserstein gradient flow in Gaussian families writes 
 $$
 \dot \Sigma_t = -2(\Sigma_t\nabla_{\Sigma_t} E(\Sigma_t)+\nabla_{\Sigma_t} E(\Sigma_t)\Sigma_t).$$
 Here $\nabla_{\Sigma_t}$ is the standard matrix derivative.

 The Wasserstein Hamiltonian flow satisfies
 \begin{equation}\label{equ:g_hf}
 \left\{
 \begin{aligned}
 &\dot \Sigma_t-2(S_t\Sigma_t+\Sigma_t S_t)=0,\\
 &\dot S_t+2S_t^2+\nabla_{\Sigma_t} E(\Sigma_t)=0,
 \end{aligned}\right.
 \end{equation}
 where $S_t\in \mbS^n$. The corresponding Hamiltonian satisfies
 $$\mcH^W(\Sigma_t,S_t)=2\tr(S_t\Sigma_tS_t)+E(\Sigma_t).$$
 \end{proposition}
The derivation of the gradient flow simply follows the definition of Wasserstein metric in Gaussian families. 

We then derive the Hamiltonian flow as follows. For $A\in \mbS^n$, we define the linear operator $M_A:\mbS^n\to \mbS^n$ by
\begin{equation*}
    M_AB = AB+BA,\quad B\in \mbS^n.
\end{equation*}
It is easy to verify that if $A\in \mbP^n$, then $M_A^{-1}$ is well-defined. For a flow $\Sigma_t\in \mbP^n, t\geq0$, we define the Lagrangian $L(\Sigma_t,\dot \Sigma_t)=\frac{1}{2}g_{\Sigma_t}(\dot\Sigma_t,\dot \Sigma_t)-E(\Sigma_t).$
The corresponding Euler-Lagrange equation writes
\begin{equation}\label{equ:g_el}
\frac{d }{d t}\frac{d L}{d\dot\Sigma_t}=\frac{d L}{d \Sigma}.
\end{equation}
Let $S_t=\frac{1}{2}M_{\Sigma_t}^{-1}\dot \Sigma_t$, i.e., $\dot \Sigma_t=2(S_t\Sigma_t+\Sigma_t S_t)$. Then, it follows
$$
\begin{aligned}
&g_{\Sigma_t}(\dot\Sigma_t,\dot\Sigma_t)=4\tr(S_t\Sigma_tS_t)=2\tr((S_t\Sigma_t+\Sigma_t S_t)S_t)\\
=&\tr(\dot\Sigma_tS_t)=\frac{1}{2}\tr(\dot \Sigma_tM_{\Sigma_t}^{-1}\dot \Sigma_t).
\end{aligned}
$$
This leads to $\frac{d L}{d \dot \Sigma_t}=\frac{1}{2}M_{\Sigma_t}^{-1}\dot\Sigma_t=S_t.$
For simplicity, we denote $g=g_{\Sigma_t}(\dot \Sigma_t,\dot \Sigma_t)$. First, we show that 
$$
\frac{d g}{d \Sigma_t}=-4S_t^2.
$$
Because $S_t=\frac{1}{2}M_{\Sigma_t}^{-1}\dot \Sigma_t$. Given $\dot\Sigma_t$, $S_t$ can be viewed as a continuous function of $\Sigma_t$. For any $A\in \mbS^n$, define $l_A=\tr((\Sigma_t S_t+S_t \Sigma_t)A)$. 
$$
\begin{aligned}
&0=\frac{d l_A}{d\Sigma_t}=\frac{\p S_t}{\p \Sigma_t} \frac{\p l_A}{\p S_t}+\frac{\p l_A}{\p \Sigma_t}\\
=&\frac{\p S_t}{\p \Sigma_t}(A\Sigma_t+\Sigma_t A)+(AS_t+S_tA).
\end{aligned}
$$
Here we view $\p S_T/\p \Sigma_t$ as a linear operator on $S^n$. Let $B=A\Sigma_t+\Sigma_t A$, then $A=M_{\Sigma_t}^{-1}B$. 
$\frac{\p S_t}{\p \Sigma_t} B+M_{S_t}M_{\Sigma_t}^{-1}B=0$ holds
for all $B\in S^n$. Therefore, we have $\frac{\p S_t}{\p \Sigma_t}=-M_{S_t}M_{\Sigma_t}^{-1}$. Hence,
$$
\begin{aligned}
\frac{d g}{d\Sigma_t}=&\frac{\p S_t}{\p \Sigma_t} \frac{\p g}{\p S_t}+\frac{\p g}{\p \Sigma_t} \\
=& -4M_{S_t}M_{\Sigma_t}^{-1}(S_t\Sigma_t+\Sigma_tS_t)+4S_t^2\\
=&-4M_{S_t}S_t+4S_t^2=-4S_t^2.
\end{aligned}
$$
As a result, the Euler-Lagrange equation \eqref{equ:g_el} is equivalent to
\begin{equation}\label{equ:g_el1}
\dot S_t =\frac{d}{dt}\frac{d L}{d\dot\Sigma_t}=\frac{d L}{d\Sigma_t}=-2S_t^2-\nabla E(\Sigma_t).
\end{equation}
Combining \eqref{equ:g_el1} with $\dot \Sigma_t=S_t\Sigma_t+\Sigma_t S_t$ renders the Hamiltonian flow in Gaussian families. 


\subsection{Proof of Proposition \ref{thm:connect}}
By adding a damping term $\alpha_tS_t$, we derive \eqref{equ:w_aig_g}, i.e., the Wasserstein AIG flow in Gaussian families. We present the proof of Proposition \ref{thm:connect} as follows. 
We first show that $\Sigma_t$ stays in $\mbP^n$. Suppose that $\Sigma_t\in \mbP_n$ for $0\leq t\leq T$. Define $H_t=H(\Sigma_t,S_t)=2\tr(S_t\Sigma_t S_t)+E(\Sigma_t)$. We observe that (W-AIG-G) is equivalent to
\begin{equation}
\dot \Sigma_t = \frac{\p H_t}{\p S_t},\quad \dot S_t = -\alpha_t S_t-\frac{\p H_t}{\p \Sigma_t}.
\end{equation}
We show that $H_t$ is decreasing with respect to $t$. 
$$
\begin{aligned}
&\frac{dH_t}{dt}=\tr\lp\frac{\p H_t}{\p S_t} \dot S_t+\frac{\p H_t}{\p \Sigma_t} \dot \Sigma_t\rp\\
=&\tr \lp\frac{\p H_t}{\p S_t}\lp-\alpha_tS_t-\frac{\p H_t}{\p \Sigma_t}\rp+\frac{\p H_t}{\p \Sigma_t}\frac{\p H_t}{\p S_t}\rp\\
=&-\alpha_t\tr\lp S_t\frac{\p H_t}{\p S_t}\rp=-2\alpha_t\tr(S_t(\Sigma_tS_t+S_t\Sigma_t))\\
=&-4\alpha_t \tr(S_t\Sigma_tS_t)\leq0.
\end{aligned}
$$
For simplicity, we denote $W^*=(\Sigma^*)^{-1}$. Let $\lambda_t$ be the smallest eigenvalue of $\Sigma_t$. Then, $\log\det(\Sigma_tW^*)=\log\det W^*+\log\det(\Sigma_t)\geq \log\det W^*+n\log\lambda_t.$
Therefore,
$$
\begin{aligned}
&-\frac{n}{2}(\log\lambda_t+1)-\frac{1}{2}\log\det W^*\\
\leq& -\frac{1}{2}\bb{\log\det (\Sigma_tW^*)+n}\\
\leq& E(\Sigma_t)\leq H(t)\leq H(0),
\end{aligned}
$$
which yields that
\begin{equation}\label{lambda_lower}
\lambda_t\geq \exp\lp-\frac{2}{n}H(0)-1-\frac{1}{n}\log\det W^*\rp.
\end{equation}
This means that as long as $\Sigma_t\in \mbP_n$, the smallest eigenvalue of $\Sigma_t$ has a positive lower bound. If there exists $T>0$ such that $\Sigma_T\notin \mbP_n$. Because $\Sigma_t$ is continuous with respect to $t$, there exists $T_1<T$, such that $\Sigma_t\in P_n$, $0\leq t\leq T_1$ and $\lambda_{T_1}<\exp\lp-2H(0)/n-1\rp$, which violates \eqref{lambda_lower}. 

We then reveal the relationship between (W-AIG) in $\mcP(\mbR^n)$ and $\mbP^n$. We observe that
$$
\begin{aligned}
&\frac{\p}{\p t}\det(\Sigma_t)=\det(\Sigma_t)\tr(\Sigma_t^{-1}\dot\Sigma_t),\\
&\frac{\p}{\p t} \Sigma_t^{-1}=-\Sigma_t^{-1}\dot \Sigma_t\Sigma_t^{-1}.
\end{aligned}
$$
Combining with $\dot\Sigma_t=2(\Sigma_tS_t+S_t\Sigma_t)$, we obtain
$$
\begin{aligned}
\tr(\Sigma_t^{-1}\dot \Sigma_t)=&2\tr(S_t+\Sigma_t^{-1}S_t\Sigma_t)=4\tr(S_t),\\
\tr(x\Sigma_t^{-1}\dot \Sigma_t \Sigma_t^{-1}x)=&2\tr(x^T\Sigma_t^{-1}S_tx+x^TS_t\Sigma_t^{-1}x)=4\tr(S_t\Sigma_t^{-1}xx^T).
\end{aligned}
$$
Therefore, it follows
$$
\begin{aligned}
\p_t \rho_t(x) =&\frac{\p }{\p t}\lp\frac{1}{{\sqrt{\det(\Sigma_t)}}}\rp\sqrt{\det(\Sigma_t)}\rho_t(x) +\frac{1}{2}\tr(x^T\Sigma_t^{-1}\dot\Sigma_t\Sigma_t^{-1}x)\rho_t(x)\\
=&-\frac{1}{2}\tr(\Sigma_t^{-1}\dot\Sigma_t)\rho_t(x)+2\tr(S_t\Sigma_t^{-1}xx^T)\rho_t(x)\\
=&-2\tr(S_t(I-\Sigma_t^{-1}xx^T))\rho_t(x).\\
\end{aligned}
$$
Note that $\nabla \Phi_t(x) = 2S_tx$. Hence, we have
$$
\begin{aligned}
&-\nabla\cdot(\rho_t\nabla\Phi_t) = -2\sum_{i=1}^n \p_i (\rho_t(x) S_tx)_i\\
=&-2\sum_{i=1}^n\lb\rho_t(x) \p_i(S_tx)_i+(S_tx)_i\p_i \rho_t(x)\rb\\
=&-2\rho_t(x) \lb\tr(S_t)+(S_tx)^T(-\Sigma_t^{-1}x)\rb\\
=&-2\rho_t(x)\tr(S_t(I-\Sigma_t^{-1}xx^T))=\p_t\rho_t(x).
\end{aligned}
$$
The first equation of (W-AIG) holds. Because $\p_t\Phi_t(x)=x^T\dot S_tx+\dot C(t)$, 
$$
\begin{aligned}
&\p_t\Phi_t(x)+\alpha_t\Phi_t(x)+\frac{1}{2}\|\nabla \Phi_t(x)\|^2\\
=&x^T\dot S_tx+\alpha_tx^TS_tx+2x^TS_t^2x+\dot C(t)\\
=&-x^T\nabla_{\Sigma_t} E(\Sigma_t)x+\dot C(t)\\
=&\frac{1}{2}x^T(\Sigma_t^{-1}-W^*)x+\dot C(t).
\end{aligned}
$$
Note that $\rho^*$ is the Gaussian density with the covariance matrix $\Sigma^*$. Because $\dot C(t) = \frac{1}{2}\log\det(\Sigma_tW^*)-1$, we can compute
$$
\begin{aligned}
&\frac{\delta E}{\delta \rho_t}=\log \rho_t(x)-\log\rho^*(x)+1\\
=&-\frac{1}{2}x^T(\Sigma_t^{-1}-W^*)x-\frac{1}{2}\log\det(\Sigma_tW^*)+1\\
=&-\frac{1}{2}x^T(\Sigma_t^{-1}-W^*)x-\dot C(t) \\
=& -(\p_t\Phi_t(x)+\alpha_t\Phi_t(x)+\frac{1}{2}\|\nabla \Phi_t(x)\|^2).
\end{aligned}
$$
Therefore, the second equation of (W-AIG) holds. Because $\Sigma_t|_{t=0}=\Sigma_0$,  $S_t|_{t=0}=0$ and $C(0)=0$, we have $\rho_t|_{t=0}=\rho_0$ and $\Phi_t|_{t=0}=0$. This completes the proof.

\section{Proof of convergence rate under Wasserstein metric}
In this section, we briefly review the Riemannian structure of probability space and present proofs of propositions in Section 4 under Wasserstein metric.
\subsection{A brief review on the geometric properties of the probability space}
Suppose that we have a metric $g_\rho$ in probability space $\mcP(\Omega)$. Given two probability densities $\rho_0,\rho_1\in\mcP(\Omega)$, we define the distance as follows
{\small
$$
\begin{aligned}
&\mcD(\rho_0,\rho_1)^2 =\inf_{\hat\rho_s}\left\{\int_0^1g_{\hat \rho_s}(\p_s\hat \rho_s,\p_s\hat \rho_s)ds:\hat \rho_s|_{s=0}=\rho_0,\hat \rho_s|_{s=1}=\rho_1\right\}.
\end{aligned}
$$}
The minimizer $\hat \rho_s$ of the above problem is defined as the geodesic curve connecting $\rho_0$ and $\rho_1$. An exponential map at $\rho_0\in\mcP(\Omega)$ is a mapping from the tangent space $T_{\rho_0}\mcP(\Omega)$ to $\mcP(\Omega)$. Namely, $\sigma\in T_{\rho_0}\mcP(\Omega)$ is mapped to a point $\rho_1\in \mcP(\Omega)$ such that there exists a geodesic curve $\hat \rho_s$ satisfying $\hat \rho_s|_{s=0}=\rho_0,\p_s\hat \rho_s|_{s=0}=\sigma,$ and $\hat \rho_s|_{s=1}=\rho_1$. 

\subsection{The inverse of exponential map}
In this subsection, we characterize the inverse of exponential map in the probability space with the Wasserstein metric. 

 \begin{proposition}\label{prop:deri}
 Denote the geodesic curve $\gamma(s)$ that connects $\rho_t$ and $\rho^*$ by $\gamma(s)=(sT_t+(1-s)\Id)\#\rho_t,\,s\in[0,1]$. Here $\Id$ is the identity mapping from $\mbR^n$ to itself. Then, $\p_s\gamma(s)|_{s=0}$ corresponds to a tangent vector $-\nabla\cdot(\rho_t(x) (T_t(x)-x))\in T_{\rho_t}\mcP(\Omega)$.
 \end{proposition}

For simplicity, we denote $T_t^{s}=(sT_t+(1-s)\Id)^{-1},s\in\bb{0,1}$. Based on the theory of optimal transport \citep{tiot}, we can write the explicit formula of the geodesic curve $\gamma(s)$ by
$$
\gamma(s) = T_t^{s}\#\rho_t=\det(\nabla T_t^s)\rho_t\circ T_t^s.
$$
Through basic calculations, we can compute that
$$
\left.\frac{d}{ds}T_t^s\right|_{s=0}=-\left.\frac{d}{ds}(sT_t+(1-s)\Id)\right|_{s=0}=\Id-T_t.
$$
{\small
$$
\begin{aligned}
\left.\frac{d}{ds}\det(\nabla T_t^s)\right|_{s=0}=&\left.\frac{d}{ds}\det(I+s(I-DT_t)+o(s))\right|_{s=0}\\
=&\tr(I-DT_t).
\end{aligned}
$$
}
Therefore, we have
$$
\begin{aligned}
&\left.\p_s \gamma(s)\right|_{s=0}(x)\\
=&\tr(I-\nabla T_t)\rho_t(x)+\la\nabla \rho_t(x),x-\varphi_t(x)\ra\\
=&\nabla\cdot(x-T_t(x))\rho_t(x)+\la\nabla \rho_t(x),x-T_t(x)\ra\\
=&-\nabla\cdot(\rho_t(x)(T_t(x)-x)),
\end{aligned}
$$
which completes the proof.


\subsection{The proof of Proposition 4 and 5}
\label{app:56}
The main goal of this subsection is to prove the Lyapunov function $\mcE(t)$ is non-increasing. 

\textbf{Preparations.} We first give a better characterization of the optimal transport plan $T_t$. We can write $T_t=\nabla \Psi_t$, where $\Psi_t$ is a strictly convex function, see \citep{tiot}. This indicates that $\nabla T_t$ is symmetric. We then introduce the following proposition.
\begin{proposition} \label{prop:geo}
Suppose that $E(\rho)$ satisfies Hess($\beta$) for $\beta\geq 0$. Let $T_t(x)$ be the optimal transport plan from $\rho_t$ to $\rho^*$, then 
$$
\begin{aligned}
E(\rho^*)\geq & E(\rho_t)+\int  \la T_t(x)-x,\nabla \frac{\delta E}{\delta \rho_t}\ra\rho dx+\frac{\beta}{2} \int  \|T_t(x)-x\|^2\rho_t dx.
\end{aligned}
$$
\end{proposition} 
This is a direct result of $\beta$-displacement convexity of $E(\rho)$ based on Proposition \ref{prop:deri}. 

\begin{lemma}
Denote $u_t=\p_t (T_t)^{-1}\circ T_t$. Then,$u_t$ satisfies
\begin{equation}\label{equ:vt}
\nabla\cdot \lp\rho_t(u_t-\nabla\Phi_t)\rp=0.
\end{equation}
We also have
\begin{equation}\label{equ:pttt}
\p_tT_t(x)=-\nabla T_t(x)u_t(x).
\end{equation}
\end{lemma}
\begin{proof}
Because $(T_t)^{-1}\#{\rho^*}={\rho_t}$, let $u_t=\p_t(T_t)^{-1}\circ T_t$ and $X_t=(T_t)^{-1}X_0$, where $X_0\sim \rho^*$. This yields $\frac{d}{dt}X_t=u_t(X_t)$. The distribution of $X_t$ follows $\rho_t$. By the Euler's equation, $\rho_t$ shall follows $$\p_t\rho_t +\nabla \cdot(\rho_t u_t)=0.$$ 
Combining this with the continuity equation \eqref{equ:cont} yields \eqref{equ:vt}. 

Then, we formulate $\p_tT_t(x)$ with $u_t$. By the Taylor expansion,
$$
T_{t+s}(x)=T_t(x)+s\p_t T_t(x)+o(s).
$$
Let $y=(T_t)^{-1}x$.  it follows
$$
\begin{aligned}
(T_{t+s})^{-1}(x)=&(T_t)^{-1}(x)+su_t((T_t)^{-1}(x))+o(s)=y+su_t(y)+o(s).
\end{aligned}
$$
Therefore, we have
$$
\begin{aligned}
&0=T_{t+s}((T_{t+s})^{-1}(x))-x\\
=&T_{t+s}(y+su_t (y)+o(s))-x\\
=&T_t(y+su_t (y))+s\p_t T_t(y+su_t (y))-x+o(s)\\
=&T_t(y)+s\nabla T_t(y) u_t(y)+s\p_t T_t(y)-x+o(s)\\
=&s\lb\nabla T_t(y) u_t(y)+\p_t T_t(y)\rb+o(s).
\end{aligned}
$$
We shall have $\nabla T_t(y) u_t(y)+\p_t T_t(y)=0$. Replacing $y$ by $x$ yields \eqref{equ:pttt}.
\end{proof}

The following lemma illustrates two important properties of $u_t$ and $\p_tT_t$.
\begin{lemma}\label{lem:vt}
For $u_t$ satisfying \eqref{equ:vt}, we have
$$
\begin{aligned}
&\int  \la \nabla \Phi_t-u_t, \nabla T_t\nabla \Phi_t\ra\rho_tdx\geq 0,\\
& \int \la\nabla \Phi_t-u_t, \nabla T_t(x)(T_t(x)-x) \ra\rho_t=0.
\end{aligned}
$$
\end{lemma}
\begin{proof}
We first notice that $u_t-\nabla\Phi_t$ is divergence-free in term of $\rho_t$. From $-\nabla T_t u_t = \p_tT_t = \nabla \p_t\Psi_t$, we observe that $-\nabla T_t u_t$ is the gradient of $\p_t\Psi_t$. Therefore, 
$$
\begin{aligned}
&\int  \la \nabla \Phi_t-u_t, \nabla T_tu_t \ra \rho_t = -\int \la \p_t \Psi_t, \nabla\cdot (\rho_t(\nabla \Phi_t-u_t))\ra =0.
\end{aligned}
$$
Based on our previous characterization on the optimal transport plan $T_t$, $\nabla T_t = \nabla^2\Psi_t$ is symmetric positive definite. This yields that
{\small
$$
\begin{aligned}
&\int  \la \nabla \Phi_t-u_t, \nabla T_t\nabla \Phi_t\ra\rho_tdx\\
=&\int  \la \nabla \Phi_t-u_t, \nabla T_t\nabla \Phi_t\ra\rho_tdx-\int  \la \nabla \Phi_t-u_t, \nabla T_tu_t \ra \rho_t\\
=&\int  \la \nabla \Phi_t-u_t,\nabla T_t(\nabla \Phi_t-u_t)\ra\rho_tdx\geq 0.
\end{aligned}
$$}
The last inequality utilizes that $\nabla T_t$ is positie definite and $\rho_t$ is non-negative. Then, we prove the equality in Lemma \ref{lem:vt}. Because $\nabla T_t(x)(T_t(x)-x) = \frac{1}{2}\nabla( \|T_t(x)-x\|^2+T_t(x)-\|x\|^2)$ is a gradient. Similarly, it follows 
$$
\int \la\nabla \Phi_t-u_t, \nabla T_t(x)(T_t(x)-x) \ra\rho_t=0.
$$
\end{proof}
Lemma \ref{lem:vt} and the relationship \eqref{equ:pttt} gives
\begin{equation}\label{equ:tttt1}
\begin{aligned}
&-\int \la \p_tT_t, \nabla\Phi_t\ra\rho_t dx=\int \la u_t, \nabla T_t\nabla \Phi_t\ra\rho_tdx\leq\int \la \nabla \Phi_t, \nabla T_t\nabla\Phi_t\ra\rho_tdx,
\end{aligned}
\end{equation}
\begin{equation}\label{equ:tttt2}
\begin{aligned}
&\int \la \p_tT_t, T_t(x)-x\ra\rho_t dx=-\int \la \nabla \Phi_t, \nabla T_t(x)(T_t(x)-x)\ra\rho_tdx.
\end{aligned}
\end{equation}
\textbf{Proof of Proposition 4.} Based on the definition of the Wasserstein metric, we have
\begin{equation*}
    \p_t E(\rho_t) = -\int  \frac{\delta E}{\delta \rho_t}\nabla \cdot(\rho_t\nabla \Phi_t)dx.
\end{equation*} 
Differentiating $\mcE(t)$ w.r.t. $t$ renders
{\small
\begin{align}
&\dot\mcE(t)e^{-\sqrt{\beta}t}\notag\\
=&\beta \int  \la \p_t T_t, T_t(x)-x\ra \rho_tdx-\frac{\beta}{2}\int  \|T_t(x)-x\|^2\nabla \cdot(\rho_t\nabla \Phi_t) dx\notag\\
&-\sqrt{\beta} \int \la \p_tT_t,\nabla \Phi_t\ra\rho_tdx-\sqrt{\beta}\int \la T_t(x)-x,\p_t\nabla \Phi_t\ra\rho_tdx\notag\\
&+\sqrt{\beta}\int  \la T_t(x)-x,\nabla \Phi_t\ra\nabla \cdot(\rho_t\nabla \Phi_t) dx+\int  \la\nabla \Phi_t,\p_t\nabla \Phi_t\ra\rho_tdx\notag\\
&-\frac{1}{2}\int  \|\nabla \Phi_t\|^2\nabla\cdot(\rho_t\nabla \Phi_t)-\int  \frac{\delta E}{\delta \rho_t}\nabla \cdot(\rho_t\nabla \Phi_t)dx\notag\\
&+\frac{\sqrt{\beta}}{2}\int  \|\nabla \Phi_t\|^2\rho_tdx-\beta \int  \la T_t(x)-x,\nabla \Phi_t(x)\ra\rho_tdx\notag\\
&+\frac{\sqrt{\beta^3}}{2} \int  \| T_t(x)-x\|^2\rho_tdx+\sqrt{\beta}(E(\rho_t)-E(\rho^*)).\label{no6}
\end{align}}
For the part \eqref{no6}, Proposition \ref{prop:geo} renders
\begin{equation}\label{equ:ttx}
\begin{aligned}
&\frac{\sqrt{\beta^3}}{2} \int  \| T_t(x)-x\|^2\rho_tdx+\sqrt{\beta}E(\rho_t)\\
\leq& -\sqrt{\beta} \int  \la T_t(x)-x,\nabla \frac{\delta E}{\delta \rho_t}\ra\rho_t dx.
\end{aligned}
\end{equation}
We first compute the terms with the coefficient $\beta^0$ in $\dot\mcE(t)e^{-\sqrt{\beta}t}\notag$. We observe that
\begin{equation}\label{equ:beta0}
\begin{aligned}
&\int   \la\nabla \Phi_t,\p_t\Phi_t\ra\rho_tdx -\frac{1}{2}\int   \|\nabla \Phi_t\|^2\nabla\cdot(\rho_t\nabla \Phi_t)dx\\
&-\int  \frac{\delta E}{\delta \rho_t}\nabla \cdot(\rho_t\nabla \Phi_t)\rho_t dx\\
=&\int \la\p_t\nabla \Phi_t+\frac{1}{2}\nabla \|\nabla \Phi_t\|^2+\nabla \frac{\delta E}{\delta \rho},\nabla \Phi_t\ra\rho_tdx\\
=&-2\sqrt{\beta}\int  \|\nabla \Phi_t\|^2\rho_tdx,
\end{aligned}
\end{equation}
where the last equality uses (W-AIG) with $\alpha_t=2\sqrt{\beta}$. Substituting \eqref{equ:ttx} and \eqref{equ:beta0} into the expression of $\dot\mcE(t)e^{-\sqrt{\beta}t}$ yields
{\small
\begin{equation}\label{equ:mce_1}
\begin{aligned}
\dot\mcE(t)e^{-\sqrt{\beta}t}\leq&\beta \int  \la \p_t T_t, T_t(x)-x\ra \rho_tdx-\frac{\beta}{2}\int  \|T_t(x)-x\|^2\nabla \cdot(\rho_t\nabla \Phi_t) dx\\
&-\beta \int  \la T_t(x)-x,\nabla \Phi_t\ra\rho_tdx-\sqrt{\beta} \int \la \p_tT_t,\nabla \Phi_t\ra\rho_tdx\\
&-\sqrt{\beta}\int \la T_t(x)-x,\p_t\nabla \Phi_t\ra\rho_tdx-\sqrt{\beta} \int  \la T_t(x)-x,\nabla \frac{\delta E}{\delta \rho_t}\ra\rho_t dx\\
&+\sqrt{\beta}\int  \la T_t(x)-x,\nabla \Phi_t\ra\nabla \cdot(\rho_t\nabla \Phi_t) dx-\frac{3\sqrt{\beta}}{2}\int  \|\nabla \Phi_t\|^2\rho_tdx.
\end{aligned}
\end{equation}}
Then, we deal with the terms with $\nabla \cdot(\rho_t\nabla \Phi_t)$. We have the following two identities
{\small
\begin{equation}\label{equ:idpt1}
\begin{aligned}
&\int  \la T_t(x)-x,\nabla \Phi_t\ra\nabla \cdot(\rho_t\nabla \Phi_t) dx\\
=&-\int  \la\nabla \la T_t(x)-x,\nabla \Phi_t\ra,\nabla \Phi_t\ra\rho_t dx\\
=&-\int  \la \nabla \Phi_t,\nabla^2\Phi_t(x) (T_t(x)-x)+(\nabla T_t-I)\nabla \Phi_t\ra \rho_tdx\\
=&-\frac{1}{2}\int  \la T_t(x)-x,\nabla\|\nabla \Phi_t\|^2\ra\rho_t dx-\int \la \nabla \Phi_t, \nabla T_t\nabla \Phi_t\ra\rho_t dx+\int  \|\nabla \Phi_t\|^2\rho_t dx.
\end{aligned}
\end{equation}}
\begin{equation}\label{equ:idpt2}
\begin{aligned}
&-\frac{1}{2}\int  \|T_t(x)-x\|^2\nabla \cdot(\rho_t\nabla \Phi_t) dx\\
=&\int \la (\nabla T_t(x)-I)(T_t(x)-x), \nabla \Phi_t\ra \rho_tdx\\
=&\int \la T_t(x)-x,\nabla T_t\nabla \Phi_t\ra\rho_tdx-\int \la T_t(x)-x,\nabla \Phi_t\ra\rho_tdx.
\end{aligned}
\end{equation}

Hence, we can proceed to compute the terms with the coefficient $\sqrt{\beta}$.  \eqref{equ:tttt1} and \eqref{equ:idpt1} yields
{\small
\begin{equation}\label{equ:sqrtb}
\begin{aligned}
&-\sqrt{\beta} \int \la \p_tT_t,\nabla \Phi_t\ra\rho_tdx-\sqrt{\beta}\int \la T_t(x)-x,\p_t\nabla \Phi_t+\nabla \frac{\delta E}{\delta \rho_t}\ra\rho_tdx\\
&-\frac{3\sqrt{\beta}}{2}\int  \|\nabla \Phi_t\|^2\rho_tdx+\sqrt{\beta}\int  \la T_t(x)-x,\nabla \Phi_t\ra\nabla \cdot(\rho_t\nabla \Phi_t) dx\\
=&-\sqrt{\beta}\int  \la \p_tT_t+ \nabla T_t\nabla \Phi_t, \nabla \Phi_t\ra\rho_t dx-\frac{\sqrt{\beta}}{2}\int  \|\nabla \Phi_t\|^2\rho_tdx\\
&-\sqrt{\beta}\int  \la T_t(x)-x,\p_t\nabla \Phi_t+\nabla\frac{\delta E}{\delta \rho}+\frac{1}{2}\nabla \|\nabla\Phi_t\|^2\ra\rho_tdx\\
\leq &-\frac{\sqrt{\beta}}{2}\int  \|\nabla \Phi_t\|^2\rho_tdx+2\beta\int \la T_t(x)-x,\nabla \Phi_t\ra\rho_tdx.\\
\end{aligned}
\end{equation}}

Substituting \eqref{equ:idpt2} and \eqref{equ:sqrtb} into \eqref{equ:mce_1} gives
$$
\begin{aligned}
&\dot\mcE(t)e^{-\sqrt{\beta}t}+\frac{\sqrt{\beta}}{2}\int  \|\nabla \Phi_t\|^2\rho_tdx\\
\leq &\beta \int  \la \p_t T_t, T_t(x)-x\ra \rho_tdx-\frac{\beta}{2}\int  \|T_t(x)-x\|^2\nabla \cdot(\rho_t\nabla \Phi_t) dx\\
&-\beta \int  \la T_t(x)-x,\nabla \Phi_t\ra\rho_tdx+2\beta\int \la T_t(x)-x,\nabla \Phi_t\ra\rho_tdx\\
=&\beta \int  \la \p_t T_t+\nabla T_t\nabla\Phi_t, T_t(x)-x\ra \rho_tdx=0,
\end{aligned}
$$
where the last equality uses \eqref{equ:tttt2}. In summary, we have
$$
\dot\mcE(t)e^{-\sqrt{\beta}t}\leq-\frac{\sqrt{\beta}}{2}\int  \|\nabla \Phi_t\|^2\rho_tdx\leq 0.
$$
\textbf{Proof of Proposition 5.} Differentiating $\mcE(t)$ w.r.t. $t$, we compute that
\begin{equation}\label{equ:et00}
\begin{aligned}
&\dot \mcE(t) \\
=&\int  \la \p_t T_t, T_t(x)-x\ra \rho_t dx-\frac{1}{2}\int  \|T_t(x)-x\|^2\nabla\cdot(\rho_t\nabla\Phi_t)dx\\
&-\int  \la \p_t T_t, \frac{t}{2}\nabla \Phi_t\ra\rho_tdx-\int \la T_t(x)-x,\frac{1}{2}\nabla\Phi_t+\frac{t}{2}\p_t\nabla\Phi_t\ra\rho_tdx\\
&+\int \la T_t(x)-x, \frac{t}{2}\nabla\Phi_t\ra\nabla\cdot(\rho_t\nabla\Phi_t)dx+\int \la \frac{t}{2}\nabla\Phi_t, \frac{1}{2}\nabla\Phi_t+\frac{t}{2}\p_t\nabla\Phi_t\ra\rho_tdx\\
&-\frac{1}{2}\int  \left\|\frac{t}{2}\nabla\Phi_t\right\|^2\nabla\cdot(\rho_t\nabla\Phi_t)dx-\frac{t^2}{4}\int \frac{\delta E}{\delta \rho_t}\nabla\cdot(\rho_t\nabla\Phi_t)dx+\frac{t}{2}(E(\rho_t)-E(\rho^*)).\\
\end{aligned}
\end{equation}

Because $E(\rho)$ is Hess($0$), Proposition \ref{prop:geo} yields
\begin{equation}\label{equ:hess0}
E(\rho_t)=E(\rho_t)-E(\rho^*)\leq-\int \la T_t(x)-x, \nabla\frac{\delta E}{\delta \rho_t}\ra\rho_t dx.
\end{equation}

Utilizing the inequality \eqref{equ:hess0} and substituting the expressions of terms involving $\p_t T_t$ and $\nabla\cdot (\rho_t\nabla \Phi_t)$ in \eqref{equ:et00} with the expressions in \eqref{equ:tttt1} \eqref{equ:tttt2} and \eqref{equ:idpt1} \eqref{equ:idpt2}, we obtain
\begin{equation}\label{equ:et_tmp1}
\begin{aligned}
\dot \mcE(t) \leq &-\int \la \nabla \Phi_t, \nabla T_t(x)(T_t(x)-x)\ra\rho_tdx+\int \la T_t(x)-x,\nabla T_t\nabla \Phi_t\ra\rho_tdx\\
&-\int \la T_t(x)-x,\nabla \Phi_t\ra\rho_tdx+\frac{t}{2}\int \la \nabla \Phi_t, \nabla T_t\nabla\Phi_t\ra\rho_tdx\\
&-\frac{1}{2}\int \la T_t(x)-x, \nabla\Phi_t\ra\rho_tdx-\frac{t}{2}\int \la \p_t\nabla \Phi_t, T_t(x)-x\ra\rho_tdx\\
&-\frac{t}{4}\int  \la T_t(x)-x,\nabla\|\nabla \Phi_t\|^2\ra\rho_t dx-\frac{t}{2}\int \la \nabla \Phi_t, \nabla T_t\nabla \Phi_t\ra\rho_t dx\\
&+\frac{t}{2}\int  \|\nabla \Phi_t\|^2\rho_t dx+\frac{t}{4}\int \|\nabla\Phi_t\|^2\rho_tdx+\frac{t^2}{4}\int \la \nabla\Phi_t, \p_t\nabla\Phi_t\ra\rho_tdx\\
&+\frac{t^2}{8}\int \la\nabla\Phi_t,\nabla\|\nabla\Phi_t\|^2\ra\rho_tdx+\frac{t^2}{4}\int \la \nabla \Phi_t, \nabla \frac{\delta E}{\delta \rho_t}\ra\rho_tdx\\
&-\frac{t}{2}\int \la T_t(x)-x, \nabla\frac{\delta E}{\delta \rho_t}\ra\rho_t dx.
\end{aligned}
\end{equation}
The expression of \eqref{equ:et_tmp1} can be reformulated into
{\small
$$
\begin{aligned}
\dot \mcE(t)\leq &-\frac{3}{2}\int \la T_t(x)-x,\nabla\Phi_t\ra\rho_tdx+\frac{3t}{4}\int \|\nabla\Phi_t\|^2\rho_tdx\\
&-\frac{t}{2}\int \la T_t(x)-x, \p_t\nabla \Phi_t+\frac{1}{2}\nabla\|\nabla\Phi_t\|^2+\nabla \frac{\delta E}{\delta \rho_t}\ra\rho_tdx\\
&+\frac{t^2}{4}\int \la \nabla\Phi_t, \p_t\nabla\Phi_t+\frac{1}{2}\nabla\|\nabla\Phi_t\|^2+\nabla \frac{\delta E}{\delta \rho_t}\ra\rho_tdx.
\end{aligned}
$$}
From (W-AIG) with $\alpha_t=3/t$, we have the following equalities.
{\small
$$
\begin{aligned}
&\frac{t^2}{4}\int\la \nabla\Phi_t, \p_t\nabla\Phi_t+\frac{1}{2}\nabla\|\nabla\Phi_t\|^2+\nabla \frac{\delta E}{\delta \rho_t}\ra\rho_tdx=-\frac{3t}{4}\int \|\nabla\Phi_t\|^2\rho_tdx,
\end{aligned}
$$
$$
\begin{aligned}
&-\frac{t}{2}\int \la T_t(x)-x, \p_t\nabla \Phi_t+\frac{1}{2}\nabla\|\nabla\Phi_t\|^2+\nabla \frac{\delta E}{\delta \rho_t}\ra\rho_tdx = \frac{3}{2}\int \la T_t(x)-x,\nabla\Phi_t\ra\rho_tdx.
\end{aligned}
$$}
As a result, $\dot\mcE(t)\leq 0$. This completes the proof. 

\subsection{Comparison with the proof in \cite{affpd}}
\label{app:cmp}
The accelerated flow in \citep{affpd} is given by 
\begin{equation}\label{equ:agf}
    \frac{d X_t}{dt} = e^{\alpha_t-\gamma_t}Y_t,\quad \frac{d Y_t}{dt} = -e^{\alpha_t+\beta_t+\gamma_t}\nabla \lp\frac{\delta E}{\delta_{\rho_t}}\rp(X_t).
\end{equation}
Here the target distribution satisies $\rho_\infty(x)=\rho^*(x)\propto \exp(-f(x))$. Suppose that we take $\alpha_t = \log p-\log t$, $\beta_t = p\log t +\log C$ and $\gamma_t = p\log t$. Here we specify $p=2$ and $C=1/4$. Then the accelerated flow \eqref{equ:agf} recovers the particle formulation of W-AIG flows if we replace $Y_t$ by $2t^{-3}V_t$. The Lyapunov function in \citep{affpd} follows
\begin{equation*}
\begin{aligned}
V(t) =& \frac{1}{2}\mbE\bb{\|X_t+e^{-\gamma_t}Y_t-T_{\rho_t}^{\rho^*}(X_t)\|^2}+e^{\beta_t}(E(\rho)-E(\rho^*))\\
=&\frac{1}{2}\mbE\bb{\|X_t+\frac{t}{2}V_t-T_{\rho_t}^{\rho^*}(X_t)\|^2}+\frac{t^2}{4}(E(\rho_t)-E(\rho^*))\\
=&\frac{1}{2}\int  \left\| - (T_t(x)-x)+\frac{t}{2} \nabla \Phi_t(x)\right\|^2\rho_t(x) dx+\frac{t^2}{4}(E(\rho_t)-E(\rho^*)).
\end{aligned}
\end{equation*}
The last equality is based on the fact that $V_t = \nabla \Phi_t(X_t)$ and $T_t=T_{\rho_t}^{\rho^*}$ is the optimal transport plan from $\rho_t$ to $\rho^*$. This indicates that the Lyapunov function in \citep{affpd} is identical to ours. The technical assumption in \citep{affpd} follows 
\begin{equation*}
    \begin{aligned}
    0=&\mbE \lb \lp X_t+e^{-\gamma_t}Y_t-T_{\rho_t}^{\rho^*}(X_t)\rp\cdot\frac{d}{dt}T_{\rho_t}^{\rho^*}(X_t)\rb\\
    =&\mbE \lb \lp X_t+\frac{t}{2}V_t-T_t(X_t)\rp \cdot\frac{d}{dt}T_t(X_t)\rb\\
    =&\mbE \lb \lp X_t+\frac{t}{2}V_t-T_t(X_t)\rp \cdot\lp (\p_tT_t)(X_t)+\nabla T_tV_t\rp\rb\\
    =&\int \la x-T_t(x)+\frac{t}{2}\nabla \Phi_t(x),\p_tT_t+\nabla T_t \nabla \Phi_t\ra \rho_tdx.
    \end{aligned}
\end{equation*}
Based on $\p_tT_t=-\nabla T_tu_t$ and Lemma \ref{lem:vt}, we have
\begin{equation*}
\begin{aligned}
    &\int \la x-T_t(x),\p_tT_t+\nabla T_t \nabla \Phi_t\ra \rho_tdx \\
    =& \int \la x-T_t(x),\nabla T_t (\nabla \Phi_t-u_t)\ra \rho_tdx=0.
\end{aligned}
\end{equation*}
\begin{equation*}
\begin{aligned}
&\int \la \nabla \Phi_t,\p_tT_t+\nabla T_t \nabla \Phi_t\ra \rho_tdx \\
=& \int \la \nabla \Phi_t,\nabla T_t(\nabla \Phi_t-u_t)\ra\rho_tdx\\
=&\int  \la \nabla \Phi_t-u_t,\nabla T_t(\nabla \Phi_t-u_t)\ra\rho_tdx\geq 0.
\end{aligned}
\end{equation*}
As a result, we have
$$
\begin{aligned}
&\mbE \lb \lp X_t+e^{-\gamma_t}Y_t-T_{\rho_t}^{\rho_\infty}(X_t)\rp\cdot\frac{d}{dt}T_{\rho_t}^{\rho_\infty}(X_t)\rb \\
=& \frac{t}{2}\int  \la \nabla \Phi_t-u_t,\nabla T_t(\nabla \Phi_t-u_t)\ra\rho_tdx\geq 0.
\end{aligned}
$$ 
In 1-dimensional case, because $\nabla\cdot \lp\rho_t(u_t-\nabla\Phi_t)\rp=0$ indicates that $\rho_t(u_t-\nabla\Phi_t)=0$. For $\rho_t(x)>0$, we have $u_t(x)-\nabla\Phi_t(x) = 0$. So the technical assumption holds. In general cases, although $u_t = \p_t(T_t)^{-1}\circ T_t$ satisfies $\nabla\cdot \lp\rho_t(u_t-\nabla\Phi_t)\rp=0$, but this does not necessary indicate that $u_t=\nabla\Phi_t$. Hence, $\mbE \lb \lp X_t+e^{-\gamma_t}Y_t-T_{\rho_t}^{\rho_\infty}(X_t)\rp\cdot\frac{d}{dt}T_{\rho_t}^{\rho_\infty}(X_t)\rb =0$ does not necessary hold except for 1-dimensional case.
\section{Proof of convergence rate under Fisher-Rao metric}
In this section, we present proofs of propositions in Section 4 under Fisher-Rao metric.

\subsection{Geodesic curve under the Fisher-Rao metric}
We first investigate on the explicit solution of geodesic curve under the Fisher-Rao metric in probability space. The geodesic curve shall satisfy
\begin{equation}\label{geo:fr}
\lbb{&\p_t\rho_t-(\Phi_t-\mbE_{\rho_t}[\Phi_t])\rho_t=0,\\
&\p_t\Phi_t+\frac{1}{2}\Phi_t^2-\mbE_{\rho_t}[\Phi_t]\Phi_t=0.}
\end{equation}
with initial values $\rho_t|_{t=0}=\rho_0$ and $\Phi_t|_{t=0}=\Phi_0$. The Hamiltonian follows
$$
\mcH(\rho,\Phi) = \frac{1}{2}(\mbE_{\rho_t}[\Phi_t^2]-\pp{\mbE_{\rho_t}[\Phi_t]}^2).
$$
We reparametrize $\rho_t$ by $\rho_t= R_t^2$ with $R_t>0$ and $\int R_t^2 dx=1$. Then, 
$$
\lbb{&\p_tR_t-\frac{1}{2}(\Phi_t-\mbE_{R_t^2}[\Phi_t])R_t=0,\\
&\p_t\Phi_t+\frac{1}{2}\Phi_t^2-\mbE_{R_t^2}[\Phi_t]\Phi_t=0.}
$$
\begin{proposition}
The solution to \eqref{geo:fr} with initial values $\rho_t|_{t=0}=\rho_0$ and $\Phi_t|_{t=0}=\Phi_0$ follows
\begin{equation}\label{geo:fr_sol}
R(x,t) = A(x)\sin(Ht)+B(x)\cos(Ht),
\end{equation}
where
\begin{equation}\label{geo:fr_para}
A(x) = \frac{1}{2H} R_0(x)(\Phi_0(x)-\mbE_{R_0^2}[\Phi_0]),\quad B(x)=R_0(x),
\end{equation}
and
$$
H=\frac{1}{2}\sqrt{\mbE_{R_0^2}[\Phi_0^2]-\pp{\mbE_{R_0^2}[\Phi_0]}^2}.
$$
We also have $\int R_t^2 dx=1$ for $t\geq 0$. 
\end{proposition}
\begin{proof}
We can compute that
$$
\begin{aligned}
2\p_{tt}R_t=& \pp{\p_t\Phi_t-2\int R_t\Phi_t\p_t R_t dx-\mbE_{R_t^2}[\p_t\Phi_t]}R_t+\p_t R_t (\Phi_t-\mbE_{R_t^2}[\Phi_t])\\
=& \pp{-\frac{1}{2}\Phi_t^2+\frac{1}{2}\mbE_{R_t^2}[\Phi_t^2]+\mbE_{R_t^2}[\Phi_t]\Phi_t-\mbE_{R_t^2}[\Phi_t]^2}R_t\\
&-\mbE_{R_t^2}[\Phi_t(\Phi_t-\mbE_{R_t^2}[\Phi_t])] R_t+\frac{1}{2} R_t(\Phi_t-\mbE_{R_t^2}[\Phi_t])^2\\
=&\pp{-\frac{1}{2}\mbE_{R_t^2}[\Phi_t^2]+\frac{1}{2}\pp{\mbE_{R_t^2}[\Phi_t]}^2}R_t.
\end{aligned}
$$
In other words,
$$
\p_{tt} R_t = \pp{-\frac{1}{4}\mbE_{R_t^2}[\Phi_t^2]+\frac{1}{4}\mbE_{R_t^2}[\Phi_t]^2}R_t.
$$

We observe that $\frac{1}{2}\mbE_{R_t^2}[\Phi_t^2]-\frac{1}{2}\mbE_{R_t^2}[\Phi_t]^2=\mcH(\rho_t, \Phi_t)$ is the Hamiltonian, which is invariant along the geodesic curve. Denote
$$
H=\sqrt{\frac{1}{2}\mcH(\rho_t,\Phi_t)} = \frac{1}{2}\sqrt{\mbE_{R_0^2}[\Phi_0^2]-\pp{\mbE_{R_0^2}[\Phi_0]}^2}.
$$
Then, we have
$$
\p_{tt}R_t=-H^2R_t,
$$
which is a wave equation. We also notice that
$$
R_t(x)|_{t=0}=R_0(x),\quad \p_tR_t(x)|_{t=0}=R_0(x)(\Phi_0(x)-\mbE_{R_0^2}[\Phi_0]).
$$
Hence, $R_t$ is uniquely determined by
$$
R_t(x) = A(x)\sin(Ht)+B(x)\cos(Ht),
$$
where $A(x)$ and $B(x)$ are given in \eqref{geo:fr_para}. Finally, we verify that $\int R_t^2 dx=1$. Actually, we can compute that
$$
\int A^2(x)dx=\frac{1}{4H^2} \mbE_{R_0^2} [(\Phi_0(x)-\mbE_{R_0^2}[\Phi_0])^2]=1,
$$
$$
\int B^2(x)dx = \int R_0^2(x)dx=1,
$$
$$
\int A(x)B(x) dx = \frac{1}{2H} \mbE_{R_0^2}[\Phi_0(x)-\mbE_{R_0^2}[\Phi_0]]=0.
$$
Hence, 
$$
\begin{aligned}
&\int R_t(x)^2dx\\
 = &\sin^2(Ht)\int A^2(x)dx+\cos^2(Ht)\int B^2(x)dx+2\sin(Ht)\cos(Ht)\int A(x)B(x) dx\\
 =&1.
\end{aligned}
$$
\end{proof}

\begin{proposition}
Suppose that $\rho_0,\rho_1>0$, $\rho_0\neq \rho_1$. Then, there exists a geodesic curve $\rho(t)$ with $\rho_t|_{t=0}=\rho_0$ and $\rho_t|_{t=1}=\rho_1$. 
\end{proposition}
\begin{proof}
We denote $R_0(x)=\sqrt{\rho_0(x)}$ and $R_1(x)=\sqrt{\rho_1(x)}$. We only need to solve $A(x)$ and $H>0$ such that
$$
R_1(x) = A(x)\sin(H)+R_0(x)\cos(H),
$$
We shall have
$$
\int R_1(x)R_0(x)dx  = \cos (H),
$$
which indicates $H=\cos^{-1} \pp{\int R_1(x)R_0(x)dx }\in (0,\pi/2]$. Hence, we have
$$
A(x) = \frac{R_1(x)-R_0(x)\cos(H)}{\sin (H)}.
$$
We can examine that
$$
\int A^2(x)dx = \frac{1-2\cos^2(H)+\cos^2(H)}{\sin^2(H)}=1.
$$
On the other hand, we shall examine that
$$
R_t(x)>0,\quad t\in [0,1].
$$
Indeed, 
$$
\begin{aligned}
R_t(x) =& A(x)\sin(Ht)+R_0(x)\cos(Ht)\\
=&\frac{\sin(Ht)(R_1(x)-R_0(x)\cos(H))+R_0(x)\cos(Ht)\sin(H)}{\sin (H)}\\
=&\frac{1}{\sin H} (\sin(Ht)R_1(x)+(\cos(Ht)\sin(H)-\sin(Ht)\cos(H))R_0(x))\\
=&\frac{1}{\sin H} (\sin(Ht)R_1(x)+\sin(H(1-t)) R_0(x))>0.
\end{aligned}
$$
Hence, $\rho_t(x)=R_t^2(x)$ is the geodesic curve. 
\end{proof}
A direct derivation is the Fisher-Rao distance between $\rho_0$ and $\rho_1$. Namely, we can recover $\Phi_0$ by
$$
\Phi_0(x) = \frac{2HA(x)}{R_0(x)}.
$$
We note that $\mcH(\rho_t,\Phi_0)=4H^2$. Hence, we have
$$
\pp{\mcD^{FR}(\rho_0,\rho_1)}^2 = \int_0^1 4H^2 dt = 4H^2.
$$
\begin{remark}
We note that the manifold $(\mcP^+(\Omega), \mcG^{FR}(\rho))$ is homeomorphic to the manifold $(S^+(\Omega), \mcG^{E}(R))$, where $S^+(\Omega)=\{R\in \mcF(\Omega):R>0,\int R^2 dx =1\}$. Here $(S^+(\Omega), \mcG^{E}(R))$ is the submanifold to $\mbL^2(\Omega)$ equiped with the standard Euclidean metric. 
\end{remark}

\subsection{Convergence analysis}
We consider accelerated Fisher-Rao gradient flows
\begin{equation}\label{aig:fr}
\lbb{&\p_t\rho_t-(\Phi_t-\mbE_{\rho_t}[\Phi_t])\rho_t=0,\\
&\p_t\Phi_t+\alpha_t\Phi_t+\frac{1}{2}\Phi_t^2-\mbE_{\rho_t}[\Phi_t]\Phi_t+\frac{\delta E}{\delta \rho_t}=0.}
\end{equation}
In the sense of $R_t$, we have
\begin{equation}\label{aig:fr_rt}
\lbb{&\p_tR_t-\frac{1}{2}(\Phi_t-\mbE_{R_t^2}[\Phi_t])R_t=0,\\
&\p_t\Phi_t+\alpha_t\Phi_t+\frac{1}{2}\Phi_t^2-\mbE_{R^2_t}[\Phi_t]\Phi_t+\frac{\delta E}{\delta \rho_t}=0.}
\end{equation}
Then, we prove the convergence results for $\beta$-strongly convex $E(\rho)$. Here we take $\alpha_t=2\sqrt{\beta}$. Consider the Lyapunov function
$$
\begin{aligned}
\mcE(t) = &\frac{e^{\sqrt{\beta }t}}{2}\int |\Phi_t-\mbE_{R_t^2}[\Phi_t]-\sqrt{\beta} T_t|^2\rho_tdx\\
&+e^{\sqrt{\beta } t} (E(\rho_t)-E(\rho^*)).
\end{aligned}
$$
Here we define
$$
T_t(x) = \frac{2 H_t}{\sin (H_t)}\frac{R^*(x)-R_t(x)\cos(H_t)}{R_t(x)}, \quad H_t=\cos^{-1} \pp{\int R_t(x)R^*(x)dx }. 
$$
We can rewrite the Lyapunov function as
$$
\begin{aligned}
\mcE(t) = &\frac{e^{\sqrt{\beta }t}}{2}\int (\Phi_t-\mbE_{R_t^2}[\Phi_t])^2\rho_tdx-\sqrt{\beta}e^{\sqrt{\beta }t}\int (\Phi_t-\mbE_{R_t^2}[\Phi_t])T_t\rho_tdx\\
&+\frac{\beta e^{\sqrt{\beta }t}}{2}\int  T_t^2\rho_tdx+e^{\sqrt{\beta } t} (E(\rho_t)-E(\rho^*)).
\end{aligned}
$$
\begin{remark}
Here it may be problematic if $R_t(x)=0$ for some $x$. But in total, 
$$
\int T_t^2\rho_t dx = \int (R_tT_t)^2 dx. 
$$
is well-defined. 
\end{remark}
From the definition of convexity in probability space, we derive the following proposition.
\begin{proposition}\label{prop:cvx_fr}
The $\beta$-convexity of $E(\rho)$ indicates that
$$
E(\rho^*)\geq E(\rho_t)+\int \pp{\frac{\delta E}{\delta \rho_t}-\mbE_{\rho_t}\bb{\frac{\delta E}{\delta \rho_t}}}T_t\rho_tdx+\frac{\beta}{2}\int  T_t^2\rho_tdx. 
$$
\end{proposition}
For simplicity, we define
$$
\mcF_t[\Psi] = \Psi-\mbE_{R_t^2}[\Psi]. 
$$
We have
$$
\begin{aligned}
\p_t(\mcF_t[\Psi]) = &\p_t\Psi - \mbE_{R_t^2}[\p_t\Psi]-\int R_t^2 \mcF_t[\Phi_t] \Psi dx=\mcF_t[\p_t\Psi] -\int R_t^2 \mcF_t[\Phi_t] \Psi dx.
\end{aligned}
$$
Before we perform computations, we establish several identities.
$$
\int \mcF_t[\Psi] R_t^2dx = 0.
$$
$$
\int \mcF_t[\Psi_1]\mcF_t[\Psi_2] R_t^2dx=\int \mcF_t[\Psi_1]\Psi_2 R_t^2dx=\int \mcF_t[\Psi_2]\Psi_1 R_t^2dx.
$$


\begin{lemma}\label{lem:obsv}
We have the following observations:
\begin{equation}
\int (\p_t T_t)\mcF_t[\Phi_t] R_t^2 dx+\frac{1}{2} \int T_t (\mcF_t[\Phi_t])^2R_t^2dx
\geq - \int (\mcF_t[\Phi_t] )^2 R_t^2dx,
\end{equation}
\begin{equation}
\int (\p_t T_t )T_t R_t^2 dx=-\int T_t \Phi_t R_t^2 dx-\frac{1}{2}\int T_t^2 \mcF_t[\Phi_t] R_t^2 dx.
\end{equation}
\end{lemma}
\begin{proof}
We note that 
$$\int T_t^2R_t^2 dx = 4H_t^2,$$
and
$$
\int (\mcF_t[R^*R_t^{-1}])^2 R_t^2 dx=\frac{\sin^2(H_t)}{4H_t^2}\int T_t^2R_t^2 dx=\sin(H_t^2).
$$
We compute the derivatives as follows:
$$
\begin{aligned}
\p_t H_t=&-\frac{1}{\sin H_t}\p_t{\int R_tR^*dx}=-\frac{1}{2\sin H_t}\int R_tR^*\mcF_t[\Phi_t] dx.
\end{aligned}
$$
$$
\begin{aligned}
\p_t T_t = &-\frac{1}{\sin H_t}\pp{\int R_tR^*\mcF_t[\Phi_t] dx}\frac{\sin(H_t)-H_t\cos(H_t) }{\sin^2(H_t)}(R^*R_t^{-1}-\cos(H_t))\\
&+\frac{2H_t}{\sin(H_t)}\pp{-\frac{1}{2}R^*R_t^{-1}\mcF_t[\Phi_t]-\frac{1}{2}\int R_tR^*\mcF_t[\Phi_t] dx}\\
= &-\frac{1}{\sin H_t}\pp{\int R^* R_t\mcF_t[\Phi_t] dx}\frac{\sin(H_t)-H_t\cos(H_t) }{\sin^2(H_t)}\mcF_t[R^* R_t^{-1}]\\
&-\frac{H_t}{\sin(H_t)}\pp{R^*R_t^{-1}\mcF_t[\Phi_t]+\int R_tR^*\mcF_t[\Phi_t] dx}.
\end{aligned}
$$

For the first inequality, we have
$$
\begin{aligned}
&\int( \p_t T_t)\mcF_t[\Phi_t] R_t^2 dx\\
=&-\frac{1}{\sin (H_t)}\pp{\int R^* R_t\mcF_t[\Phi_t] dx}\frac{\sin(H_t)-H_t\cos(H_t) }{\sin^2(H_t)}\int \mcF_t[R^* R_t^{-1}] \mcF_t[\Phi_t] R_t^2 dx\\
&-\frac{H_t}{\sin(H_t)}\int (R^*R_t^{-1}\mcF_t[\Phi_t]) \mcF_t[\Phi_t] R_t^2 dx\\
=&-\frac{\sin(H_t)-H_t\cos(H_t)}{\sin^3(H_t)}\pp{\int \mcF_t[R^* R_t^{-1}]\mcF_t[\Phi_t]  R_t^2dx}^2\\
&-\frac{1}{2}\frac{2H_t}{\sin(H_t)}\int R^*R_t^{-1}\mcF_t[\Phi_t] \mcF_t[\Phi_t] R_t^2 dx\\
\geq&-\frac{\sin(H_t)-H_t\cos(H_t)}{\sin^3(H_t)}\pp{\int (\mcF_t[\Phi_t])^2 R_t^2 dx}\pp{\int (\mcF_t[R^*R_t^{-1}])^2 R_t^2 dx}\\
&-\frac{1}{2}\frac{2H_t}{\sin(H_t)}\int (R^*R_t^{-1}-\cos(H_t))(\mcF_t[\Phi_t])^2R_t^2dx\\
&-\frac{1}{2}\frac{2H_t}{\sin(H_t)}\int \cos(H_t)(\mcF_t[\Phi_t])^2R_t^2dx\\
=& -\frac{\sin(H_t)-H_t\cos(H_t)}{\sin (H_t)}\pp{\int (\mcF_t[\Phi_t])^2 R_t^2 dx}-\frac{1}{2} \int T_t (\mcF_t[\Phi_t])^2R_t^2dx\\
&-\frac{H_t\cos(H_t)}{\sin(H_t)}\int R_t^2(\mcF_t[\Phi_t])^2 dx\\
=&-\frac{1}{2} \int T_t (\mcF_t[\Phi_t])^2R_t^2dx-\int (\mcF_t[\Phi_t])^2 R_t^2 dx.
\end{aligned}
$$
The inequality is based on Cauchy inequality. 
For the second inequality, we have
$$
\begin{aligned}
&\int( \p_t T_t) T_t R_t^2 dx\\
=&-\frac{1}{\sin H_t}\pp{\int R^* R_t\mcF_t[\Phi_t] dx}\frac{\sin(H_t)-H_t\cos(H_t) }{\sin^2(H_t)}\int T_t\mcF_t[R^* R_t^{-1}] R_t^2 dx\\
&-\frac{H_t}{\sin(H_t)}\int T_t R^*R_t^{-1} \mcF_t[\Phi_t] R_t^2 dx\\
=&-\frac{1}{\sin H_t}\pp{\int R^* R_t\mcF_t[\Phi_t] dx}\frac{\sin(H_t)-H_t\cos(H_t) }{2\sin(H_t) H_t}\int T_t^2R_t^2 dx\\
&-\frac{1}{2}\frac{2H_t}{\sin(H_t)}\int(R^* R_t -\cos(H_t)) T_t \mcF_t[\Phi_t] R_t^2 dx-\frac{1}{2}\frac{2H_t\cos(H_t)}{\sin(H_t)}\int T_t \mcF_t[\Phi_t] R_t^2 dx\\
=&-\frac{1}{2 H_t}\pp{\int T_t \Phi_t R_t^2 dx}\frac{\sin(H_t)-H_t\cos(H_t) }{2\sin(H_t) H_t}\int T_t^2R_t^2 dx\\
&-\frac{1}{2}\int T_t^2 \mcF_t[\Phi_t] R_t^2 dx-\frac{H_t\cos(H_t)}{\sin(H_t)}\int T_t \Phi_t R_t^2 dx\\
=&-\pp{\frac{\sin(H_t)-H_t\cos(H_t) }{\sin (H_t)}+\frac{H_t\cos(H_t)}{\sin(H_t)}}\int T_t \Phi_t R_t^2 dx-\frac{1}{2}\int T_t^2 \mcF_t[\Phi_t] R_t^2 dx\\
=&-\int T_t \Phi_t R_t^2 dx-\frac{1}{2}\int T_t^2 \mcF_t[\Phi_t] R_t^2 dx.
\end{aligned}
$$
This completes the proof. 
\end{proof}
Hence, we can compute that
$$
\begin{aligned}
e^{-\sqrt{\beta }t}\p_t\mcE(t) = &\frac{\sqrt{\beta}}{2}\int (\mcF_t[\Phi_t])^2R_t^2dx+\int \mcF_t[\Phi_t]\pp{\mcF_t[\p_t\Phi_t]-\int R_t^2 \mcF_t[\Phi_t] \Phi_t dx} R_t^2 dx\\
&+\frac{1}{2}\int (\mcF_t[\Phi_t])^2 \mcF_t[\Phi_t] R_t^2 dx-\beta\int (\Phi_t-\mbE_{R_t^2}[\Phi_t])T_t\rho_tdx\\
&-\sqrt{\beta}\int \pp{\mcF_t[\p_t\Phi_t]-\int R_t^2 \mcF_t[\Phi_t] \Phi_t dx}  T_tR_t^2 dx \\
&-\sqrt{\beta}\int \p_tT_t \mcF_t[\Phi_t] R_t^2 dx-\sqrt{\beta}\int (\mcF[\Phi_t])^2T_t R_t^2 dx \\
&+\frac{\beta\sqrt{\beta}}{2}\int  T_t^2R_t^2dx+\beta \int \p_t T_t T_t R_t^2dx+\frac{\beta}{2} \int T_t^2 \mcF_t[\Phi_t]R_t^2 dx\\
&+ \sqrt{\beta} (E(\rho_t)-E(\rho^*))+\int \mcF_t[\Phi_t]\mcF_t\bb{\frac{\delta E}{\delta \rho_t}} R_t^2 dx. 
\end{aligned}
$$
From Proposition \ref{prop:cvx_fr}, we have
$$
\begin{aligned}
 &\sqrt{\beta} (E(\rho_t)-E(\rho^*))+\frac{\beta\sqrt{\beta}}{2}\int  T_t^2R_t^2dx\leq-\sqrt{\beta} \int\mcF_t\bb{\frac{\delta E}{\delta \rho_t}}T_t\rho_tdx. 
\end{aligned}
$$
We first compute terms with coefficient $\beta^0$. We have
$$
\begin{aligned}
&\int \mcF_t[\Phi_t]\pp{\mcF_t[\p_t\Phi_t]-\int R_t^2 \mcF_t[\Phi_t] \Phi_t dx} R_t^2 dx\\
&+\frac{1}{2}\int (\mcF_t[\Phi_t])^2 \mcF_t[\Phi_t] R_t^2 dx+\int \mcF_t[\Phi_t]\mcF_t\bb{\frac{\delta E}{\delta \rho_t}}R_t^2 dx\\
=&\int \mcF_t[\Phi_t] \p_t\Phi_t R_t^2dx+\frac{1}{2}\int (\mcF_t[\Phi_t])^2 \mcF_t[\Phi_t] R_t^2 dx+\int \mcF_t[\Phi_t]\frac{\delta E}{\delta \rho_t} R_t^2 dx\\
=&\int \mcF_t[\Phi_t]\pp{-\sqrt{\beta}\Phi_t-\frac{1}{2}\Phi_t^2+\mbE_{R_t^2}[\Phi_t]\Phi_t+\frac{1}{2}\mcF_t[\Phi_t]^2} R_t^2 dx\\
=&\int \mcF_t[\Phi_t]\pp{-\sqrt{\beta}\Phi_t+\frac{1}{2}(\mbE_{R_t^2}[\Phi_t])^2} R_t^2 dx\\
=&-2\sqrt{\beta} \int \mcF_t[\Phi_t] \Phi_t R_t^2 dx.
\end{aligned}
$$
We then proceed to compute terms with coefficient $\beta^{1/2}$. 
$$
\begin{aligned}
&\frac{\sqrt{\beta}}{2}\int (\mcF_t[\Phi_t])^2R_t^2dx-\sqrt{\beta}\int \pp{\mcF_t[\p_t\Phi_t]-\int R_t^2 \mcF_t[\Phi_t] \Phi_t dx}  T_tR_t^2 dx\\
&-2\sqrt{\beta} \int \mcF_t[\Phi_t] \Phi_t R_t^2 dx-\sqrt{\beta}\int \p_tT_t \mcF_t[\Phi_t] R_t^2 dx-\sqrt{\beta}\int (\mcF[\Phi_t])^2T_t R_t^2 dx\\
&-\sqrt{\beta} \int \mcF_t\bb{\frac{\delta E}{\delta \rho_t}}T_t\rho_tdx\\
=&-\frac{3\sqrt{\beta}}{2}\int (\mcF_t[\Phi_t])^2R_t^2dx-\sqrt{\beta}\int \p_t\Phi_t T_tR_t^2 dx-\sqrt{\beta}\int \p_tT_t \mcF_t[\Phi_t] R_t^2 dx\\
&-\sqrt{\beta}\int (\mcF[\Phi_t])^2T_t R_t^2 dx-\sqrt{\beta} \int \frac{\delta E}{\delta \rho_t}T_tR_t^2dx\\
=&-\sqrt{\beta} \int T_tR_t^2\pp{\p_t\Phi_t+\frac{\delta E}{\delta \rho_t}+\frac{1}{2}(\mcF[\Phi_t])^2}-\frac{3\sqrt{\beta}}{2}\int (\mcF_t[\Phi_t])^2R_t^2dx\\
&-\sqrt{\beta}\int \p_tT_t \mcF_t[\Phi_t] R_t^2 dx-\frac{\sqrt{\beta}}{2}\int (\mcF[\Phi_t])^2 T_t R_t^2 dx\\
\leq &2\beta\int T_t\Phi_t R_t^2-\frac{\sqrt{\beta}}{2}\int (\mcF_t[\Phi_t])^2R_t^2dx.
\end{aligned}
$$
The last inequality is based on Lemma \ref{lem:obsv}.  Finally, we compute terms with coefficient $\beta$:

$$
\begin{aligned}
&2\beta\int T_t\Phi_t R_t^2dx-\beta \int \Phi_t T_tR_t^2 dx+\beta \int \p_t T_t T_t R_t^2 dx+ \frac{\beta}{2} \int T_t^2 \mcF_t[\Phi_t]R_t^2 dx=0.\\
\end{aligned}
$$

In summary, we have
$$
e^{-\sqrt{\beta }t}\p_t\mcE(t) \leq-\frac{\sqrt{\beta}}{2}\int (\mcF_t[\Phi_t])^2R_t^2dx\leq 0.
$$

For convex $E(\rho)$, we let $\alpha_t=3/t$. Consider
$$
\mcE(t)=\frac{1}{2} \int \pp{-T_t+\frac{t}{2}\Phi_t}^2R_t^2dx+\frac{t^2}{4}(E(R_t^2)-E(\rho^*)).
$$
We can compute that
$$
\begin{aligned}
\dot \mcE(t) =&\int  (\p_t T_t) T_t R^2_t dx+\frac{1}{2}\int T_t^2\mcF[\Phi_t] R_t^2dx-\frac{1}{2}\int T_t\Phi_tR_t^2dx\\
&-\frac{t}{2}\int T_t\pp{\p_t\Phi_t}R_t^2dx- \frac{t}{2}\int (\p_t T_t)\Phi_t R_t^2 dx\\
&-\frac{t}{2}\int T_t(\mcF_t[\Phi_t])^2 R_t^2 dx+\frac{t}{4} \int (\mcF_t[\Phi_t])^2R_t^2dx\\
&+\frac{t^2}{4} \int (\p_t \mcF_t[\Phi_t]) \mcF_t[\Phi_t] R_t^2dx+\frac{t^2}{8} \int (\mcF_t[\Phi_t])^3 R_t^2 dx\\
&-\frac{t^2}{4}\int \mcF_t\bb{\frac{\delta E}{\delta \rho_t}}\mcF_t[\Phi_t]R_t^2dx+\frac{t}{2}(E(R^2_t)-E(\rho^*)).\\
\end{aligned}
$$
Because $E(\rho)$ is convex, we have
$$
E(R^2_t)-E(\rho^*)\leq -\int \mcF_t\bb{\frac{\delta E}{\delta \rho_t}}T_tR_t^2dx.
$$
From Lemma \ref{lem:obsv}, we have
$$
\begin{aligned}
\dot \mcE(t) \leq &-\frac{3}{2}\int T_t\Phi_tR_t^2dx-\frac{t}{2}\int T_t\pp{\p_t\Phi_t}R_t^2dx \\
&-\frac{t}{4}\int T_t(\mcF_t[\Phi_t])^2 R_t^2 dx+\frac{3t}{4} \int (\mcF_t[\Phi_t])^2R_t^2dx\\
&+\frac{t^2}{4} \int (\p_t \Phi_t) \mcF_t[\Phi_t] R_t^2dx+\frac{t^2}{8} \int (\mcF_t[\Phi_t])^3 R_t^2 dx\\
&-\frac{t^2}{4}\int \frac{\delta E}{\delta \rho_t}\mcF_t[\Phi_t]R_t^2dx-\frac{t}{2}\int \mcF_t\bb{\frac{\delta E}{\delta \rho_t}}T_tR_t^2dx\\
=&-\frac{3}{2}\int T_t\Phi_tR_t^2dx-\frac{t}{2}\int  T_tR_t^2\pp{\p_t\Phi_t+\frac{1}{2}(\mcF_t[\Phi_t])^2+\frac{\delta E}{\delta \rho_t}}\\
&+\frac{3t}{4} \int (\mcF_t[\Phi_t])^2R_t^2dx+\frac{t^2}{4}\int \mcF_t[\Phi_t] R_t^2\pp{\p_t\Phi_t+\frac{1}{2}(\mcF_t[\Phi_t])^2+\frac{\delta E}{\delta \rho_t}}dx\\
=&0.
\end{aligned}
$$
The last equality utilize the fact that $\p_t\Phi_t+\frac{1}{2}(\mcF_t[\Phi_t])^2+\frac{\delta E}{\delta \rho_t}=-\frac{3}{t}\Phi_t$. 

\section{Discrete-time algorithm of AIG flows}
In this section, we introduce the discrete-time algorithm for Kalman-Wasserstein AIG flows and Stein AIG flows. Here $E(\rho)$ is the KL divergence from $\rho$ to $\rho^*\propto \exp(-f)$.
\subsection{Discrete-time algorithm of KW-AIG flows}
For KL divergence, the particle formulation \eqref{equ:kw_aig_p} of KW-AIG flows writes
\begin{equation}\label{equ:kw_aig_p_kl}
\left\{
\begin{aligned}
&dX_t = C^\lambda(\rho_t)V_t dt,\\ 
&dV_t =-\alpha_tV_t dt-\mathbb{E}[V_tV_t^T](X_t-\mbE[X_t])dt-(f(X_t)+\nabla \log \rho_t(X_t)) dt.
\end{aligned}\right.
\end{equation}
Consider a particle system $\{X_0^i\}_{i=1}^N. $In $k$-th iteration, the update rule follows: for $i=1,2,\dots N$,
\begin{equation}\label{equ:vx_1}
\left\{
\begin{aligned}
&X_{k+1}^i =X_{k}^i +\sqrt{\tau_k}C^\lambda_k V_k,\\ 
&V_{k+1} =\alpha_k V_k-\sqrt{\tau_k}\left[\sum_{i=1}^N(V_k^i)(V_k^i)^T\right](X_{k}^i-m_k)-\sqrt{\tau_k}(f(X_k^i)+\xi_k(X_k^i)).
\end{aligned}\right.
\end{equation}
Here $\xi_k$ is an approximation of $\nabla\log \rho_k$ and we denote
$$
m_k = \frac{1}{N}\sum_{i=1}^N X_k^i,\quad C^\lambda_k = \frac{1}{N-1}\sum_{i=1}^N (X_k^i-m_k)(X_k^i-m_k)^T+\lambda I.
$$
The choice of $\alpha_k$ is similar to the discrete-time algorithm of W-AIG flows. If $E(\rho)$ is $\beta$-strongly convex, then $\alpha_k = \frac{1-\sqrt{\beta\tau_k}}{1+\sqrt{\beta\tau_k}}$; if $E(\rho)$ is convex or $\beta$ is unknown, then $\alpha_k = \frac{k-1}{k+2}$. 

About the adaptive restart technique, the restarting criterion follows
\begin{equation}\label{phik_1}
    \varphi_k = - \sum_{i=1}^N\la C^\lambda_kV_{k+1}^i, \nabla f(X_k^i)+\xi_k(X_k^i)\ra.
\end{equation}
The overall algorithm is summarized as follows.
\begin{algorithm}[!htp]
\caption{Discrete-time particle implementation of KW-AIG flow}
\label{alg:kw_aig}
\begin{algorithmic}[1]
\REQUIRE initial positions $\{X_0^i\}_{i=1}^N$, step size $\tau_k$, number of iteration $L$.
\STATE Set $k=0$, $V_0^i=0, i=1,\dots N$. Set the bandwidth $h_0$ by MED. 
\FOR{$l=1,2,\dots L$}
\STATE Compute $h_l$ based on BM method:  $h_l=\text{BM}(h_{l-1},\{X_k^i\}_{i=1}^N,\sqrt{\tau})$.
\STATE Calculate $\xi_k(X_k^i)$ as an approximation of $\nabla \log \rho_k(X_k^i)$. 
\STATE 
For $i=1,2,\dots N$, update $V_{k+1}^i$ and $X_{k+1}^i$ by \eqref{equ:vx_1}.
%
\STATE Compute $\varphi_k$ by \eqref{phik_1}. 
\STATE If $\varphi_k<0$, set $X_0^i=X_k^i$ and $V_0^i=0$ and $k=0$; otherwise set $k=k+1$. 
\ENDFOR
\end{algorithmic}
\end{algorithm}

\subsection{Discrete-time algorithm for S-AIG flows}
For KL divergence, the particle formulation of S-AIG flows writes
\begin{equation}\label{equ:s_aig_p_kl}
\left\{
\begin{aligned}
&\frac{d}{dt} X_t = \int k(X_t,y) \nabla\Phi_t(y)\rho_t(y)dy,\\
&\frac{d}{dt}  V_t =-\alpha_t V_t-\int V_t^T\nabla \Phi_t (y) \nabla_x k(X_t,y) \rho_t(y)dy-\nabla f(X_t)-\nabla\log \rho_t.
\end{aligned}\right.
\end{equation}
Consider a particle system $\{X_0^i\}_{i=1}^N$. In $k$-th iteration, the update rule follows: for $i=1,2,\dots N$,
\begin{equation}\label{equ:kx1}
\left\{
\begin{aligned}
& X_{k+1}^i = X_k^i+\frac{\sqrt{\tau_k}}{N}\sum_{j=1}^N k(X_k^i,X_k^j) V_{k+1}^j,\\
&V_{k+1}^i =\alpha_k V_k^i-\frac{\sqrt{\tau_k}}{N}\sum_{j=1}^N (V_k^i)^TV_k^j\nabla_x k(X_k^i,X_k^j) -\sqrt{\tau_k}(\nabla f(X_k^i)+\xi_k(X_k^i)).
\end{aligned}\right.
\end{equation}
Here $\xi_k$ is an approximation of $\nabla\log \rho_k$. The choice of $\alpha_k$ is similar, depending on the convexity of $E(\rho)$ w.r.t. Stein metric.

About the adaptive restart technique, the restarting criterion follows
\begin{equation}\label{phik_2}
    \varphi_k = - \sum_{i=1}^N\sum_{j=1}^Nk(X_k^j,X_k^i)\la V_{k+1}^j, \nabla f(X_k^i)+\xi_k(X_k^i)\ra.
\end{equation}

The overall algorithm is summarized as follows.
\begin{algorithm}[!htp]
\caption{Discrete-time particle implementation of S-AIG flow}
\label{alg:s_aig}
\begin{algorithmic}[1]
\REQUIRE initial positions $\{X_0^i\}_{i=1}^N$, step size $\tau_k$, number of iteration $L$.
\STATE Set $k=0$, $V_0^i=0, i=1,\dots N$. Set the bandwidth $h_0$ by MED. 
\FOR{$l=1,2,\dots L$}
\STATE Compute $h_l$ based on BM method:  $h_l=\text{BM}(h_{l-1},\{X_k^i\}_{i=1}^N,\sqrt{\tau})$.
\STATE Calculate $\xi_k(X_k^i)$ as an approximation of $\nabla \log \rho_k(X_k^i)$. 
\STATE 
For $i=1,2,\dots N$, update $V_{k+1}^i$ and $X_{k+1}^i$ by \eqref{equ:kx1}.
%
\STATE Compute $\varphi_k$ by \eqref{phik_2}. 
\STATE If $\varphi_k<0$, set $X_0^i=X_k^i$ and $V_0^i=0$ and $k=0$; otherwise set $k=k+1$. 
\ENDFOR
\end{algorithmic}
\end{algorithm}

\section{Implementation details in the numerical experiments}
\label{app:num}
In this section, we provide extra numerical experiments and elaborate on the implementation details in the numerical experiments.

\subsection{Details in Subsection 6.1}
We follow the same setting as \cite{SVGD}, which is also adopted by \cite{afomo,uaapb}. The dataset is split into $80\%$ for training and $20\%$ for testing. We use the stochastic gradient and the mini-batch size is taken as $100$. For MCMC, the number of particles is $N=1000$; for other methods, the number of particles is $N=100$. The BM method is not applied to SVGD in selecting the bandwidth. 

The initial step sizes for the compared methods are given in Table \ref{tab:step}, which are selected by grid search over $1\times 10^i$ with $i=-3,-4,\dots ,-9$. (For SVGD, we use the initial step size in \citep{SVGD}.) The step size of SVGD is adjusted by Adagrad, which is same as \citep{SVGD}. For WNAG and WRes, the step size is give by $\tau_l = \tau_0/l^{0.9}$ for $l\geq 1$. The parameters for WNAG and Wnes are identical to \citep{afomo} and \citep{uaapb}. For other methods, the step size is multiplied by $0.9$ every $100$ iterations. For methods under Kalman-Wasserstein metric, we require a smaller step size (around 1e-8) to make the algorithm converge. For all discrete-time algorithms of AIGs, we apply the restart technique.  
\begin{table}[!htbp]
    \centering
    \begin{tabular}{|c|c|c|c|c|c|}
    \hline
         Method&MCMC&WNAG  &   WNes  &   W-GF   & W-AIG  \\\hline
         Step size $\tau_0$&1e-5&1e-6&1e-5&1e-5&1e-6\\\hline
         Method & KW-GF  & KW-AIG    & SVGD &   S-AIG&\\\hline
         Step size $\tau_0$&1e-7&1e-8&0.05&1e-5&\\\hline
    \end{tabular}
    \caption{Initial step sizes for compared algorithms in Bayesian logistic regression.}
    \label{tab:step}
\end{table}
We record the cpu-time for each method in Table \ref{tab:time}. The computational cost of the BM method is much higher than the MED method because we need to evaluate the MMD of two particle systems several times in optimizing the subproblem. We may update the bandwidth using the BM method every 10 iterations to deal with the high computation cost of the BM method. On the other hand, using the MED method for bandwidth, the computational cost of S-AIG is much higher than other methods. This results from the multiple times of computation of particle interacting in updating $X_k^i$ and $V_k^i$. 
\begin{table}[!htbp]
    \centering
    \begin{tabular}{|c|c|c|c|c|c|}
    \hline
         Method&MCMC&WNAG  &   WNes  &   W-GF   & W-AIG  \\\hline
         BM&26.181  &164.980  &165.407 & 167.308  &170.116  \\\hline
         MED& 27.200   & 7.585  &  7.688  &  7.501  &  7.719 \\\hline
         Method& KW-GF  & KW-AIG    & SVGD &   S-AIG &\\\hline
         BM&168.711  &173.670    &7.193  &200.016& \\\hline
         MED& 8.847  & 10.065   & 7.755 &  21.303& \\\hline
    \end{tabular}
    \caption{Averaged cpu\ time(s) cost for algorithms in Bayesian logistic regression.}
    \label{tab:time}
\end{table}

\subsection{Details in Subsection 6.2}
We follow the setting of Bayesian neural network as \citep{svgd_mat}. The kernel bandwidth is adjusted by the MED method. We list the number of epochs and the batch size for each datasets in Table \ref{tab:num_epoch}. For each dataset, we use $90\%$ of samples as the training set and $10\%$ of samples as the test set. The step size of SVGD is adjusted by Adagrad. For W-GF and W-AIG , the step size is multiplied by 0.64 every $1/10$ of total epochs. We select the initial step size by grid search over $\{1,2,5\}\times 10^i$ with $i=-3,-4,\dots ,-7$ to ensure the best performance of compared methods. We list the initial step sizes for each dataset in Table \ref{tab:bnn_step}. For W-AIG, we apply the adaptive restart.

\begin{table}[!htbp]
    \centering
    \begin{tabular}{|c|c|c|c|c|}
    \hline
         Dataset&Boston&Combined&Concrete\\\hline
         Epochs&50&500&500\\\hline
         Batch size&100&100&100\\\hline
         Dataset&Kin8nm&Wine&Year\\\hline
         Epochs&200&20&10\\\hline
         Batch size&100&100&1000\\\hline
    \end{tabular}
    \caption{Number of epochs and batch size in Bayesian neural network.}
    \label{tab:num_epoch}
\end{table}

\begin{table}[!htbp]
    \centering
    \begin{tabular}{|c|c|c|c|c|}
    \hline
         Dataset&Boston&Combined&Concrete\\\hline
         AIG&2e-5&2e-4&2e-5\\\hline
         WGF&1e-4&1e-3&2e-5\\\hline
         SVGD&5e-4&5e-3&5e-4\\\hline
         Dataset&Kin8nm&Wine&Year\\\hline
         AIG&2e-5&5e-6&2e-7\\\hline
         WGF&1e-4&1e-4&2e-6\\\hline
         SVGD&5e-3&2e-3&5e-3\\\hline
    \end{tabular}
    \caption{Initial step sizes for compared methods in Bayesian neural network.}
    \label{tab:bnn_step}
\end{table}

\end{document}